\documentclass[11pt]{article}

\title{How Big Should Your Data Really Be?\\ 
Data-Driven Newsvendor: Learning One Sample at a Time
}
\author{Omar Besbes, Omar Mouchtaki \\ Columbia University, Graduate School of Business\\
\texttt{obesbes@columbia.edu}, \texttt{om2316@gsb.columbia.edu} }
\usepackage[utf8]{inputenc}
\usepackage[T1]{fontenc}
\usepackage[english]{babel}
\usepackage[]{algorithm2e}
\usepackage{mwe}
\usepackage{amsmath,mathtools}

\usepackage[normalem]{ulem}

\usepackage{amsthm,bm}
\usepackage{lineno}

\usepackage{tikz,pgfplots}
\usepackage{subfigure}
\usetikzlibrary{intersections}
\usepackage{multirow}

\usepackage[shortlabels]{enumitem}

\usepackage[colorlinks=true, linkcolor=blue, citecolor=blue, hypertexnames=false]{hyperref}  
\usepackage{nameref}
\usepackage{cleveref}

\usepackage{bbm}
\newtheorem{lemma}{Lemma}
\newtheorem{proposition}{Proposition}
\newtheorem{corollary}{Corollary}
\newtheorem{theorem}{Theorem}
\newtheorem{assumption}{Assumption}

\newtheorem{definition}{Definition}

\newcommand {\beq}{\begin{equation}}
\newcommand {\eeq}{\end{equation}}
\newcommand {\beqn}{\begin{equation*}}
\newcommand {\eeqn}{\end{equation*}}
\newcommand {\bear}{\begin{eqnarray}}
\newcommand {\eear}{\end{eqnarray}}
\newcommand {\bearn}{\begin{eqnarray*}}
\newcommand {\eearn}{\end{eqnarray*}}

\DeclareMathOperator*{\argmax}{arg\,max}

\newcommand{\ratio}{{\cal R} }
\newcommand{\opt}{\mbox{opt}}
\newcommand{\numS}{n}
\newcommand{\OS}[1]{ \pi^{OS_{#1}} }
\newcommand{\setOS}{\Pi^{OS}_{\numS}}
\newcommand{\SAA}{ \pi^{\text{SAA}}}
\newcommand{\simplex}{\Delta_{\numS}}
\newcommand{\mix}{ \pi^{\bm{\lambda}}}
\newcommand{\optpol}{ \pi^{k,\gamma}}
\newcommand{\cu}{b}
\newcommand{\co}{h}

\newcommand{\DS}{\pi^{\Sigma_\mathbf{e}}} 
\newcommand{\DSf}{\pi^{\Sigma_\mathbf{f}}} 
\newcommand{\derpol}{ \pi^{\text{cvx}(k,\gamma)}}

\usepackage[right=1.0in, top=1in, bottom=1.0in, left=1.0in]{geometry}

\usepackage{lmodern}

\usepackage{graphicx}
\usepackage{lscape}
\usepackage[final]{pdfpages}
\usepackage{amsthm}
\usepackage[authoryear,round]{natbib}
\usepackage{amsfonts}
\usepackage{mathtools}
\usepackage{bbm}
\usepackage{dsfont}

\DeclarePairedDelimiter\ceil{\lceil}{\rceil}

\usepackage{setspace}

\begin{document}

\date{first version: March 15, 2021;  last revised: July 22, 2022}

\maketitle
\vspace{-10mm}
\begin{abstract}
We study the classical newsvendor problem in which the decision-maker must trade-off underage and overage costs. In contrast to the typical setting, we assume that the decision-maker does not know the underlying distribution driving uncertainty but has  only access to historical data. In turn, the key questions are \textit{how to map existing data to a decision} and what type of performance to expect \textit{as a function of the data size}. We analyze the classical setting with access to past samples drawn from the   distribution (e.g., past demand), focusing not only on asymptotic performance but also on what we call the \textit{transient regime of learning}, i.e., performance for arbitrary data sizes.  We evaluate the performance of any algorithm through its worst-case relative expected regret, compared to an oracle with knowledge of the distribution. We provide the first finite sample \textit{exact} analysis of the classical  Sample Average Approximation (SAA)  algorithm for this class of problems across \textit{all} data sizes. This allows to uncover novel fundamental insights on the value of data: it  reveals that \textit{tens of samples} are  sufficient to perform very efficiently but also that more data can lead to worse out-of-sample performance for SAA.  We then focus on the general class of mappings from data to decisions \textit{without} any restriction on the set of policies and derive an optimal algorithm (in the minimax sense) as well as characterize its associated performance. This leads to significant improvements for limited data sizes, and allows to exactly quantify the value of historical information.

\medskip

\end{abstract}
\section{Introduction}
The newsvendor problem is a prototypical model of decision-making under uncertainty that captures the trade-offs emerging in capacity or inventory decisions in the face of uncertainty in future outcomes. For example, when setting inventory decisions, a decision-maker typically faces uncertainty with regard to the demand that will materialize. For a given decision, if the demand turns out to be lower, the decision-maker would incur overage costs and if the demand realization is higher than the inventory, then some underage costs would be incurred for unsatisfied demand. Such trade-offs for inventory decisions  represent some of the most common operational problems faced by retailers. Newsvendor type trade-offs  also emerge in a variety of other applications, e.g., in revenue management for capacity setting or overbooking, or in electricity markets for setting capacity levels. 

The key to solving the trade-offs above and optimizing decisions is a statistical characterization of the uncertainty the decision-maker faces, typically captured by a distribution. In the inventory example above, this would correspond to the distribution of demand in the period between replenishments. 
In practice, the distribution is typically unknown and the only way to solve the trade-offs above,  is through the data that has been collected. The main questions this paper focuses on are the following: 
How should a newsvendor decision-maker optimize decisions given the historical data they have collected? What is the optimal performance they can garner as a function of the data size? 
We are interested in understanding the full spectrum of performance of policies across data sizes and refer to this approach as the transient regime lens for learning.

In more detail, we are interested in analyzing central policies in the literature, in optimizing data-driven policies, and in understanding whether one could characterize the performance achievable \textit{across data sizes}, small and large. The motivation to develop a framework for understanding performance for arbitrary data sizes has strong anchoring in both practice and theory. Despite the apparent wide availability of demand data, we posit that ``relevant'' data may be limited in practice due to the heterogeneity of market characteristics. For example, a year of weekly demand for a product only represents tens of samples, and assuming homogeneity of demand on a longer period of time may be too strong of a practical assumption. On the theory front, such a  framework would  provide a foundation for a bottom-up approach to data-driven decisions and would reveal the true robust value of data.

To analyze these questions, we focus on the typical data structure that comes in the form of past samples from the unknown distribution. This would correspond, for example, to demand observations in the inventory example. A data-driven policy is then a mapping from historical data to decisions. For any such policy, we evaluate its performance according to the worst-case (over all possible distributions) expected relative regret defined as the difference between the expected out-of-sample cost incurred by the data-driven policy and the expected optimal cost of an oracle that knows the distribution, normalized by the latter cost.

  It is important to highlight that evaluating the worst-case performance of a given data-driven algorithm against an arbitrary unknown distribution amounts to an intricate infinite dimensional optimization problem over the space of distributions. The  characterization of an optimal algorithm and its performance, or even the  exact performance of specific data-driven algorithms, have been mostly elusive to date. 
\vspace{-3mm}
\subsection{Main contributions}

\textit{Sample Average Approximation Analysis.} A popular and central approach to such data-driven problems is the Sample  Average Approximation (SAA) algorithm (also typically referred to as Empirical Optimization) which minimizes the expected cost according to the empirical distribution induced by the  observed samples. This method has been introduced to solve various stochastic optimization problems and enjoys asymptotic guarantees (\cite{SAA_First}).   In the context of newsvendor decisions,  state-of-the-art \textit{instance-independent} results on the number of samples required for SAA to achieve a particular confidence level were derived in \cite{levi2} and in \cite{cheung2019sampling}. While they capture the correct dependence on the confidence level as the number of samples grows large, state-of-the-art lower and upper bounds on the number of samples required to achieve a particular confidence level differ by orders of magnitude, leading to significant uncertainty on the  value of information, or on the quality of SAA. As such, despite its wide use and central role in the literature and in practice, to date, a significant gap exists in the understanding of the \textit{actual}  performance of this policy and the value it can capture from data. 

Our first main contribution is the characterization of the \textit{exact} performance of SAA for newsvendor problems for arbitrary data sizes (\Cref{thm:SAA}).  While determining the performance of the SAA algorithm is a priori an infinite dimensional optimization problem over the space of possible distributions, we actually establish that it is possible to reduce it to a one dimensional optimization problem and in turn derive a quasi closed-form solution that gives the exact worst-case relative regret of SAA. This worst-case is derived \textit{for any number of samples} and can be computed efficiently using a line-search.
Our method relies on the structure of the newsvendor problem and develops an analysis that leads to the  SAA performance as a corollary. We establish that for any policy that can be expressed as an order statistic or a randomization of order statistics of the empirical distribution,   one can transform the initial problem into a pointwise optimization problem of an appropriate functional, which ultimately leads to identifying the family of worst-case distributions for such policies (\Cref{thm:main_mixture}).  In particular, we establish, quite interestingly, that across all distributions, the worst case is a Bernoulli distribution whose mean depends on the number of samples observed. In turn, this yields the worst-case performance of the SAA policy as the solution of a one-dimensional search.  While our analysis reveals that the worst-case distribution is a Bernoulli, the induced worst-case performance we obtain can be applied to bound the relative regret of SAA against \textit{any} distribution.

Quite notably, this enables, for the first time, to fully characterize   the spectrum of performances achievable by SAA \textit{across data sizes}. 
These results highlight many fundamentally novel insights on the value of information.  As examples, with 20 samples, SAA already leads to a relative regret of  $26.8 \%$, and with 100 samples, the relative regret shrinks to $8.1 \%$. This highlights the possibility of making effective decisions already with very limited data.
 As a matter of fact, in \Cref{tab:Number_samples} below, we show that the number of samples required by SAA to achieve a certain level of accuracy, as derived from the analysis in this paper, is actually two orders of magnitudes lower than the number induced by state-of-the art upper bounds on the expected relative regret in the existing literature (see \Cref{sec:comp} and \Cref{sec:opt-VI}).
\begin{table}[h!]
\centering
\begin{tabular}{lllllll}
& & \multicolumn{5}{c}{Expected relative regret target}\\
\cline{3-7}
& & 25\% & 20\% & 15\% & 10\% & 5\%  \\ 
 \hline
SAA & This paper &  21&  23&  42& 71&  210 \\
        & Best known to date    &  5,088&  7,780&  13,530& 29,762& 100,000+ \\
\hline
Optimal Algorithm&  This paper  &  14&  19&  25& 50& 161\\
\end{tabular}
\caption{\textbf{Number of samples that ensures  a target relative regret.}  The table reports the induced number of samples needed to reach a relative regret accuracy level. For SAA, we  compare the best instance-independent known bounds to date \citep{levi2}  and the exact worst-case analysis  developed in the present paper.  We also report  the number of samples needed by an optimal data-driven algorithm, derived in the present paper. Example with service level of $0.9$.}
\label{tab:Number_samples}
\end{table}

Our analysis of the SAA policy also leads to another striking new insight: the relative regret is not monotone in the number of samples available. This  implies that SAA is suboptimal but also that SAA is not able to accumulate information appropriately, sometimes ``destroying'' information. We highlight that the possibility to uncover this non-monotonic behavior has been enabled by the transient regime lens for learning we take.

 \textit{Optimal data-driven policy.} In turn, the next question we tackle pertains to optimal worst-case performance in the space of data-driven policies. Indeed, it is important to note that SAA is only one possible prescription among all possible mappings from data to decisions.  Our next contribution lies in characterizing the minimal worst-case expected relative regret  in the general space of data-driven policies and across all data sizes. In the remaining of the paper we refer to the algorithm achieving the minimax optimality for the worst-case expected relative regret as the optimal algorithm.
 
To prove this fundamental optimality result, we first establish a central reduction in the space of mechanisms. We show that, \textit{without loss of optimality}, one may restrict attention to mixtures of order statistics (\Cref{thm:minimax_reduction}). In turn, we derive necessary conditions for optimality in that subspace. This leads to a candidate policy. The last step consists of establishing optimality of this candidate.  For that, we introduce an alternative minimax problem in which we relax the space of distributions to be distributions over distributions, and show that the candidate, together with  a proper mixture of distributions, is a saddle point for the alternate problem. This yields an optimal data-driven algorithm and its associated performance for the original problem (\Cref{thm:optimal_min_max}).

We establish that an optimal data-driven policy takes actually a simple form: it is a randomization between two consecutive order statistics, and we  provide a procedure to compute its tuning parameters as well as its performance. As a corollary, we obtain that an alternative policy, which prescribes a convex combination of consecutive order statistics, is also minimax optimal while also (weakly) improving over the optimal mixture of order statistics policy for all possible demand distributions.

This result has significant implications. First, we can now assess the potential losses stemming from using the suboptimal SAA policy. We show that these can be significant for small data sizes and become smaller as the data size increases. Second, it allows to exactly \textit{assess the value of the historical information} and to understand how effective one can be as a function of the data at hand without any  assumption on the underlying distribution. This further emphasizes the possibility of operating effectively with limited data. In  \Cref{tab:Number_samples}, we report the number of samples required to reach a particular level of accuracy, and one sees that, compared to SAA,  the optimal algorithm reduces the amount of data  needed to reach a particular level significantly (by 17\% to 40\% across the targets illustrated). We also note that, as a byproduct of our analysis, we also obtain the worst-case  performance ratio for any Bayesian problem.

We highlight here that there has been significant work on data-driven policies. We discuss more these in the literature review and refer to \cite{lam2021impossibility} for a very recent overview of various subfamilies of policies considered in the literature. We also emphasize that when searching for optimal policies, we consider all possible mappings from data to decisions, and do not restrict attention to a subfamily of policies.

We note that, while our analysis is tailored around the worst-case relative regret, the minimax optimal policy that we derive is not overly conservative on ``mild'' instances. As a matter of fact, we show numerically in \Cref{sec:experiment} that the performance of the optimal policy is typically on par or better than the one of SAA on a broad range of distributions. As a consequence, the ``robustification'' of SAA in the worst-case comes at no cost and even  typically translates into better performance on many common distributions.

 \textit{Optimal asymptotic performance.}  When the data size becomes large, there are various existing results in the literature, and a corollary of these leads to upper bounds on the rate of convergence to zero of the expected relative regret of SAA  as the data size $\numS$ grows to infinity: it converges to zero at rate $O(1/\sqrt{\numS})$. While this makes progress on capturing the  dependence in the data size $\numS$, even asymptotically, there is still limited understanding of the actual performance. In particular, there is no understanding of the  \textit{constant} characterizing the rate of convergence for SAA, nor for the best such constant achievable by a data-driven algorithm.
  
  We leverage the exact finite sample analysis to derive, from the bottom up, the exact rate of convergence to zero, fully characterizing the constant for SAA and optimal performance. In particular, we show that the optimal relative regret asymptotically scales like $C^*/ \sqrt{\numS}$ where the number of samples $\numS$ is large and provide a closed form expression for $C^*$ (\Cref{thm:asymptotic}). This highlights how the rate of convergence is affected by the various economic parameters associated with the newsvendor decision. In addition, we establish that SAA asymptotically achieves rate optimality with the \textit{same} limiting constant.  As such, while SAA could lead to high suboptimality gaps for small data sizes, it satisfies a very strong notion of near-optimality for large data sizes.

 Stepping back, one may see the present study as building a foundation for a ``bottom-up'' approach to data-driven decision-making in newsvendor settings. We highlight that our transient regime lens for learning and the associated exact analysis account for all the expected out-of-sample cost implications of deviations, small or large,  that SAA (or an optimal policy) could generate compared to the oracle. As such, it allows to build an understanding of data-driven policies ``one data point at a time.'' Compared to existing approaches that are mostly anchored around large data regimes, this new perspective establishes that it is indeed possible to characterize performance across data sizes. It leads to new insights for small as well as large data regimes. We hope that the techniques developed here lead to further the understanding of the transient regime of learning in richer settings relating to newsvendors, but also across problem classes.  

\vspace{-3mm}
\subsection{Related literature}

Capacity management problems in the face of uncertainty are central across literatures and the present paper builds on and contributes to a vast existing literature.

 Our work first relates to the study of this class of problems with limited information on the underlying distribution of demand. In early work, \cite{scarf1958min} and \cite{gallego1993distribution} characterizes the min-max optimal solution for the inventory problem when the mean and the variance of the demand are known. \cite{perakis2008regret}  derive robust policies that achieve minimax regret under various types of partial information on the demand function such as moments of the distribution, modes or symmetry.   \cite{natarajan2018asymmetry} carries this robustness analysis for asymmetric distributions. 
 
 Information about the distribution may also be improved by a data-driven approach. \cite{xu2021robust} construct ambiguity sets by using non-parametric information on the distribution along with observed samples. \cite{saghafian2016newsvendor} develop a Maximum Entropy approach to leverage information from data combined with moment and tail bounds.  \cite{liyanage2005practical,chu2008solving}  introduce the operational statistic framework which integrates estimation and optimization tasks for newsvendor problems under parametric classes of distributions. \cite{chu2008solving}, assuming that the decision-maker knows the distribution of demand up to a scale parameter,  derives a mapping from data to decision that maximizes expected profit for any value of the unknown scale parameter.   In the present work, we do not make any assumption on the underlying distribution of demand.  
 
When no such information is initially available, the question becomes how to go from data to decisions. There are various dimensions associated with this problem, first on the level of uncertainty about the underlying distribution, and second on the offline or online aspect of the decision-making problem. 
  
 The present paper focuses on a non-parametric setting in which little, if anything is known about the underlying distribution and only data in the form of samples can be used to make decisions. A first foundational question for this class of problems is one pertaining to sample complexity: How many samples are needed to achieve a certain level of accuracy?   Closest  to our work  are \cite{levi2007approximation}, \cite{levi2} and \cite{cheung2019sampling} which establish bounds on probabilistic guarantees of the relative regret of Sample Average Approximation (SAA). In particular \cite{levi2} presents bounds that are problem-independent and apply to any distribution,  and we compare to those in  \Cref{sec:comp}.  \cite{levi2} also improve these bounds by deriving instance dependent guarantees in cases where the decision-maker has additional information about the class of distributions to which the demand belongs. In contrast, our work improves the characterization of the worst-case expected relative regret of SAA without any supplementary information about the distribution. \cite{cheung2019sampling} provides a lower bound on the number of samples required to achieve a target relative regret with a probability exceeding a given threshold. Their result implies that the upper bound derived in \cite{levi2} has the correct dependence in the problem parameters.    \cite{ban2020confidence}  establishes consistent estimators for $(s,S)$ policies for both censored and uncensored information regimes. They derive bounds on the regret by constructing asymptotic confidence intervals around the policy.

  In the offline setting, other related papers study the contextual version of the problem in which the decision-maker observes previous samples of demand along with features that give additional information on the environment \citep{ban2019big,qi2021distributionally}. \cite{ban2019big}  proposes approaches based on Empirical Risk Minimization and kernel methods to derive generalization bounds for the cost of a feature-based data-driven decision. In the special case without contexts, their approach recovers the instance-independent bound derived by \cite{levi2}. In contextual optimization, we also refer the reader to  \cite{elmachtoub2021smart} for a general data-driven approach that explicitly accounts for the nature of the optimization problem at hand.

  Our paper also relates to the rich literature on the analysis of SAA. This approach has been applied broadly for discrete stochastic optimization problems when the underlying distribution is either unknown or when the expected objective function is hard to optimize, \cite{SAA_First}, or for multi-stage stochastic optimization problem (\cite{swamy2005sampling, shapiro2008stochastic}). It has also been used in the specific context of multi-stage inventory planning (\cite{levi2007approximation,cheung2019sampling}) or for the newsvendor model (\cite{levi2, besbes2013implications}). 

    \cite{gupta2022data}  share the motivation of limited data sizes, and explore the possibilities associated with pooling data across products.   \cite{ bertsimas2018robust}, \cite{esfahani2018data} develop robust approaches enjoying  probabilistic guarantee over the out-of-sample error.  \cite{bertsimas2018robust} considers ambiguity sets containing all distributions that pass a statistical goodness-of-fit test for given historical data. \cite{esfahani2018data} proposes a data-driven distributionally robust approach by constructing an uncertainty ball around the empirical distribution. They show that the worst-case expectation over a Wasserstein ambiguity set can in fact be computed efficiently via convex optimization techniques for various loss functions. In contrast to theses strategies, we highlight that we do not restrict the space of policies to those that construct uncertainty sets, but explore the entire space of mappings from data to decisions.

More broadly, our work relates to sequential  decision-making under uncertainty. In the class of inventory decisions, a rich line of work on dynamic decisions has been developed.  Different information structures (observable or censored demand) are studied and adaptive algorithms with desirable asymptotic properties derived. \cite{godfrey2001adaptive,huh2009nonparametric,RyzinMcGill2000Revenue} develop gradient based methods to solve this sequential problem whereas \cite{huh2011adaptive} uses the Kaplan-Meier estimator and \cite{ErenMaglaras2015maximum} studies a maximum entropy approach to dynamically adjust capacity levels.  \cite{besbes2013implications} studies the price of demand censoring in these sequential decision making problem with stationary demand, and \cite{lugosi2021hardness} study censoring in a  setting when demand is non-stationary.
\cite{chen2021nonparametric} studies the interplay of pricing and inventory decisions. We note that in all these studies, the performance of policies are characterized asymptotically up to multiplicative constants, but there is no characterization of exact optimal performance. We hope that the exact performance characterization, and optimality results, developed in this work in the offline case with  demand observations, will lead to future progress in this related class of problems.

Our work is also connected to the rich literature in statistics \label{page:lit-stat} which focuses on the performance of various quantile estimators, and our problem may be reframed as the one of deriving minimax quantile estimators for a particular metric (here the relative regret between newsvendor losses). Depending on the application and the desired properties, many quantile estimators have been derived, either as L-estimators based on order statistics \citep{harrell1982new,Kalgh82,yang1985smooth} or by using different methods such as Stochastic Approximation \citep{tierney1983space}. This profusion of heuristics motivated the study of estimators achieving certain forms of optimality. In parametric settings, \cite{rukhin1982estimating}, and \cite{rukhin1983class} derive minimax equivariant quantile estimators for a normalized squared loss. In non-parametric settings, \cite{zielinski1999best} restricts attention to the set of equivariant estimators and derives an estimator uniformly better with respect to a particular metric, the worst-case F-Mean Absolute Deviation.
 Our work differs from \cite{zielinski1999best} along various crucial dimensions. We focus on the commonly studied objective of  minimax relative regret, on a newsvendor cost,  and we do not restrict the space of decisions. 
 
Our work is also remotely related to the understanding of the learning curve defined as the expected generalization performance, i.e., the out-of-sample performance, of a learner as a function of the size of the training set. We refer the reader to the recent review of \cite{viering2021shape}. Our approach complements this line of work and gives a theoretical understanding of the robust (worst-case) value of data sizes for the newsvendor problem.

Finally, we also note that another related line of work pertains to modeling uncertainty differently. Another framework that has been widely studied  is parametric and Bayesian, in which the decision-maker is assumed to have access to a prior about an underlying unknown parameter that characterizes the distribution.  The seminal work of  \cite{scarf1959bayes} analyses a bayesian setting in which the decision maker has a prior belief on the nature of the distribution and updates his belief as he observes samples from the demand. The goal is then to analyze methods that use historical data to prescribe inventory decisions on the fly. This line of work has also a rich literature dealing with different information structures (censored versus uncensored) observations. See, e.g.,  \cite{Azoury1985Bayes}, \cite{LarivierePorteus1999Stalking},   \cite{DingPutermanEtAl2002Censored} (and the related notes by \cite{LuSongEtAl2005The} and \cite{BensoussanCakanyildirimEtAl2009Technical}) and \cite{BesbChane2019exploration}.


\section{Problem Formulation}
We consider a newsvendor problem in which the decision maker decides on a capacity/inventory decision $x$ in the face of uncertainty on the underlying demand $D$ that will materialize. Any excess inventory leads to overage costs $\co >0$ per unit, and any demand that is not satisfied leads to  underage cost of $\cu >0$ per unit.  In turn, the cost associated with decision $x$ is given by 
 \bearn
 c(x,D) := \cu (D - x)^{+} +  \co (x - D)^{+}.
 \eearn

\paragraph{Decision-making with knowledge of the distribution of $D$.} Suppose that demand $D$ is drawn from a distribution $F$ supported on $\mathbb{R}_{+}$. Then, in the classical newsvendor problem, the decision-maker minimizes the expected cost given by
\begin{equation} \label{eq:fullinfo}
c_F(x) :=  \mathbb{E}_{D \sim F} \left[{ \cu (D - x)^{+} +  \co (x - D)^{+}}\right].
\end{equation}

 Let ${\cal G}$ denote the set of distributions (cdf) with non-negative support. By convention, we assume that all cdf's are cadlag. Furthermore, we let 
$$ {\cal F} = \{F \in {\cal G}: \mathbb{E}_F[D] < \infty\}$$
denote the set of distributions with bounded first moment. We will assume that $D$ is drawn from a distribution in $ {\cal F}$, so that the above cost always  admits a well defined expectation.

When the distribution $F$ is known, the inventory decision minimizing the cost $c_F(\cdot)$ is given by
\begin{equation*}
x^*_F := \min \{x \geq 0 : F(x) \geq q\},
\end{equation*}
where \begin{equation*}
q = \frac{\cu}{\cu+\co}.
\end{equation*}
We will refer to $q$ as the critical quantile, and we will let
\bearn
\opt(F) := c_F(x^*_F)
\eearn
denote the minimal achievable cost with knowledge of the distribution $F$.

\paragraph{Data-driven decision-making.}  In the present paper, we consider a setting in which the distribution $F$ is unknown to the decision-maker and only data is available in the form of past demand observations.  The decision-maker observes $\numS$ historical samples of demand,  $\mathbf{D}_1^\numS := (D_1, \ldots, D_{\numS})$, where $D_i$ are independently sampled from $F$.  In this context, an admissible policy $\pi$ is a mapping from observed demand $\mathbf{D}_1^\numS$ to inventory decision $x^\pi$. 
Formally, we consider the class of policies $\Pi_\numS$ of mappings from $\mathbb{R}_{+}^{\numS}$ into the set of distributions $\mathcal{F}$. In particular, a policy $\pi$ is a mapping
\begin{equation*}
\pi : \mathbf{D}_{1}^\numS \mapsto G_{\mathbf{D}_1^\numS},
\end{equation*}
where $G_{\mathbf{D}_1^\numS} \in \mathcal{F}$. That is to say $\pi$ maps previous demand observations to a randomized inventory decision. We note that, even when the policy $\pi$ is a deterministic function of the observed demand $\mathbf{D}_1^\numS$, the inventory decision $x^\pi$ is a random variable as it depends on  $\mathbf{D}_1^\numS$.

When using a policy  $\pi$, given  $\numS$ samples from an underlying distribution $F$, the out-of-sample expected cost incurred is defined as,
\begin{equation*}
\mathcal{C}(\pi,F,\numS) := \mathbb{E}_{\mathbf{D}_1^\numS \sim F}\left [ \mathbb{E}_{x \sim { \pi \left( \mathbf{D}_{1}^\numS \right)} } \left[  c_F(x) \right] \right].
\end{equation*}
Note that the dependence between the underlying demand distribution and the  expected cost of a data-driven algorithm is an intricate one in general, as \textit{the demand distribution $F$ affects the history the decision-maker observes, but also the out-of-sample performance of data-driven decisions.}

\textit{Objective.} We are interested in understanding and quantifying the performance of  data-driven algorithms for newsvendor problems. To that end, we evaluate the performance of a policy $\pi \in \Pi_\numS$ through the relative regret defined for every $F \in \mathcal{F}$ as\footnote{Note that $\opt(F)=0$ if and only if the distribution $F$ has all its mass at a single point. In such a case, we set,  by convention,  $\ratio_\numS(\pi,F) = 0$, for any policy $\pi \in \Pi_\numS$ such that $\mathcal{C} (\pi,F,\numS) = \opt(F) = 0$.},
\begin{equation*}
\ratio_\numS(\pi,F) := \frac{\mathcal{C}(\pi,F,\numS)-\opt(F)}{\opt(F)}.
\end{equation*}
Note that the ratio above is always greater or equal than $0$ and takes value in $[0, \infty) \cup \{ \infty \}$.
Given that the decision-maker does not know the distribution, we evaluate its performance through the worst-case relative regret defined as follows 
\begin{equation}
\label{eq:ratio}
\sup_{F \in \mathcal{F}}  \ratio_\numS(\pi,F).
\end{equation}
It represents the relative loss stemming from the gap between observing data of size $n$ and full information on the demand distribution.   We will be interested in characterizing the performance of specific policies considered in the literature, but also in the optimal achievable performance 
\begin{equation}
\label{eq:minmax}
\ratio^{*}_\numS  := \inf_{\pi \in \Pi_\numS} \sup_{F \in \mathcal{F}}  \ratio_\numS(\pi,F).
\end{equation}
We note that the above problem involves two infinite dimensional optimization problems, that of the decision-maker when selecting a policy, and that of nature when selecting a worst-case distribution. We also remark that there are no restrictions on the class of policy used in \eqref{eq:minmax}.

\noindent \textbf{Notation.} For any $\mu$ in $[0,1]$, we let $\mathcal{B}(\mu)$ denote the distribution of a Bernoulli with mean $\mu$. For any set $A$, $\Delta \left( A \right)$ denotes the set of distributions on $A$. We further let $\Delta_{\numS}$ denote the simplex in $\numS$ dimensions, i.e., $\Delta_n = \{\bm{\lambda} \in \mathbb{R}^n:  \lambda_i \ge 0, \: i \in \{1,\ldots,\numS\}, \: \sum_{i=1}^\numS \lambda_i = 1\}$. For any deterministic sequences $(a_k)_{k \in \mathbb{N}}$ and $(b_k)_{k \in \mathbb{N}}$ both indexed by a common index $k$ that goes to $\infty$, we say that $a_k = o(b_k)$ if $a_k/b_k \to 0$, $a_k = \mathcal{O} \left( b_k \right)$ if there exists a finite $M >0$ such that $|a_k| \leq M |b_k|$ for $k$ large enough, $a_k = \omega(b_k)$ if $|a_k/b_k| \to \infty$, $a_k = \Omega(b_k)$ if there exist a finite $M > 0$, such that $|a_k| > M |b_k|$ for $k$ large enough, $a_k = \Theta \left( b_k \right)$ if $a_k = \mathcal{O} \left( b_k \right)$ and $a_k = \Omega \left( b_k \right)$, and $a_k \sim b_k$ if $a_k/b_k \to 1$.   

All proofs are deferred to the Online Appendix.


\section{Sample Average Approximation: Performance Analysis across Data Sizes} \label{sec:SAA}
 
The data-driven newsvendor problem is a particular instance of a data-driven stochastic optimization problem. One of the most common approaches to solve this type of problems is the Sample Average Approximation (SAA). As highlighted earlier, this approach has been applied to a broad set of problems. Previous works derived convergence guarantees for this method (\cite{SAA_First}) but also bounds on probabilistic finite sample performance (\cite{levi2007approximation,levi2,cheung2019sampling}). In the newsvendor setting, SAA is also related to the broader literature of Distributionally Robust Optimization as it happens to be equivalent to the distributionally robust policy over the Wasserstein ball as noted in \cite[Remark 6.7]{esfahani2018data}.

In the context of the newsvendor problem,  SAA consists in solving the optimization problem
\begin{equation}
\label{eq:empirical_opt}
\min_x \frac{1}{\numS} \left[ \sum_{i=1}^\numS \cu (D_i -x)^+ + \co (x - D_i )^+ \right].
\end{equation}
This approach approximates the expectation in \eqref{eq:fullinfo}  with the empirical expectation, and solves the resulting problem. In particular, the solution of problem \eqref{eq:empirical_opt} is the $q^{th}$-empirical quantile. More precisely, let us define the order statistics of the historical dataset of demands observed as  
\bearn
D_{1:\numS} \le \ldots \le D_{\numS:\numS}. 
\eearn
The SAA policy, which we will denote by $ \SAA$ prescribes the $\ceil{q \numS}^{th}$ order statistic. With some abuse of notation, we have\footnote{Technically speaking, this is the policy that prescribes a point mass at $D_{\ceil{q \numS}:\numS}$.}
 \begin{equation}
 \label{eq:SAA_OS}
 \SAA(\mathbf{D}_{1}^\numS) = D_{\ceil{q \numS}:\numS}.
 \end{equation}

As highlighted in the introduction, this policy has been extensively studied. In particular, as the number of demand samples $\numS$ grows, it is known that SAA leads to a solution  that ensures that its worst-case relative regret approaches $0$ as $\numS$ grows to $\infty$. Previous approaches derived upper bounds on the rate at which such convergence takes place, leveraging large deviations bounds. However, despite its widespread use, there is no characterization of its \textit{actual} performance for a finite number of samples, and as a result, there is no robust quantification of the amount of data needed to achieve a particular level of performance.

Analyzing exactly the worst-case performance of SAA, or any policy, as in Problem \eqref{eq:ratio}, is an infinite dimensional optimization problem over a non-parametric class of distributions. For any policy $\pi \in \Pi_\numS$, let $G^{\pi}_{\mathbf{D}_1^{\numS}} := \pi \left( \mathbf{D}_1^{\numS} \right)$ be the distribution induced by $\pi$ on the inventory level conditional on observing historical demand $\mathbf{D}_1^{\numS}$ .
The cost incurred by a policy $\pi \in \Pi_\numS$ against a distribution $F \in \mathcal{G}$ can be expressed as follows 
\begin{align}
\label{eq:general_cost}
\mathcal{C} \left( \pi, F, n \right) =  \int_{ \mathbf{D}_1^{\numS} \in [0, \infty)^{\numS} } \int_0^{\infty} \int_0^{\infty} c(x,D)  \: dF(D) \: dG^{\pi}_{\mathbf{D}_1^{\numS}} (x) \: d F(D_1) \: \ldots  \: d F(D_{\numS})
\end{align}
Therefore, the cost of a data-driven policy has in general a complex dependence on the demand distribution which appears in the integration measure.

In what follows, for any policy $\pi \in \Pi_\numS$, when solving for the worst-case performance,  we will be working with the epigraph formulation of problem \eqref{eq:ratio} . Note that $\sup_{F \in \mathcal{F}} \ratio_\numS(\pi,F) \ge 0$ and
\begin{subequations}
\begin{alignat*}{2}
\sup_{F \in \mathcal{F}} \ratio_\numS(\pi, F) \:=\: &\! \inf_{z \in \mathbb{R}_{+} }        &\qquad& z \nonumber\\
&\text{s.t.} &      & \ratio_\numS(\pi, F) \leq z \qquad \forall F \in \mathcal{F}. \nonumber
\end{alignat*}
\end{subequations}
It is easy to see that the problem can be further written as (this claim is formally established in \Cref{lem:epigraph_form})
\begin{subequations}
\label{eq:whole_class_pb_body}
\begin{alignat}{2}
\sup_{F \in \mathcal{F}} \ratio_\numS(\pi, F) \:=\: &\! \inf_{z \in \mathbb{R}_{+} }        &\qquad& z \\
&\text{s.t.} &      & \mathcal{C} \left( \pi, F, \numS \right) - (z+1) \opt(F) \leq 0 \qquad \forall F \in \mathcal{F}. \label{eq:constraint_epi_body}
\end{alignat}
\end{subequations}
This problem thus involves infinitely many constraints, and each of the constraints has a complex dependence in $F$ as highlighted in \eqref{eq:general_cost}. 
In \Cref{sec:OS_structure}, we analyze a general class of order statistic policies, of which SAA is a special case, and  show  that for these, the optimization problem \eqref{eq:whole_class_pb_body}  can be significantly simplified and as a matter of fact exactly solved. In particular, we leverage the structure of order statistic policies to simplify the expression of the cost function described in \eqref{eq:general_cost}. This allows us to reduce the set of constraints \eqref{eq:constraint_epi_body} to constraints parametrized by a one dimensional set.

\subsection{Order Statistic Policies and Structural Results} \label{sec:OS_structure}
As presented in \eqref{eq:SAA_OS}, in the context of the newsvendor problem, SAA prescribes an inventory level according to an order statistic of the samples observed, the $\ceil{q \numS}^{th}$ order statistic. This   can be seen as a special case of prescribing an order statistic or even a randomization over order statistics. To that end, we next define general order statistics policies which will also play a central role when we discuss optimal performance in \Cref{sec:optimal}.
\begin{definition}[Mixture of Order Statistics]\label{def:OS}
Fix $\numS \ge 1$. For every $i \in \{1,\ldots,\numS\}$, we let $\OS{i}$ denote the policy that uses the $i^{th}$ order statistic $D_{i:\numS}$ with probability one. Formally,
\begin{equation*}
\OS{i} : \mathbf{D}_1^{\numS}  \mapsto \mathbbm{1} \left \{ x \geq D_{i:\numS} \right \} .
\end{equation*}
For any $\bm{\lambda} \in \simplex$. We let   $\pi^{\bm{\lambda}}$ denote the mixture of order statistics policy defined as follows: $\pi^{\bm{\lambda}}$  prescribes  the $i^{th}$ order statistic $D_{i:\numS}$ with probability $\lambda_i$. Formally, 
\begin{equation*}
\mix : \mathbf{D}_1^{\numS}  \mapsto \sum_{i=1}^{\numS} \lambda_i \mathbbm{1} \left \{ x \geq D_{i:\numS} \right \}.
\end{equation*}
 In the following, we denote by $\Pi_\numS^{OS}$ the space of mixture of order statistics policies with $\numS$ samples. 
\end{definition}
Another important class of policies are convex combinations of order statistics policies, which prescribe a deterministic convex combination instead of randomizing between different order statistics. We will relate the performance of policies in this class to the one for mixtures of order statistics policies in \Cref{sec:lower_bound}.

For a policy $\pi \in \Pi_{\numS}$, the expression of the relative regret involves the ratio between  $\mathcal{C} \left( \pi, F,\numS \right)$ and $\opt(F)$. In general, both quantities require to compute complex integral expressions in which the dependence on the demand distribution is intricate. 
Our first structural result establishes that for a mixture of order statistics policy, the cost incurred against any demand distribution $F$ can be expressed as a single integral in which the integrand is a polynomial of the demand distribution and the integrating measure is the Lebesgue measure. We similarly show that the cost of the oracle is the integral of a piecewise-linear function of the demand distribution.  Formally, we show the following.
\begin{proposition}
\label{lem:quantile_policy_cost}
For any $F \in \mathcal{F}$, any $\numS \geq 1$, and any mixture of order statistics policy $\pi^{\bm{\lambda}}$ we have,
\bearn
\mathcal{C} \left( \pi^{\bm{\lambda}}, F, n \right) &=& (\cu+\co) \left[\int_0^{\infty} \sum_{i=1}^{\numS} \lambda_i \left( (1- B_{i,\numS} (F(y)) ) (F(y)-q)  + q(1-F(y)) \right)dy \right],\\
\opt(F) &=& (\cu+\co) \int_0^{\infty}   \min\{(1-q)F(y),  q(1-F(y))\} dy,
\eearn
where $B_{i,\numS}$ is a Bernstein polynomial defined  for any $y \in [0,1]$ as
\begin{equation*}
\label{eq:order_stat}
B_{i,\numS}(y)  = \sum_{j=i}^\numS b_{j,\numS}(y),
\end{equation*}
with $b_{j,\numS}(y) = {\numS \choose j} y^j(1-y)^{\numS-j}$.
\end{proposition}

Recall the expression in \eqref{eq:general_cost}, \Cref{lem:quantile_policy_cost} shows that for mixture of order statistics policies, the cost function can be expressed as an integral in which the dependence in $F$ \textit{only} appears in the integrand but does not appear in the integrating measure anymore.  As we will see, this is a key step towards the understanding of worst case distributions for this family of policies.

The main step in the proof of this result follows from Riemann–Stieltjes integration by part and from exploiting the special form of the cumulative distribution function of an order statistic. Indeed, for any $F \in \mathcal{F}$, $\numS \geq 1$ and $r \in \{1,\ldots, \numS\}$, the cumulative distribution of $D_{r:\numS}$ denoted by $F_{r:\numS}$ satisfies for $x \in \mathbb{R}_{+}$,
\begin{equation*}
F_{r:\numS}(x) = B_{r,\numS} \left( F(x) \right).
\end{equation*}
The new expressions in \Cref{lem:quantile_policy_cost} imply that the epigraph formulation \eqref{eq:whole_class_pb_body} can be simplified. In particular, by rewriting the set of constraints \eqref{eq:constraint_epi_body}, we obtain the following formulation.
\begin{subequations}
\label{eq:whole_class_pb_2_body}
\begin{alignat}{2}
 &\! \inf_{z \in \mathbb{R}  }        &\qquad& z \\
&\text{s.t.} &      & \sup_{F \in \cal{F}} \int_0^\infty \Psi_z^{\bm{ \lambda}}(F(y)) dy \leq 0,
\end{alignat}
\end{subequations}
where $\Psi_z^{\bm{\lambda}}$ is a continuous mapping from $[0,1]$ to $\mathbb{R}$.
The constraint of \eqref{eq:whole_class_pb_2_body} now involves a nonparametric optimization problem for which the demand distribution only appears in the integrand. This expression allows us to reduce the functional optimization problem over the class of distributions $\mathcal{F}$ to a pointwise optimization problem. 
We now present formally this result as our first main contribution. 

Recall that $\mathcal{B} \left( \mu \right)$ denotes a Bernoulli distribution with mean $\mu$.
Our first main result is a characterization of the exact performance of any mixture of order statistics policy. 
\begin{theorem}
\label{thm:main_mixture}
Fix $\numS \ge 1$ and  $\mix \in \Pi_\numS^{OS}$. The worst-case performance of the policy  $\pi^{\bm{\lambda}}$ satisfies
\bearn
\sup_{F \in \mathcal{F}} \ratio_\numS(\pi^{\bm{\lambda} }, F)  &=& \sup_{\mu \in [0,1]} \ratio_\numS(\pi^{\bm{\lambda}}, \mathcal{B}(\mu)).
\eearn
Furthermore for every $\mu \in [0,1]$,
\begin{equation*}
 \ratio_\numS(\pi^{\bm{\lambda}}, \mathcal{B}(\mu)) = \sum_{i=1}^{\numS} \lambda_i \frac{ (1- B_{i,\numS}(1-\mu))(1 - \mu - q) + q \cdot \mu } { \min \left\{ (1-q) (1 - \mu), q \cdot \mu \right\}} - 1.
\end{equation*}

\end{theorem}
This result has many implications.  First, it establishes the notable fact that for any data size $\numS$, the worst-case performance of mixture of order statistics policies over the \textit{entire} space of distributions ${\cal F}$  is achieved at a Bernoulli distribution. 

Second, in evaluating the worst-case performance of these policies,  \Cref{thm:main_mixture}  leads to a reduction from a non-parametric  general space of distributions  $\mathcal{F}$ to a space of distributions parametrized by a single parameter, the mean of the Bernoulli distribution. Moreover,  in the case of Bernoulli distributions, the relative regret  has a \textit{closed-form} expression.
Therefore, the \textit{exact} worst-case performance of mixture of order statistics policies can be computed for \textit{any} number of samples $\numS$ through a \textit{simple line search.}  In \Cref{sec:sub_analysis_SAA}, we analyze the implications of this result for SAA. This result will also be instrumental when we analyze optimal policies in the entire space of mappings from data to decisions in \Cref{sec:optimal}.

\paragraph{Remark.} At first glance, the result of \Cref{thm:main_mixture} may seem counter-intuitive as one would expect that in the broad family of distributions $\mathcal{F}$, a ``hard'' instance for order statistics policies would have unbounded support. This result proves that on the contrary, the difficulty of the data-driven newsvendor problem does not stem from the tail of the distribution. Indeed,  for distributions that are hard to learn, such as heavy-tail distributions, oracle costs are also large. Bernoulli distributions are flexible enough to lead to a low cost for the oracle due to the simple structure of the distribution, but also exacerbates the cost of mistakes for a decision-maker that does not know the distribution as the problem boils down to deciding between two extreme actions: prescribing $0$ or $1$.

\paragraph{Remark (Absolute vs. relative regret).}
We highlight here that a similar argument as the one developed to prove \Cref{thm:main_mixture} can be used to establish a parallel result for the worst-case expected \emph{absolute} regret metric,  $\mathcal{C}(\pi,F,\numS)-\opt(F)$. In such a case, if one restricts attention to the space of distributions supported on a  bounded interval $[0,M]$ (for some real value $M$), then one can show that the worst-case expected absolute regret for any mixture of order statistics is achieved for a two point distribution with mass at $0$ and $M$.

\subsection{Performance Analysis of SAA across Data Sizes}\label{sec:sub_analysis_SAA}

A direct and important consequence of \Cref{thm:main_mixture} is the ability to characterize the transient regime of learning, or exact performance associated with the central algorithm SAA for an \textit{arbitrary number of samples}. Since SAA is a special case of mixture of order statistics policy, the following theorem is a direct corollary.
\begin{theorem}[SAA Finite Sample Performance] \label{thm:SAA} 
For any $\numS \geq 1$, the performance of the SAA policy is given by
\bearn
\sup_{F \in \mathcal{F}} \ratio_\numS( \SAA, F)  &=&  \sup_{\mu \in [0,1]} \ratio_\numS(\SAA, \mathcal{B}(\mu)) = \sup_{\mu \in [0,1]} \frac{ (1- B_{\ceil{q \numS} ,\numS}(1-\mu))(1 - \mu - q) + q \cdot \mu } { \min \left\{ (1-q) (1 - \mu), q \cdot \mu \right\}} - 1.
\eearn
\end{theorem}
\Cref{thm:SAA} leads to the notable result that one may  characterize \textit{exactly} the worst-case performance of the central SAA algorithm across \textit{all data sizes} by performing a simple line search!  
As such, it allows to exactly measure the implications of all possible \textit{out-of-sample} ``mistakes'' (compared to the oracle) made by SAA, and this is for any  data size. Next, we analyze the implications of this result on the value of data, compare this result to  earlier bounds in the literature, as well as uncover novel insights on the quality of SAA as a data-driven policy in this class of problems.

\subsubsection{Performance of SAA and Comparison to Existing Related Results}\label{sec:comp}
As mentioned earlier, SAA has been widely studied in various settings. In the context of the newsvendor problem,   \cite{levi2007approximation}  establish bounds relying on large deviations arguments to derive probabilistic results, which were later improved in \cite{levi2}, with associated relative regret guarantees. More formally, \cite[Theorem 2]{levi2} show that for any $\epsilon > 0$, any $\numS \geq 1$ and any demand distribution $F$, the relative-regret of SAA satisfies
\begin{equation*}
\label{eq:levi2}
\mathbb{P}_{\SAA} \left( \frac{c_F(x) - \opt(F)}{\opt(F)} > \epsilon \right)  \leq 2 \text{exp} \left (-\frac{\numS \epsilon^2}{18 + 8 \epsilon} \min(q,(1-q)) \right),
\end{equation*}
with the associated bound on relative regret given by
\bearn
 \sup_{F \in \mathcal{F}} \ratio_{\numS} \left( \SAA, F \right)  &=& \sup_{F \in \mathcal{F}} \mathbb{E}_{\SAA} \left[ \frac{c_F(x) - \opt(F)}{\opt(F)}  \right] \\
 &=& \sup_{F \in \mathcal{F}} \int_0^\infty \mathbb{P}_{\SAA} \left( \frac{c_F(x) - \opt(F)}{\opt(F)} > \epsilon \right)  d \epsilon \\
 &\le&  \int_0^\infty 2 \text{exp} \left (-\frac{\numS \epsilon^2}{18 + 8 \epsilon} \min(q,(1-q)) \right) d\epsilon =: U(\numS).
\eearn
The function $U(\numS)$ represents state of the art instance-independent bounds for the worst-case performance of SAA as a function of the data size in the literature to date. Our result in \Cref{thm:SAA}  allows to characterize $\sup_{F \in \mathcal{F}} \ratio_{\numS} \left( \SAA, F \right) $,  the actual worst-case performance of SAA. In \Cref{tab:complete_comparison}, we present a comparison of the number of samples required to guarantee various levels of relative regret ($25\%$, $20\%$,...,$5\%$) for different values of the critical fractile. We do so using the induced number  based on \Cref{thm:SAA} in this paper, and based on $U(\numS)$. Formally, for a given performance threshold $\tau \geq 0$, we define
\begin{eqnarray*}
N^{\text{exact-SAA}} (\tau) &:=& \min \left\{ m \geq 1 \, \Big \vert \, \forall \numS \geq m, \: \sup_{F \in \mathcal{F}} \ratio_{\numS}\left( \SAA,F \right) \leq \tau   \right\}\\
N^{\text{UB}} (\tau) &:=& \min \left\{ m \geq 1 \, \Big \vert \, \forall \numS \geq m, \: U(\numS) \leq \tau   \right\}.
\end{eqnarray*}

\begin{table}[h!]
\centering
\begin{tabular}{c l   c c c c c }
&& \multicolumn{5}{c}{Expected relative regret target ($\tau$)}\\
\cline{3-7}
q&Bound used & 25\% & 20\% & 15\% & 10\% & 5\%  \\ 
 \hline
 \hline
 \vspace{-4mm}
 \\
 \multirow{3}{*}{.7}&  $N^{\text{UB}}(\tau) $ (best known to date)   &  1,696&  2,594&  4,510& 9,921& 38,779\\
&$N^{\text{exact-SAA}}(\tau)$ (this paper) &  8&  11&  15& 31&  84  \\
\hline
 \vspace{-4mm}
 \\
 \multirow{3}{*}{.8}& $N^{\text{UB}}(\tau)$   &  2,544&  3,890&  6,765& 14,881& 58,168\\
& $N^{\text{exact-SAA}}(\tau)$ &  11&  16&  21& 41&  116  \\
\hline
 \vspace{-4mm}
 \\
\multirow{3}{*}{.9}& $N^{\text{UB}}(\tau)$    &  5,088&  7,780&  13,530& 29,762& 100,000+\\
& $N^{\text{exact-SAA}}(\tau)$ &  21&  23&  42& 71&  210  \\
\end{tabular}
\caption{\textbf{Number of samples ensuring that SAA achieves a target relative regret.}  The table reports  induced number of samples needed to reach a relative regret accuracy level, comparing the best instance-independent known bounds to date $U(\numS)$ \citep{levi2} and the exact worst-case analysis of SAA developed in \Cref{thm:SAA}, for different values of the critical fractile $q$.}
\label{tab:complete_comparison}
\end{table}

Notably, the exact analysis developed in the present paper yields a number of samples two orders of magnitude lower than the best known guarantee to date.  
The improvements above stem from the novel type of analysis conducted that enables to quantify the implications of  all out-of-sample ``mistakes''  (compared to the oracle) that SAA could do, compared to the existing approaches for SAA analysis that are anchored around large deviations bounds to ensure near-optimality of the SAA solution. 

Another fundamental  insight from \Cref{tab:complete_comparison} stems from the actual values of the minimum number of samples  $N^{\text{exact}}$ to ensure a particular relative regret level. For example, less than 71 samples are sufficient to achieve a relative regret  of $10\%$ for the various critical fractiles above! \Cref{thm:SAA} and the associated bounds enable to develop a new understanding of the value of data sizes, highlighting that smaller data sizes are extremely valuable and lead to very effective decisions for this class of problems.   In practice, even in a data-rich environments such as online retail, the time granularity used to evaluate the demand is usually at a weekly level. As a consequence, a year of demand data for a single product may only represent tens of samples. The above table highlights that such data sizes already ensure very strong performance.

\subsubsection{Transient Regime of Learning for  SAA and Non-Monotonicity}\label{sec:non_monotone}

In  \Cref{fig:exact_SAA}, we depict the exact worst-case performance of $\SAA$ for sample sizes ranging from $2$ to $100$, with a critical fractile of $q$ in $\{ 0.7, 0.8, 0.9\}$. We emphasize  that the performance depicted is the \textit{exact worst-case} relative regret of SAA and not a bound on it.  Various observations are striking.

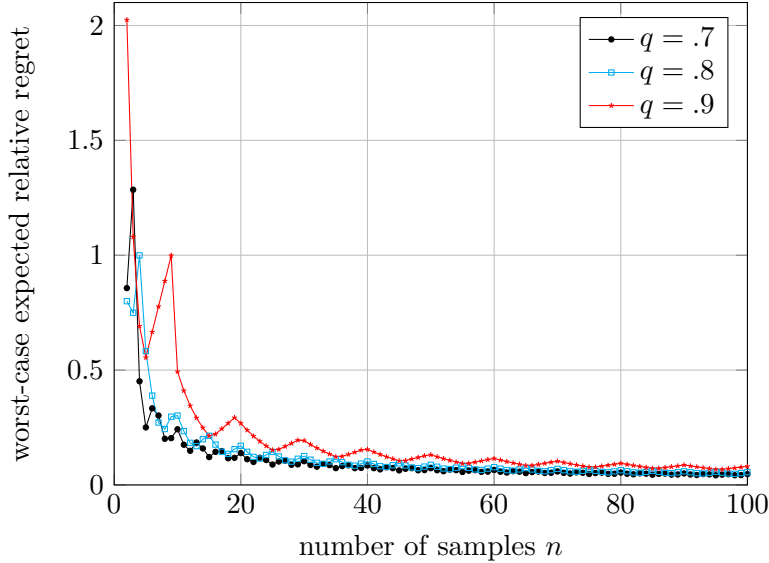
\begin{figure}[h!]
\begin{center}
\begin{tikzpicture}
\begin{axis}[
            title={},
            xmin=0,xmax=100,
            ymin=0.0,ymax=2.1,
            width=10cm,
            height=8cm,
            table/col sep=comma,
            xlabel = number of samples $\numS$,
            ylabel = worst-case expected relative regret,
            grid=both,
            skip coords between index={0}{1},
            legend pos=north east]      
\addplot [black,mark=*,mark options={scale=.5}] table[x=n,y=SAA] {Data/SAA_and_opt_paper_07.csv};
    \addlegendentry{$q=.7$}
\addplot [cyan,mark=square,mark options={scale=.5}] table[x=n,y=SAA] {Data/SAA_and_opt_paper_08.csv};
    \addlegendentry{$q=.8$}
\addplot [red,mark=star,mark options={scale=.5}] table[x=n,y=SAA] {Data/SAA_and_opt_paper_09.csv};
    \addlegendentry{$q=.9$}
\end{axis}
\end{tikzpicture}
\end{center}
\caption{\textbf{SAA performance}. The figure depicts the performance of the SAA policy as a function of the number of samples $\numS$ for different critical fractiles.  }\label{fig:exact_SAA}
\end{figure}

First, we observe  that the relative regret decays sharply even after observing very few samples $\numS$.
Consider the case where $q = .9$. With 10 samples, SAA is guaranteed to achieve a relative regret of $ 49.3\% $, with 20 samples it achieves $26.8 \%$ and with 100 samples, $8.1 \%$.  It highlights again the impressive guarantees that  SAA yields for the newsvendor problem even when the number of samples  is small. It also shows how good SAA is at capturing information relevant to the underlying optimization problem. Indeed, one would not expect such a quick decay when trying to estimate the entire demand distribution. This in turn leads to a new understanding of the transient regime of learning and the performance possibilities across data sizes, small or large.

Another highly notable observation in  \Cref{fig:exact_SAA} is that the worst-case performance of SAA is non-monotone in the number of samples $\numS$. The performance curve admits various peaks. We emphasize that the peaks observed are not due to stochasticity when evaluating the performance of the policy but represent an actual deterioration of the performance in the worst-case for SAA when adding a sample. This result can seem counter-intuitive and  establishes two  notable facts: i) SAA is  a suboptimal data-driven policy for various sample sizes; and ii) furthermore, more data is not synonymous with better worst-case performance when using SAA.
 
\paragraph{Remark (Non-monotonicity).} \label{R2-non-mono} Note that above, when considering  the the worst-case relative regret, the worst-case distribution can change with the data size $\numS$. Another question could  consist in assessing whether there exists a \emph{fixed} distribution $F$ and a data size $\numS$ such that the performance deteriorates from $\numS$ samples to $\numS+1$ samples (from the same distribution $F$). In a few problem classes, examples have been exhibited  such as pricing, from one to two samples \citep{babaioff2018two}  and misspecified linear regression \citep{loog2019minimizers}. We next argue that the non-monotonicity observed in the worst-case performance in \Cref{fig:exact_SAA} is actually a stronger statement as it implies that there exists a distribution $F \in \mathcal{F}$ and a data size $\numS$ such that, 
\begin{equation*}
\ratio_\numS(\SAA, F) < \ratio_{\numS+1}(\SAA, F).
\label{eq:fixed_dist_non_monotone}
\end{equation*}
Indeed, let $\numS$ be such that the worst-case relative regret is increasing by adding an additional sample to $\numS$. In other words, we have  
\begin{equation*}
\delta := \sup_{F \in \mathcal{F}}  \ratio_{\numS+1}(\SAA, F) - \sup_{F \in \mathcal{F}} \ratio_\numS(\SAA, F) > 0.
\end{equation*}
Then, consider a distribution $F^* \in \mathcal{F}$ such that $\ratio_{\numS+1}(\SAA, F^*) > \sup_{F \in \mathcal{F} }{\ratio_{\numS+1}(\SAA, F)} - \delta$. We have
\begin{equation*}
\ratio_\numS(\SAA, F^*) \leq \sup_{F \in \mathcal{F}} \ratio_\numS(\SAA, F)  = \sup_{F \in \mathcal{F}}  \ratio_{\numS+1}(\SAA, F) - \delta  < \ratio_{\numS+1}(\SAA, F^*).
\end{equation*} 
This shows that the worst-case non-monotonicity of the relative regret implies the existence of an instance $F^*$ for which the relative regret is non-monotonic.

In \Cref{sec:optimal}, we explore some intuition underlying this shortcoming of SAA, but also characterize an optimal data-driven algorithm.


\section{Optimal Data-Driven Policy} \label{sec:optimal}

While SAA is a natural and widely used data-driven policy, we observed in \Cref{fig:exact_SAA} that the performance of SAA is not monotonically decreasing as a function of the number of samples $\numS$, implying that it is suboptimal from a minimax perspective.  Therefore a natural question is how to improve upon SAA and more generally if it is possible to characterize an optimal data-driven policy in the general space of mappings from data to decisions. Recall that we refer to the optimal policy as the one that solves the minimax optimization problem defined in \eqref{eq:minmax}.    In this section, we investigate the minimax relative regret  $\ratio^{*}_\numS$ presented in equation \eqref{eq:minmax}  and associated optimal policies.
Compared to solving the worst-case distribution for a particular algorithm, solving \eqref{eq:minmax} now involves two non-parametric and infinite dimensional optimization problems.

We approach the problem as follows. We first establish a fundamental reduction in the space of policies and show that one can restrict attention to mixture of order statistics policies (introduced in \Cref{def:OS}), without loss of optimality. In this class,  we leverage the structure of the problem \eqref{eq:ratio} for mixture of order statistics that we established in \Cref{sec:SAA} and we derive a necessary condition that a mixture of order statistics policy needs to satisfy to be optimal in this subclass. We then show that it is possible to construct a ``simple'' policy that satisfies this necessary condition. This policy is our candidate optimal policy. The worst-case performance of this policy naturally leads to an upper bound on  $\ratio_{\numS}^{*}$. To establish that this policy is actually optimal \textit{in the entire class of data-driven policies} $\Pi_\numS$, we introduce an alternative minimax problem which is equal to $\ratio_{\numS}^{*}$ and in which we extend the space of strategies that nature may take, to randomized ones. For this minimax problem, we construct a candidate prior over the space of distributions and show that the candidate policy above, together with the candidate prior, form a saddle point. This yields  the optimality of the candidate policy but also a characterization of its performance.

\subsection{Space Reduction from Arbitrary Mappings to Order Statistics} \label{sec:ub}
We first reduce the minimax optimization problem \eqref{eq:minmax} involving two non-parametric infinite dimensional optimization problems to a minimax problem over two finite dimensional spaces. 

Our next result shows that \eqref{eq:minmax} is equivalent to an optimization problem over the space of mixture of order statistic policies which has a much simpler structure than the general set of mappings from data to decisions.
More formally, we show the following.
\begin{theorem}
\label{thm:minimax_reduction}
For any $\numS \geq 1$,
\begin{equation*}
\inf_{\pi \in \Pi_\numS} \sup_{F \in \mathcal{F}} \ratio_{\numS} \left( \pi, F \right) =  \inf_{\mix \in \setOS} \sup_{F \in \mathcal{F}} \ratio_{\numS} \left( \mix, F \right).
\end{equation*}
\end{theorem}
\Cref{thm:minimax_reduction} enables a crucial space reduction of the policy space. In particular, it allows us to reduce our optimization problem to the space of mixture of order statistics policies which is parametrized by the $\numS$-dimensional vector of probabilities $\bm{\lambda}$. A notable step in the proof of the theorem consists in showing that,
\begin{equation}
\label{eq:MOS_opt_for_bern}
\inf_{\pi \in \Pi_\numS} \sup_{\mu \in [0,1]} \ratio_{\numS} \left( \pi, \mathcal{B} \left(\mu \right) \right) =  \inf_{\mix \in \setOS} \sup_{\mu \in [0,1]} \ratio_{\numS} \left( \mix, \mathcal{B} \left(\mu \right) \right).
\end{equation}
This equation complements \Cref{thm:main_mixture} which states that Bernoulli distributions are the worst-case distribution against mixture of order statistic policies. On the other hand, \eqref{eq:MOS_opt_for_bern} implies that mixture of order statistics policies are the best data-driven policies when facing a Bernoulli distribution. This result is established through a series of reductions, without loss of optimality,  in the space of policies. We first show that against Bernoulli distributions, one may restrict attention to policies prescribing inventory in the support of the historical demands. Second, we show that one may restrict attention to policies that prescribe identical inventory conditional on the number of ones observed. Third, we show that one may restrict attention  to policies that prescribe a monotonically increasing inventory as the number of ones observed
grows. We finally show that for any policy in the latter class, there exists a mixture of order statistics policy incurring a (weakly) lower cost.

Moreover, by leveraging the characterization of worst-case performance for mixture of order statistics policies derived in \Cref{thm:minimax_reduction}, we obtain that
 \begin{align*}
 \ratio_\numS^{*} =  \inf_{ \pi \in \Pi_{\numS} }  \sup_{F \in \mathcal{F}} \ratio_{\numS} \left( \pi, F\right) \stackrel{(a)}{=} \inf_{ \mix \in \Pi_\numS^{OS}}  \sup_{F \in \mathcal{F}} \ratio_{\numS} \left( \pi^{\bm{\lambda}}, F\right) \stackrel{(b)}{=} \inf_{ \mix \in \Pi_\numS^{OS}}  \sup_{\mu \in [0,1]} \ratio_{\numS} \left( \pi^{\bm{\lambda}}, \mathcal{B} \left(  \mu\right) \right).
 \end{align*}
where $(a)$ holds by \Cref{thm:minimax_reduction} and $(b)$ follows from \Cref{thm:main_mixture}.

This implies that Problem \eqref{eq:minmax} is equivalent to the following problem
\begin{equation}
 \inf_{ \mix \in \Pi_\numS^{OS}}  \sup_{\mu \in [0,1]} \ratio_{\numS} \left( \pi^{\bm{\lambda}}, \mathcal{B} \left(  \mu\right) \right) \label{eq:problem_ub},
\end{equation}
which now involves optimization over two finite dimensional spaces. 
In \Cref{sec:candidate_opt} we construct a candidate optimal policy for \eqref{eq:problem_ub}.
 
 \vspace{-2mm}
\subsection{Candidate Policy for Optimality}\label{sec:candidate_opt}
In general, prescribing a single order statistic policy can be suboptimal. However, there are particular cases in which extremal policies (either prescribing the minimum sample or the maximum one) achieve optimality.
 We first describe degenerate cases in which extremal order statistics are optimal.
\begin{proposition}
\label{prop:degen_opt}
For every $\numS$,
\begin{enumerate}
\item If $\sup_{\mu \in [0,1-q]} \ratio_{\numS} \left(\OS{1},\mathcal{B} \left( \mu \right) \right) > \sup_{\mu \in [1-q,1]} \ratio_{\numS} \left(\OS{1},\mathcal{B} \left( \mu \right) \right)$, then $\OS{1}$ is optimal for Problem \eqref{eq:minmax}.
\item If $\sup_{\mu \in [0,1-q]} \ratio_{\numS} \left(\OS{\numS},\mathcal{B} \left( \mu \right) \right) < \sup_{\mu \in [1-q,1]} \ratio_{\numS} \left(\OS{\numS},\mathcal{B} \left( \mu \right) \right)$, then $\OS{\numS}$ is optimal for Problem \eqref{eq:minmax}.
\end{enumerate}
\end{proposition}
This result implies that the optimal performance is obtained by extremal order statistics under some particular conditions. Note that these two conditions cannot hold simultaneously (we formally discuss this in \Cref{lem:single_cond} in \Cref{app:add_aux}). We highlight here that the conditions of \Cref {prop:degen_opt} do not hold for all data sizes. As a matter of fact,  we formally show in \Cref{lem:crossing}, stated and proved in \Cref{app:add_aux}, that these do not hold for any $\numS \geq \frac{2}{\min \left(q,1-q \right)^2}$. Next, we analyze the structure of optimal policies when the conditions do not hold. To that effect, we introduce the following assumption. 
\begin{assumption} \label{ass:non-deg} We say that a data size $\numS$ is non-degenerate if  the following two conditions on the performance of extremal order statistics policies hold
\begin{eqnarray}
\sup_{\mu \in [0,1-q]} \ratio_{\numS} \left(\OS{1},\mathcal{B} \left( \mu \right) \right) &\leq&\sup_{\mu \in [1-q,1]} \ratio_{\numS} \left(\OS{1},\mathcal{B} \left( \mu \right) \right) \label{eq:extreme_1} \\
\sup_{\mu \in [0,1-q]} \ratio_{\numS} \left(\OS{\numS},\mathcal{B} \left( \mu \right) \right) &\geq&\sup_{\mu \in [1-q,1]} \ratio_{\numS} \left(\OS{\numS},\mathcal{B} \left( \mu \right) \right). \label{eq:extreme_n}
\end{eqnarray}
\end{assumption}

In the case in which \Cref{ass:non-deg} holds,  one may benefit from randomization.  Next, we establish a necessary condition satisfied for a mixture of order statistics policies to  solve \eqref{eq:problem_ub}.
\begin{proposition}
\label{prop:necessary}
For every $\numS$ such that \Cref{ass:non-deg} holds,  for any solution $\pi^{\lambda} \in \Pi^{OS}$ that achieves the infimum in Problem (11), the solution $\pi^{\lambda}$ must satisfy
\begin{equation}
\sup_{\mu \in [0,1-q]} \ratio_{\numS} \left(\pi^{\bm{\lambda}},\mathcal{B} \left( \mu \right) \right) = \sup_{\mu \in [1-q,1]} \ratio_{\numS} \left(\pi^{\bm{\lambda}},\mathcal{B} \left( \mu \right) \right). \label{eq:necessary_balance}
\end{equation} 
\end{proposition}
\Cref{prop:necessary} establishes a necessary condition for a mixture of order statistics policy to be optimal for \eqref{eq:problem_ub}. 
In particular, $\pi^{\bm{\lambda}}$ must balance between worst-cases among Bernoulli distributions with mean smaller than $1-q$ and ones with mean larger than $1-q$. Intuitively, if the policy does not satisfy this property, it is possible to improve it by adding mass on lower or higher order statistics.

In \Cref{sec:non_monotone}, we observed that the worst-case performance of SAA is not monotonic as the number of samples grows and deduced its suboptimality. \Cref{prop:necessary} highlights why this is the case. One can show that SAA does not satisfy \eqref{eq:necessary_balance} in general, and by not doing so enables nature to exploit the imbalance in worst cases to ``hurt'' the decision-maker.  We present in \Cref{app:subopt_SAA} a more detailed discussion about the sub-optimality of SAA.

Our next result establishes that it is possible to construct a simple mixture of order statistics policy that satisfies  \eqref{eq:necessary_balance}, and randomizes between at most two consecutive order statistics.
\begin{proposition}
\label{prop:balancing_regret}
For every $\numS$ such that \Cref{ass:non-deg} holds, there exist $k \in \{2,\ldots,\numS\}$ and $\gamma \in [0,1]$ such that the policy $\pi^{k,\gamma}$ that prescribes the order statistic  $D_{k:\numS}$ w.p $ \gamma$ and $D_{k-1:\numS}$ w.p $1 - \gamma$ satisfies  \eqref{eq:necessary_balance} i.e., there exist $ \mu^{-} \in [0,1-q]$ and $\mu^{+} \in [1-q,1]$ such that,
\begin{equation}
\label{eq:indifference_body}
\ratio_n\left(\pi^{k,\gamma}, \mathcal{B}( \mu^{-}) \right) =  \ratio_n\left(\pi^{k,\gamma}, \mathcal{B}( \mu^{+}) \right) = \sup_{\mu \in [0,1] } \ratio_n\left(\pi^{k,\gamma}, \mathcal{B}( \mu) \right) .
\end{equation}
Moreover, $k$ satisfies
\begin{eqnarray}
\sup_{\mu \in [1- q,1] } \ratio_n\left(\OS{k-1}, \mathcal{B}(\mu) \right)& \geq & \sup_{\mu \in [0,1 - q]} \ratio_n\left(\OS{k-1}, \mathcal{B}(\mu) \right)\label{eq:crossing_1}
\\
\sup_{\mu \in [1- q,1] } \ratio_n\left(\OS{k}, \mathcal{B}(\mu) \right)& \leq & \sup_{\mu \in [0,1 - q]} \ratio_n\left(\OS{k}, \mathcal{B}(\mu) \right). \label{eq:crossing_2}
\end{eqnarray}
\end{proposition}
In other words, \Cref{prop:balancing_regret} intuitively characterizes the simplest candidate optimal mixture of order statistics policy one could consider when no single order statistic policy satisfies the necessary condition \eqref{eq:necessary_balance}.

This candidate policy alleviates the imbalance of the expected relative regret incurred by single order statistic policies. Indeed, letting $k$  denote  the largest order statistic prescribed by the candidate policy, we have by \eqref{eq:crossing_2} that the worst case performance of $\OS{k}$ on Bernoulli distributions with relatively small mean supersedes the one for Bernoulli distributions with large ones. On the contrary, according to \eqref{eq:crossing_1}, this imbalance is reverted for $\OS{k-1}$.

 Based on \Cref{prop:balancing_regret}, we have a candidate policy $\optpol$ satisfying a necessary condition for optimality for Problem \eqref{eq:problem_ub}.  This policy induces an upper bound on the value of \eqref{eq:problem_ub} as we have 
\begin{equation*}
 \inf_{ \mix \in \Pi_\numS^{OS}}  \sup_{\mu \in [0,1]} \ratio_{\numS} \left( \pi^{\bm{\lambda}}, \mathcal{B} \left(  \mu\right) \right) \leq \sup_{\mu \in [0,1]} \ratio_{\numS} \left( \optpol, \mathcal{B} \left(  \mu\right) \right).
\end{equation*}

In \Cref{sec:lower_bound}, we show that the candidate policy $\optpol$ not only satisfies a necessary optimality condition for order statistic policies, but is actually optimal in this space of policies which, by \Cref{thm:minimax_reduction}, implies its optimality in the general space of data-driven policies $\Pi_{\numS}$.

\vspace{-2mm}
\subsection{Optimal Data-Driven Policy and its Performance}\label{sec:lower_bound}
After deriving a candidate optimal policy, we now show that this policy is optimal for the initial Problem \eqref{eq:minmax} by proving its optimality for \eqref{eq:problem_ub}.
Remark that for $\numS \geq 1$, Problem \eqref{eq:problem_ub} is equivalent to the following problem in which we extend the space of Bernoulli distributions to the space of distributions over Bernoulli distributions
\begin{equation}
 \inf_{ \mix \in \Pi_\numS^{OS}}  \sup_{p \in \Delta \left([0,1] \right)} \mathbb{E}_{\mu \sim p} \left[ \ratio_{\numS} \left( \pi^{\bm{\lambda}}, \mathcal{B} \left(  \mu\right) \right) \right],\label{eq:opt_mixed_strategies_body}
\end{equation}
where $\Delta \left( [0,1] \right)$ is the set of distributions supported on $[0,1]$. Furthermore, we have
\begin{equation*}
 \inf_{ \mix \in \Pi_\numS^{OS}}  \sup_{p \in \Delta \left([0,1] \right)} \mathbb{E}_{\mu \sim p} \left[ \ratio_{\numS} \left( \pi^{\bm{\lambda}}, \mathcal{B} \left(  \mu\right) \right) \right]  \geq  \sup_{p \in \Delta \left([0,1] \right)}  \inf_{ \mix \in \Pi_\numS^{OS}}  \mathbb{E}_{\mu \sim p}\left[ \ratio_{\numS} \left( \pi^{\bm{\lambda}}, \mathcal{B} \left(  \mu\right) \right) \right]. 
\end{equation*}
To derive a lower bound matching the upper bound of \Cref{sec:ub}, it is sufficient to show that there exists a prior $p^{*}$, such that the policy  $\pi^{k,\gamma}$ introduced in \Cref{prop:balancing_regret} satisfies,
\begin{eqnarray}
\label{eq:sufficient}
\inf_{ \mix \in \Pi_{\numS}^{OS} } \mathbb{E}_{\mu \sim p^{*}}\left[\ratio_{\numS} \left( \mix, \mathcal{B} \left(\mu \right) \right) \right] &=& \mathbb{E}_{\mu \sim p^{*}}\left[\ratio_{\numS} \left( \pi^{k,\gamma}, \mathcal{B} \left(\mu \right) \right) \right], \label{eq:dual_sol} \\
\mathbb{E}_{\mu \sim p^{*}}\left[\ratio_{\numS} \left( \pi^{k,\gamma}, \mathcal{B} \left(\mu \right) \right) \right]  &=& \sup_{\mu \in [0,1]} \ratio_\numS \left( \pi^{k,\gamma}, \mathcal{B} \left( \mu \right) \right). \label{eq:no_gain_info}
\end{eqnarray}

Equality \eqref{eq:dual_sol} would imply that the policy $\pi^{k,\gamma}$ presented in \Cref{prop:balancing_regret} is the best response when Nature selects prior $p^*$. Equality \eqref{eq:no_gain_info} would ensure that prior $p^{*}$ leads to the worst-case performance of $\pi^{k,\gamma}$.

Consider $ \mu^{-} \in [0,1-q]$ and $\mu^{+} \in [1-q,1]$ as introduced in \Cref{prop:balancing_regret}. 
Note that \eqref{eq:indifference_body} implies that for any prior $p_0$ supported on $\{ \mu^{-}, \mu^{+}\}$, we have
 $$
\mathbb{E}_{\mu \sim p_0}\left[\ratio_{\numS} \left( \pi^{k,\gamma}, \mathcal{B} \left( \mu \right)\right) \right] = \sup_{\mu \in [0,1]} \ratio_\numS \left( \pi^{k,\gamma}, \mathcal{B} \left( \mu \right) \right).
 $$
 It follows that \eqref{eq:no_gain_info} holds for any such prior. This motivates restricting attention to the set of priors supported on two Bernoulli distributions.
 Our next result shows that in the class of priors over two Bernoulli distributions, there exists a prior for which \eqref{eq:sufficient} holds. Formally we show the following.
\begin{proposition}
\label{prop:prior}
For any $k \in \{2,\ldots, \numS \}$, $\gamma \in [0,1]$, $\mu^{-} \in (0,1-q)$ and $\mu^{+} \in (1-q,1)$, there exists a prior $p^{*}$ on $\{ \mu^{-}, \mu^{+} \}$ such that,
\begin{equation*}
\inf_{ \mix \in \Pi_{\numS}^{OS} } \mathbb{E}_{\mu \sim p^{*}}\left[\ratio_{\numS} \left( \mix, \mathcal{B} \left(\mu \right) \right) \right] = \mathbb{E}_{\mu \sim p^{*}}\left[\ratio_{\numS} \left( \pi^{k,\gamma}, \mathcal{B} \left(\mu \right) \right) \right] .
\end{equation*}
\end{proposition}

We are now in a position to state our next main result. The next result provides a characterization of an optimal policy and its performance. In particular, we build on \Cref{prop:prior} and on the upper bound derived in  \Cref{sec:ub}  to establish that, when \Cref{ass:non-deg} holds, an optimal data-driven policy, in the \textit{entire} space of possible mappings from data to decision, is given by a randomization over at most two consecutive order statistics of the historical demand samples (in the case where one of the conditions does not hold, we have already established that an extremal order statistic is optimal). Formally we show the following.
\vspace{-1mm}
\begin{theorem}[Optimal Data-Driven Policy]
\label{thm:optimal_min_max}
For every $\numS$ such that \Cref{ass:non-deg} holds, there exists $k \in \{2,\ldots,\numS\}$ and $\gamma \in [0,1]$ such that the policy $\pi^{k,\gamma}$ that prescribes the order statistic  $D_{k:\numS}$ w.p $ \gamma$ and $D_{k-1:\numS}$ w.p $1 - \gamma$  satisfies
\begin{equation*}
\sup_{F \in \mathcal{F}} \ratio_n\left( \pi^{k,\gamma}, F \right) = \ratio_{\numS}^{*}.
\end{equation*}
 Moreover, $k$ satisfies \eqref{eq:crossing_1} and \eqref{eq:crossing_2}.

Furthermore, if \eqref{eq:extreme_1} does not hold, $\OS{1}$ is optimal for Problem \eqref{eq:minmax}. Similarly, if \eqref{eq:extreme_n} does not hold, $\OS{\numS}$ is optimal for Problem \eqref{eq:minmax}.
\end{theorem}
This result provides a full characterization of an \textit{optimal} data-driven policy across data sizes. Notably, $i.)$ an optimal policy and associated optimal performance can be characterized for this class of problems; and $ii.)$ the optimal data-driven policy takes a surprisingly simple structure: it randomizes between two consecutive order statistics. This result allows not only to obtain an optimal algorithm but also to \textit{quantify exactly the robust value of data} associated with historical demand for this class of problems.

\paragraph{Remark. (A ``better'' minimax optimal policy)} A corollary of \Cref{thm:optimal_min_max} is that the \textit{deterministic} policy which selects the inventory level equal to $(1-\gamma) D_{k-1:\numS} + \gamma D_{k:\numS}$ is also minimax optimal (where $k$ and $\gamma$ are the parameters defined in \Cref{thm:optimal_min_max}). In addition, this policy is uniformly better (across all instances) than the minimax optimal mixture of order statistics policy, and its performance coincides with the latter against Bernoulli distributions on which it yields the same worst-case relative regret. We  formalize this in \Cref{cor:derand} below.
\begin{corollary}
\label{cor:derand}
For every $\numS \geq 1$. 
Let $\optpol$ be the minimax optimal policy defined in \Cref{thm:optimal_min_max} and let $\derpol$ be the policy which prescribes the inventory level $\gamma D_{k-1:\numS} + (1-\gamma) D_{k:\numS}$. Then,
\begin{equation*}
\sup_{F \in \mathcal{F}} \ratio_n\left( \optpol, F \right) = \sup_{F \in \mathcal{F}} \ratio_n\left( \derpol, F \right) = \ratio_{\numS}^{*}.
\end{equation*}
Furthermore, for every $F \in \mathcal{F}$,
\begin{equation*}
\ratio_n\left( \derpol , F \right) \leq \ratio_n\left( \optpol, F \right).
\end{equation*}
\end{corollary}

\paragraph{Remark.}We also observe that a byproduct of our analysis has implications for the value of data in the Bayesian newsvendor problem.  Our analysis shows that there is no gap between the frequentist problem \eqref{eq:minmax} and its bayesian counterpart in the sense that, against the worst prior (which is a randomization between two Bernoulli distributions), the Bayesian problem is as hard (in the sense of the value of data) as the frequentist one and achieves the same worst-case relative regret.

\vspace{-3mm}

\subsection{Optimal Performance and the Robust Value of Data} \label{sec:opt-VI}
\Cref{alg:optimal} (presented in \Cref{apx:implement_alg}) enables us to compute the performance of the optimal policy defined in \Cref{thm:optimal_min_max}. \Cref{fig:opt_vs_SAA} presents a comparison of the performance of SAA and the best achievable performance for a data-driven policy for different critical fractiles.  
\begin{figure}[h!]
\centering
\subfigure[$q = .8$]
{
\begin{tikzpicture}[scale=.75]
\begin{axis}[
            title={},
            xmin=0,xmax=100,
            ymin=0.0,ymax=2.2,
            width=10cm,
            height=8cm,
            table/col sep=comma,
            xlabel = number of samples $\numS$,
            ylabel = relative regret,
            grid=both,
            skip coords between index={0}{1},
            legend pos=north east]
      
\addplot [red,mark=square,mark options={scale=.5}] table[x=n,y=SAA] {Data/SAA_and_opt_paper_08.csv};
    \addlegendentry{$\ratio_\numS( \SAA)$}
\addplot [blue,mark=*,mark options={scale=.5}] table[x=n,y=Opt] {Data/SAA_and_opt_paper_08.csv};
    \addlegendentry{$\ratio_\numS^*$}
\end{axis}
\end{tikzpicture}
}
\subfigure[$q = .9$]
{
\begin{tikzpicture}[scale=.75]
\begin{axis}[
            title={},
            xmin=0,xmax=100,
            ymin=0.0,ymax=2.2,
            width=10cm,
            height=8cm,
            table/col sep=comma,
            xlabel = number of samples $\numS$,
            ylabel = relative regret,
            grid=both,
            skip coords between index={0}{1},
            legend pos=north east]
                    
\addplot [red,mark=square,mark options={scale=.5}] table[x=n,y=SAA] {Data/SAA_and_opt_paper_09.csv};
    \addlegendentry{$\ratio_\numS( \SAA)$}
\addplot [blue,mark=*,mark options={scale=.5}] table[x=n,y=Opt] {Data/SAA_and_opt_paper_09.csv};
    \addlegendentry{$\ratio_\numS^*$}
\end{axis}
\end{tikzpicture}
}
\caption{\textbf{Optimal performance.} The figure depicts optimal performance versus the performance of SAA as a function of the number of samples $\numS$ for  different critical fractiles. }
\label{fig:opt_vs_SAA}
\end{figure}
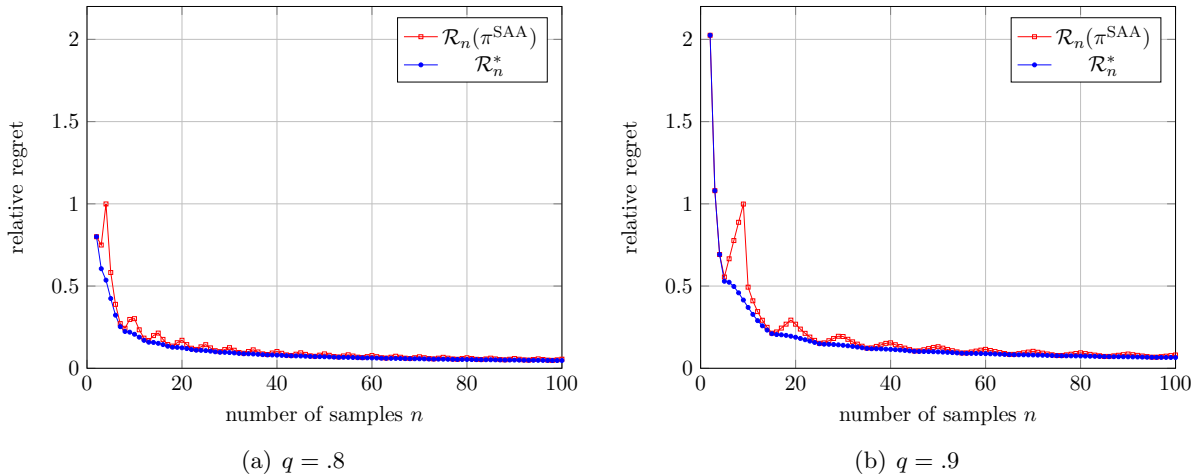

In this plot, the curve associated to the optimal policy describes the exact value of historical demand data in the newsvendor problem. It gives a clear sense of the inherent hardness of this class of data-driven problems.
Deriving the full spectrum of performances  for both SAA and the optimal data-driven policy shows that SAA can be considerably improved when the number of samples is relatively small. In particular, when $q = .9$, the relative regret for SAA at $\numS = 9$ can be reduced by more than 50\% by using the optimal policy, and for $\numS = 19$, it can be reduced by $33\%$. 
We also remark that the performance of SAA matches more closely the optimal one as $\numS$ becomes large. The amplitudes of the peaks decrease as the number of samples increases. 
We further explore the asymptotic performance in \Cref{sec:asymptotic}.

In \Cref{tab:comparison_SAA_opt}, we present a comparison of the number of samples required to guarantee various levels of relative regret for different values of the critical fractile for both SAA and the optimal policy. Recall the definition of $N^{\text{exact-SAA}}$ presented in \Cref{sec:comp}. We similarly define, the number of samples required to achieve a given performance threshold $\tau \geq 0$ when using the optimal policy, as
\begin{eqnarray*}
N^{\text{opt}} (\tau) &:=& \min \left\{ m \geq 1 \, \Big \vert \, \forall \numS \geq m, \: \ratio_{\numS}^* \leq \tau   \right\}.
\end{eqnarray*}

\begin{table}[h!]
\centering
\begin{tabular}{c l   c c c c c }
&& \multicolumn{5}{c}{Expected relative regret target ($\tau$)}\\
\cline{3-7}
q&Bound used & 25\% & 20\% & 15\% & 10\% & 5\%  \\ 
 \hline
 \hline
 \vspace{-4mm}
 \\
 \multirow{2}{*}{.7}&  $N^{\text{exact-SAA}}(\tau)$ &  8&  11&  15& 31&  84  \\
 &$N^{\text{opt}}(\tau) $   &  5&  8&  12& 21& 68\\
\hline
 \vspace{-4mm}
 \\
 \multirow{2}{*}{.8} & $N^{\text{exact-SAA}}(\tau)$ &  11&  16&  21& 41&  116  \\
& $N^{\text{opt}}(\tau)$   &  8&  11&  16& 28& 91\\
\hline
 \vspace{-4mm}
 \\
\multirow{2}{*}{.9} & $N^{\text{exact-SAA}}(\tau)$ &  21&  23&  42& 71&  210  \\
& $N^{\text{opt}}(\tau)$    &  14&  19&  25& 50& 161\\
\end{tabular}
\caption{\textbf{Number of samples required by SAA and by the optimal policy to ensure a target relative regret.}  The table reports the \textit{exact} number of samples needed to reach a relative regret accuracy level, comparing the exact worst-case analysis of SAA developed in \Cref{thm:SAA}, to the optimal minimax performance presented in \Cref{thm:optimal_min_max} for different values of the critical fractile $q$.}
\label{tab:comparison_SAA_opt}
\end{table}
We observe that the number of samples required to ensure a particular level of accuracy across all distributions can be reduced by $17$ to $40$ \% (across the targets tested) when moving from SAA to the minimax optimal policy.

\paragraph{Remark (Structure of the optimal policy).}  \label{pg:R1-k}
While \Cref{thm:optimal_min_max} does not provide an exact characterization of the parameter $k$, we have observed numerically that $k$ is either equal to $\ceil{q \numS}$ or $\ceil{q \numS} +1$. We discuss in more details this aspect in \Cref{apx:values_k}. As a consequence, the optimal policy can be interpreted as a  correction of SAA. 


\section{Asymptotic Analysis of Optimal Performance} \label{sec:asymptotic}

We derived in \Cref{sec:optimal} a characterization of an optimal data-driven policy for an arbitrary finite number of samples.  We now provide a simple approximation of the optimal performance as the number of samples grows large and derive the exact convergence rate of the minimax relative regret to $0$ with its associated  multiplicative constant.

In this section, as we study what happens when $\numS$ changes, we will introduce the notion of a policy sequence. A policy sequence is  defined as a sequence $\bm{\pi} := (\pi_{\numS})_{\numS \geq 1}$ of mappings where for every $n \geq 1$, we have $\pi_{\numS} \in \Pi_{\numS}$. For example,   $\bm{\pi}^{\text{SAA}}$ denotes the sequence of policies such that for any $\numS \geq 1$,
\begin{equation*}
\pi^{\text{SAA}}_{\numS} (\mathbf{D}_{1}^\numS) = D_{\ceil{q \numS}:\numS}.
\end{equation*}

There are three types of asymptotic results one could consider. A first characterization of the performance, which is typically referred to as consistency or first order optimality states that the cost of a data-driven policy converges to the optimal cost as the number of samples goes to infinity.
In our setting,  a policy sequence $\bm{\pi}$ is said to be consistent if
$$ \sup_{F \in \mathcal{F}} \ratio_\numS \left(\bm{\pi}, F \right) \to 0 \qquad \text{as $\numS\to \infty$}.$$ 

A second, more refined characterization consists in establishing the rate of convergence of the worst-case performance of a data-driven policy. In the data-driven newsvendor model, for a given sequence $(u_\numS)_{\numS \in \mathbb{N}}$ converging to $0$, we say that the cost of a data-driven policy-sequence $\bm{\pi}$ converges to zero at  rate $u_\numS$ if
$$ \sup_{F \in \mathcal{F}} \ratio_\numS \left( \bm{\pi}, F \right) = \mathcal{O} \left( u_\numS \right) \qquad \text{as $\numS\to \infty$}.$$
At the rate level, the best result one can aim at is to prove that a policy-sequence converges at a rate of $\ratio^*_\numS$, in which case we say that the policy achieves rate-optimality.

A third, yet more refined characterization enables a sharper understanding of the asymptotic performance of a data-driven policy. It consists in deriving a sequence equivalent to the relative regret as the number of samples goes large. In particular, for a given sequence $(u_\numS)_{\numS \in \mathbb{N}}$ we say that the performance of a data-driven policy-sequence $\bm{\pi}$ is asymptotically equivalent to $u_\numS$ if,
$$ \sup_{F \in \mathcal{F}} \ratio_\numS \left( \bm{\pi}, F \right) = u_\numS + o \left( u_\numS \right) \qquad \text{as $\numS\to \infty$}.$$
Deriving an equivalent sequence is a much stronger result than  the rate of convergence as it requires to characterize the convergence rate \textit{as well as} the multiplicative constant associated with the rate. When a policy-sequence has a performance asymptotically equivalent to $\ratio_\numS^*$, we say that it is rate-optimal at the multiplicative constant level.

From the work of \cite{levi2} one may derive consistency results  and the rate of convergence for $\sup_{F \in \mathcal{F}} \ratio_{\numS}(\SAA,F)$. In particular, we show in \Cref{lem:levi_bound}, stated and proved in \Cref{app:add_aux}, that their bound implies that $\sup_{F \in \mathcal{F}} \ratio_{\numS}(\SAA,F)$ scales at a $\mathcal{O} \left( 1/\sqrt{\numS} \right)$ rate.

The next result characterizes the asymptotic equivalent of the relative regret for the optimal data-driven policy and establishes that SAA is not only rate-optimal but also rate-optimal at the multiplicative constant level. 
\begin{theorem}[Optimal Asymptotic Behavior] 
\label{thm:asymptotic}
\begin{enumerate}[i.)]

\item The optimal performance $\ratio^{*}_{\numS}$ converges to zero and satisfies
\begin{equation}
\label{eq:lb_asymptotic}
\ratio^{*}_{\numS} = \frac{C^*}{\sqrt{\numS}} + o \left( \frac{1}{\sqrt{\numS}} \right) \qquad \text{as $\numS\to \infty,$}
\end{equation}
where 
\begin{equation*}
C^* := \frac{1}{\sqrt{q(1-q)}} \max_{p \geq 0} p \left(1 - \Phi(p) \right) \approx \frac{.17}{\sqrt{q(1-q)}},
\end{equation*}
with $\Phi$ denoting the cdf of a standard Gaussian distribution.
\item In addition, the policy sequence associated with SAA is rate-optimal at the multiplicative constant level. In particular,
\begin{equation*}
\sup_{F \in \mathcal{F}} \ratio_{\numS} \left( \pi^{\text{SAA}}_{\numS}, F \right) =  \frac{C^*}{\sqrt{\numS}} + o \left( \frac{1}{\sqrt{\numS}} \right) \qquad \text{as $\numS\to \infty$}.
\end{equation*}
\end{enumerate}
\end{theorem}
This result describes the exact rate of convergence of the optimal relative regret as the number of samples goes to infinity. 
While Sections \ref{sec:SAA} and \ref{sec:optimal} yield the first exact results for arbitrary sample sizes, the significant novelty in this section lies in explicitly characterizing more finely the rate of convergence of the performance of the optimal data-driven policy as data grows. Indeed, we derive a semi-closed form expression of  the exact constant $C^*$ associated with the rate of convergence of the optimal policy for this class of problems. This expression highlights the role of the critical fractile $q$ in affecting optimal relative regret performance. Problems with high and low values of $q$ are ``harder'' in that they lead to higher constant $C^*$, and in turn slower convergence to zero. 

 In addition, we are able to establish that, while SAA was suboptimal for finite samples in general,  it satisfies a very strong form of near-optimality when the number of samples is large. While SAA leads to relative regret that converges to zero at rate  $\mathcal{O} \left(1/\sqrt{\numS} \right)$, it also  leads to the \textit{optimal} constant that one could achieve at this rate of convergence. 

By leveraging our novel analysis across all data sizes, we derive new insights in the asymptotic regime. Therefore, understanding more finely the performance of data-driven policies with finite data also improves our understanding of their performance as the number of samples goes to infinity.


\section{ Instance-Dependent Performance}
\label{sec:experiment}

Our approach enables us to develop a sharp understanding of the robust performance of SAA and of a minimax optimal policy for the data-driven newsvendor problem. Our analysis quantifies exactly the worst-case performance of these algorithms when the worst-case is taken over the whole class of distributions with finite first moment without any shape restriction. In this section, to illustrate the range of possible performances that could emerge, we compute the empirical performances of both algorithms against various distributions: Uniform,  Exponential, Lognormal and Pareto. Note that these distribution families are the ones used in \cite[Table 1]{levi2}  which are supported on $[0,\infty)$. 

We compute numerically the expected regret of a data-driven policy $\pi$ that uses $\numS$ samples against a \emph{fixed} distribution by repeating independently $M = 10^5$ times the following procedure. For every $m \in \{1,\ldots,M\}$, we first draw  $\numS$ independent samples  $\{d^m_1, \ldots, d^m_\numS\}$ representing in-sample demand realizations. We then draw independently $\{\tilde{d}^m_1, \ldots, \tilde{d}^m_K \}$ (where $K = 1000$) samples to compute the out-of-sample cost.  We finally  draw a decision realization $\tilde{x}^m$ from the distribution $\pi(  d_1^m, \ldots, d_\numS^m)$ and compute the average realized cost
$\tilde{c}^m = \frac{1}{K} \sum_{k=1}^K c(\tilde{x}^m, \tilde{d}^m_k).$
Our estimator of the expected relative regret for policy $\pi$ is finally defined as
\begin{equation*}
\frac{1}{M} \sum_{m=1}^M  \left( \frac{\tilde{c}^m }{\opt(F)} -1 \right).
\end{equation*}

\Cref{tab:numerical_comparison} presents the number of samples\footnote{We report the minimum number of samples such that the upper bound of the $95\%$ confidence interval is below the desired relative regret target.} required to achieve a target accuracy level for both SAA and the minimax optimal policy presented in \Cref{cor:derand}. 

\begin{table}[h!]
\centering
\begin{tabular}{c l   c c c c c }
&& \multicolumn{5}{c}{Expected relative regret target ($\tau$)}\\
\cline{3-7}
Policy& Distribution & 25\% & 20\% & 15\% & 10\% & 5\%  \\ 
 \hline
 \hline
 \vspace{-4mm}
 \\
 \multirow{5}{*}{SAA}& Worst-case (Bernoulli)&  21&  23&  42& 71&  210   \\
 & Uniform(0,1)   &  6&  11&  12& 14& 25\\
 & Exponential(1)   &  7&  10&  13& 20& 40\\
 & Log-normal($\mu=1,\sigma=1.805$)   &  10&  10&  10& 20& 40\\
 & Pareto($\alpha =1.5$, $x_m =1$)   &  10&  16&  16&20& 93\\
\hline
 \vspace{-4mm}
 \\
 \multirow{5}{*}{Minimax Optimal  Policy}  & Worst-case (Bernoulli)&    14&  19&  25& 50& 161   \\
 & Uniform(0,1)   &  6&  7&  10& 13& 22\\
 & Exponential(1)   &  6&  8&  10& 18& 37\\
 & Log-normal($\mu=1,\sigma=1.805$)   &  9&  11&  14& 19& 36\\
 & Pareto($\alpha =1.5$, $x_m =1$)   &  10&  16&  16&18& 93\\
\end{tabular}
\caption{\textbf{Number of samples required by SAA and by the minimax optimal policy $\derpol$ (cf. \Cref{cor:derand}) to achieve a target relative regret.}  The table reports a numerical estimation of the number of samples needed to reach a relative regret accuracy level against several fixed distributions. The worst-case line indicates the \emph{exact} number of samples required to achieve a certain target performance level.}
\label{tab:numerical_comparison}
\end{table}

While the minimax policy is optimized relatively to a rather conservative measure, it is notable that its performance is  on par or most often better than the one of SAA even in ``mild'' cases. In other words, the ``robustification'' of SAA  provides significant benefits in the worst case along with improvements against a variety of ``mild'' distributions.


\section{Conclusion}

In this paper, we investigate the central class of data-driven newsvendor problems. 
We analyze the performance of the central SAA algorithm across all data sizes and establish a characterization of its \textit{actual} worst-case performance. The exact performance characterization of this widely studied policy leads to a new understanding of the economics of data sizes, highlighting the very strong performance achievable with limited data. At the same time, it also demonstrates a notable phenomenon:  when using SAA, more data is not synonymous with better worst-case performance. 

 In turn, we optimize over the entire space of data-driven algorithms that maps data to decisions and derive an optimal algorithm (in the minimax sense) and its associated performance. This provides the first optimality result in this class of data-driven problems. It also perfectly quantifies the value of data and the potential associated with corrections to the classical SAA algorithm, especially with smaller data sizes. It further emphasizes that for this class of problems, a decision-maker may operate efficiently even in environments with limited data.

Finally, we provide a simple approximation of the optimal worst-case performance achievable by a data-driven algorithm when the number of samples is large. In particular, we leverage our exact analysis across all data sizes to characterize the exact rate of convergence of the minimax relative regret and characterize in semi-closed form the multiplicative constant associated with it. We further show that while SAA is suboptimal in general, it is rate-optimal at the multiplicative constant level when the number of samples is large.

The present paper offers a new lens, that of the transient regime of learning, through which some data-driven problems may be approached, but also highlights the possibility to operate effectively with limited data. There are many avenues for future research, ranging from exploring the possibility of performance characterization and optimization across data sizes for sequential decision-making problems with different information structures (e.g., censoring)  to exploring the transient regime of learning in contextual newsvendor problems, or more general  stochastic problem classes.

\section*{Acknowledgment}
The authors are grateful to the editor, the associate editor and two anonymous reviewers whose valuable suggestions lead to many significant improvements in the final version of the paper. They also thank Nick Arnosti, Santiago Balseiro, Gah-Yi Ban, Omar El Housni, Yale T. Herer, Nathan Kallus, Will Ma, and Dan Russo for their questions and comments which helped improve this work.

{\setstretch{0.8}
\bibliographystyle{agsm}
\bibliography{ref}}

\renewcommand{\theequation}{\thesection-\arabic{equation}}
\renewcommand{\theproposition}{\thesection-\arabic{proposition}}
\renewcommand{\thelemma}{\thesection-\arabic{lemma}}
\renewcommand{\thetheorem}{\thesection-\arabic{theorem}}

\appendix 

\setstretch{1.15}

\newpage

\pagenumbering{arabic}
\renewcommand{\thepage}{App-\arabic{page}}

\begin{center}
 {\Large Electronic Companion: 
 Appendix for \\ 
 \textbf{How Big Should Your Data Really Be?\\ Data-Driven Newsvendor: Learning One Sample at a Time\\
}}
\end{center}

\setcounter{equation}{0}
\setcounter{proposition}{0}
\setcounter{lemma}{0}
\setcounter{theorem}{0}

\section{Proofs for Section \ref{sec:SAA}} \label{apx:SAA}

\begin{proof}[\textbf{Proof of \Cref{lem:quantile_policy_cost}.}]
Fix $F \in \mathcal{F}$.  
For every $r \in \{1, \ldots, \numS\}$, let $F_{D_{r:\numS}}$ denote the distribution of the random variable $D_{r:\numS}$. We will use the following alternative expression for $c_F(x)$. 
\begin{lemma}\label{lem:cost}
For any distribution $F \in \mathcal{F}$, and any $x \geq 0,$ 
\begin{equation*}
c_F(x) = b(\mathbb{E}_F[D] - x) + (b+h) \int_{0}^x F(y)dy.
\end{equation*}
\end{lemma}
This result is proved in \Cref{apx:sec_body}. In what follows, we use $\bar{F}$ to denote the complementary cumulative distribution, i.e., $\bar{F} = 1 - F$.

We have
\begin{align*}
\mathbb{E}_{ x \sim F_{D_{r:\numS} }} [c_F(x)] &\stackrel{(a)}{=} b (\mathbb{E}_F[D] - \mathbb{E}_{F_{D_{r:\numS}} }[D_{r:\numS}] ) + (b+h) \int_0^{\infty} \int_0^s F(y)dy dF_{D_{r:\numS}}(s)\\
& \stackrel{(b)}{=} b (\mathbb{E}_F[D] - \mathbb{E}_{F_{D_{r:\numS}} }[D_{r:\numS}] ) + (b+h) \int_0^{\infty} \int_y^{\infty} dF_{D_{r:\numS}}(s) F(y)dy\\
& =  b (\mathbb{E}_F[D] - \mathbb{E}_{F_{D_{r:\numS}} }[D_{r:\numS}] ) + (b+h)\int_0^{\infty} \bar{F}_{D_{r:\numS}}(y)F(y)dy\\
& = b \left (\int_0^{\infty} \bar{F}(y) dy  - \int_0^\infty \bar{F}_{D_{r:\numS}}(y) dy \right) + (b+h)\int_0^{\infty} \bar{F}_{D_{r:\numS}}(y)F(y)dy\\
& = (b+h) \left [ q \left (\int_0^{\infty} \bar{F}(y) dy  - \int_0^{\infty} \bar{F}_{D_{r:\numS}}(y) dy \right) + \int_0^{\infty} \bar{F}_{D_{r:\numS}}(y)F(y)dy \right] \\
& = (b+h) \left [ \int_0^{\infty} \left( \bar{F}_{D_{r:\numS}}(y)(F(y)-q)  + q(1-F(y)) \right)dy \right] \\
& \stackrel{(c)}{=} (b+h) \left [ \int_0^{\infty} \left( (1- B_{r,\numS} (F(y)) ) (F(y)-q)  + q(1-F(y)) \right)dy \right].
\end{align*}
Here, $(a)$ follows from \Cref{lem:cost}. Equality $(b)$ follows from Fubini-Tonelli which holds because, $s \mapsto 1$ is a positive function and $(\mathbb{R},dF_{D_{r:\numS}})$ and $(\mathbb{R},dx)$ are complete, $\sigma$-finite measure spaces. 
Moreover, $(c)$ holds because the cumulative distribution function of  $D_{r:\numS}$ satisfies
\begin{equation*}
F_{D_{r:\numS}}(x) = B_{r,\numS} (F(x)).
\end{equation*}
Therefore, one can derive the desired expression by decomposing the performance of $\mix$ as follows.
\begin{align*}
\mathcal{C} \left( \mix, F, \numS \right) &= \mathbb{E}_{\mathbf{D}_1^\numS \sim F}\left [ \mathbb{E}_{x \sim {\mix \left( \mathbf{D}_{1}^\numS \right)} } \left[  c_F(x) \right] \right] \stackrel{(a)}{=} \sum_{i=1}^{\numS} \lambda_i \mathbb{E}_{\mathbf{D}_1^\numS \sim F}\left [  c_F(D_{i:\numS})  \right]  = \sum_{i=1}^{\numS} \lambda_i \mathbb{E}_{ x \sim F_{D_{i:\numS} }} [c_F(x)],
\end{align*}
where $(a)$ follows from law of total expectation conditioning on the value of $\mix \left( \textbf{D}_1^{\numS} \right)$.

Next, we analyze  $\opt(F)$. Using \Cref{lem:cost} we can rewrite the optimal cost as 
\begin{align*}
\opt(F) &= c_F(x^*_F)\\
&= b(\mathbb{E}_F[D] - x^*_F) + (b+h) \int_{0}^{x^*_F} F(y)dy\\
&= b \left( \int_0^{x^*_F} \bar{F}(y)dy + \int_{x^*_F}^\infty \bar{F}(y)dy  - x^*_F \right) + (b+h) \int_{0}^{x^*_F} F(y)dy\\
&= b \left(\int_{x^*_F}^\infty \bar{F}(y)dy -  \int_0^{x^*_F} F(y)dy \right) + (b+h) \int_{0}^{x^*_F} F(y)dy\\
&= b\int_{x^*_F}^\infty \bar{F}(y)dy + h \int_{0}^{x^*_F} F(y)dy\\
&=(b+h) \left[ q\int_{x^*_F}^\infty \bar{F}(y)dy + (1-q) \int_{0}^{x^*_F} F(y)dy \right]\\
&= (b+h) \int_0^{\infty}   (1-q)F(y) \mathbbm{1}\{y < x^*_F\} + q(1-F(y)) \mathbbm{1}\{y \geq x^*_F\} dy\\
&\stackrel{(a)}{=} (b+h) \int_0^{\infty}   (1-q)F(y) \mathbbm{1}\{F(y) < q\} + q(1-F(y)) \mathbbm{1}\{F(y) \geq q\} dy\\
&= (b+h) \int_0^{\infty}   \min\{(1-q)F(y) , q(1-F(y)) \} dy.
\end{align*}
$(a)$ holds by definition of $x^*_F$.
\end{proof}

\begin{proof}[\textbf{Proof of \Cref{thm:main_mixture}}.]

\textbf{Step 1.} For any mixture of order statistics policy $\pi^{\bm{\lambda}}$, by plugging the simplified expressions of $\mathcal{C} \left( \mix, F,\numS \right)$ and $\opt(F)$ computed in \Cref{lem:quantile_policy_cost} in the epigraph formulation derived in \Cref{lem:epigraph_form}, we obtain that problem \eqref{eq:ratio} is equivalent to,
\begin{subequations}
\label{eq:whole_class_pb_2}
\begin{alignat}{2}
 &\! \inf_{z \in \mathbb{R}  }        &\qquad& z \\
&\text{s.t.} &      & \sup_{F \in \cal{F}} \int_0^\infty \Psi_z^{\bm{ \lambda}}(F(y)) dy \leq 0.
\end{alignat}
\end{subequations}
where $\Psi_z^{\bm{ \lambda}}: [0,1] \rightarrow \mathbb{R}$ is such that for every $x \in [0,1]$,
\begin{equation*}
\Psi_z^{\bm{ \lambda}}(x) =  \sum_{i=1}^{\numS} \lambda_i \left[ (1- B_{i,\numS} (x) ) (x-q)  + q(1-x) - (z+1)\min\{ (1-q)x, q(1-x)\} \right].
\end{equation*}

\textbf{Step 2.} We next aim to further simplify Problem \eqref{eq:whole_class_pb_2}. To that end, we establish the following equivalence
\begin{equation}
\label{eq:space_reduction}
\sup_{F \in \cal{F}} \int_0^\infty \Psi_z^{\bm{ \lambda}}(F(y)) dy \leq 0  \quad \mbox{if and only if } \quad   \sup_{\alpha \in (0,1)}\Psi_z^{\bm{ \lambda}}(\alpha) \leq 0 .
\end{equation}
First assume that $\sup_{\alpha \in (0,1)} \Psi_z^{\bm{ \lambda}}(\alpha) \leq 0$. Noting that $\Psi_z^{\bm{ \lambda}}(\cdot)$ is continuous on $[0,1]$, we also have  $\sup_{\alpha \in [0,1]} \Psi_z^{\bm{ \lambda}}(\alpha) \leq 0$. In such a case, for all $F \in \mathcal{F}$, since $F(y) \in[0,1]$, we have that $\Psi_z^{\bm{ \lambda}}(F(y)) \leq 0 $ for all $y \ge0$ and it directly follows that
\begin{equation*}
\int_0^\infty \Psi_z^{\bm{ \lambda}}(F(y)) dy \leq 0.
\end{equation*}
Conversely, suppose that $\sup_{F \in \cal{F}} \int_0^\infty \Psi_z^{\bm{ \lambda}}(F(y)) dy \leq 0$. Note that for any $z \in \mathbb{R}$, $\Psi_z^{\bm{ \lambda}}(\cdot)$ is continuous on $[0,1]$ and therefore, it achieves its maximum on $[0,1]$. Let $\alpha^* \in \argmax_{\alpha \in [0,1]} \Psi_z^{\bm{ \lambda}}(\alpha)$.
Let $G$ be defined by, 
\begin{equation*}
G(x) = \begin{cases}
0 \qquad \text{if $x < 0$}\\
\alpha^* \qquad \text{if $x \in [0,1)$}\\
1 \qquad \text{if $x \ge 1$}
\end{cases}
\end{equation*}
In turn we have
\begin{equation*}
\sup_{\alpha \in (0,1)} \Psi_z^{\bm{ \lambda}}(\alpha) =  \sup_{\alpha \in [0,1]} \Psi_z^{\bm{ \lambda}}(\alpha) = \Psi_z^{\bm{ \lambda}}(\alpha^*) = \int_0^1 \Psi_z^{\bm{ \lambda}}(G(y))dy \stackrel{(a)}{\leq} \sup_{F \in \cal{F}} \int_0^\infty \Psi_z^{\bm{ \lambda}}(F(y)) dy \leq 0,
\end{equation*}
where $(a)$ holds because $G \in \mathcal{F}$. As a consequence, \eqref{eq:space_reduction}  holds.

Furthermore, note that \eqref{eq:space_reduction} implies that problem \eqref{eq:whole_class_pb_2} is equivalent to,
\begin{subequations}\label{eq:whole_class_pb_3}
\begin{alignat}{2}
&\! \inf_{z \in \mathbb{R} }        &\quad& z \\
&\text{s.t.} &      &  \sup_{\alpha \in (0,1)} \sum_{i=1}^{\numS}  \lambda_i \left[  (1- B_{i,\numS}(\alpha))(\alpha - q) ) + q(1- \alpha) \right] - (z+1) \min \left\{ (1-q) \alpha, q(1-\alpha) \right\} \leq 0.
\end{alignat}
\end{subequations}
Remark that \eqref{eq:whole_class_pb_3} is the epigraph formulation of
\begin{equation*}
 \sup_{\alpha \in (0,1)} \sum_{i=1}^{\numS} \lambda_i \left[ \frac{ (1- B_{i,\numS}(\alpha))(\alpha - q) + q(1- \alpha) } { \min \left\{ (1-q) \alpha, q(1-\alpha) \right\}} - 1 \right].
\end{equation*}
Hence, by equivalence between \eqref{eq:whole_class_pb} and \eqref{eq:whole_class_pb_3} we conclude that,
\begin{equation*}
\sup_{F \in \mathcal{F}} \ratio_\numS \left( \pi^{\bm{\lambda}}, F \right) =  \sup_{\alpha \in (0,1)} \sum_{i=1}^{\numS} \lambda_i \left[ \frac{ (1- B_{i,\numS}(\alpha))(\alpha - q) + q(1- \alpha) } { \min \left\{ (1-q) \alpha, q(1-\alpha) \right\}} - 1 \right].
\end{equation*}

For the last step of the proof, we use the following lemma (whose proof is deferred to \Cref{apx:sec_body}), which establishes that the worst-case computed above is achieved by a Bernoulli distribution.
\begin{lemma}
\label{lem:OS_vs_bern}
For any $r \in \{1 \ldots, \numS\}$ and $\alpha \in [0,1]$,
\begin{equation*}
\ratio_{\numS} \left( \OS{r} , \mathcal{B} \left( 1 - \alpha \right)  \right) = \frac{ (1- B_{r,\numS}(\alpha))(\alpha - q) + q(1- \alpha) } { \min \left\{ (1-q) \alpha, q(1-\alpha) \right\}} - 1.
\end{equation*}
\end{lemma}
 This completes the proof.
\end{proof}


\setcounter{equation}{0}
\setcounter{proposition}{0}
\setcounter{lemma}{0}
\setcounter{theorem}{0}

\section{Proofs for Section \ref{sec:optimal}}\label{apx:section_opt}

\begin{proof}[\textbf{Proof of \Cref{thm:minimax_reduction}}]
Fix $\numS \geq 1$. It is easy to see that
\begin{equation*}
\inf_{\pi \in \Pi_\numS} \sup_{F \in \mathcal{F}} \ratio_{\numS} \left( \pi, F \right) \leq  \inf_{\mix \in \setOS} \sup_{F \in \mathcal{F}} \ratio_{\numS} \left( \mix, F \right).
\end{equation*}
We now prove that
\begin{equation*}
\inf_{\pi \in \Pi_\numS} \sup_{F \in \mathcal{F}} \ratio_{\numS} \left( \pi, F \right) \geq  \inf_{\mix \in \setOS} \sup_{F \in \mathcal{F}} \ratio_{\numS} \left( \mix, F \right).
\end{equation*}
To do so, we claim that we only need to show that mixture of order statistics policies are optimal when reducing the space of distributions to Bernoulli ones. Formally, we need to show that
\begin{equation}
\label{eq:equivalent_minimax}
\inf_{\pi \in \Pi_\numS} \sup_{ \mu \in [0,1]} \ratio_{\numS} \left( \pi, \mathcal{B} \left( \mu \right) \right) = \inf_{\mix \in \setOS} \sup_{ \mu \in [0,1] } \ratio_{\numS} \left( \mix, \mathcal{B} \left( \mu \right) \right).
\end{equation}
Indeed, assuming that \eqref{eq:equivalent_minimax} holds, one concludes the proof by remarking that,
\begin{align*}
\inf_{\pi \in \Pi_\numS} \sup_{F \in \mathcal{F}} \ratio_{\numS} \left( \pi, F \right) &\geq \inf_{\pi \in \Pi_\numS} \sup_{ \mu \in [0,1]} \ratio_{\numS} \left( \pi, \mathcal{B} \left( \mu \right) \right)\\
&\stackrel{(a)}{=} \inf_{\mix \in \setOS} \sup_{ \mu \in [0,1] } \ratio_{\numS} \left( \mix, \mathcal{B} \left( \mu \right) \right)\\
&\stackrel{(b)}{=} \inf_{\mix \in \setOS} \sup_{ F \in \mathcal{F} } \ratio_{\numS} \left( \mix, F \right),
\end{align*}
where $(a)$ would follow from \eqref{eq:equivalent_minimax} and $(b)$ is a consequence of \Cref{thm:main_mixture}.

We now prove \eqref{eq:equivalent_minimax}.

We first reduce the set of policies $\pi \in \Pi_\numS$ without loss of optimality for the following problem.
\begin{equation}
\label{eq:minimax_bern}
 \inf_{\pi \in \Pi_\numS} \sup_{ \mu \in [0,1]} \ratio_{\numS} \left( \pi, \mathcal{B} \left( \mu \right) \right)
\end{equation}
 We show that one may restrict attention to policies such that the support of the distribution of inventory is included in the interval defined by the smallest observed demand and the largest one.  Formally, the following result in proved in \Cref{apx:sec_body}.
\begin{lemma}
\label{lem:support}
For any policy $\pi \in \Pi_\numS$  there exists a policy $\pi' \in \Pi_\numS$ with a lower cost such that for every $\mathbf{D}_1^{\numS} \in \{0,1\}^{\numS}$, the support of $\pi' \left( \mathbf{D}_1^{\numS} \right)$ is a subset of $[D_{1:\numS},D_{\numS:\numS}]$.
\end{lemma}
Note that \Cref{lem:support} implies that $\pi' \left( \mathbf{0}_1^{\numS} \right) = 0$ and $\pi' \left( \mathbf{1}_1^{\numS} \right) = 1 $ where $\mathbf{0}_1^{\numS}$ (resp. $\mathbf{1}_1^{\numS}$) is the sequence of historical data in which all demand observations are $0$ (resp. $1$).
In turn, we leverage this result to further reduce the space of policies without loss of optimality to the $(\numS+1)$-dimensional space of sum-based policies defined as follows.

\begin{definition} (Sum-based policies)
Consider a sequence $ \mathbf{e} = \left(  e_i \right)_{i \in \{0, \ldots, \numS\}} \in [0,1]^{\numS+1} $.
We say that a policy $\DS$ is a sum-based policy if for any $ i \in \{0, \ldots, \numS\}$ and any $\mathbf{D}_1^{\numS }\in \{0,1\}^{\numS} $, such that $\sum_{j=1}^\numS D_j = i$, we have that,
\begin{equation*}
\DS \left( \mathbf{D}_1^\numS \right) = e_i.
\end{equation*}
\end{definition}

Let $\pi \in \Pi_\numS$ be a policy which support is included in the interval defined by the smallest observed demand and the largest one. By \Cref{lem:support} this restriction is without loss of optimality. We construct a sum-based policy that ensures the same cost as $\pi$ against any Bernoulli distribution. Define for every $i \in \{ 0 , \ldots, \numS \}$ the set $\mathcal{D}^i_\numS$ as
\begin{equation*}
\mathcal{D}^i_\numS:= \left \{ \mathbf{D}_1^{\numS} \in \{0,1\}^{\numS} \, \Big \vert \, \sum_{j=1}^{\numS} D_j = i \right \}.
\end{equation*}
Moreover, consider the sequence $ \mathbf{e} = \left(  e_i \right)_{i \in \{0, \ldots, \numS\}} \in [0,1]^{\numS+1} $ such that for every $i \in \{0 ,\ldots, \numS\}$ 
\begin{equation*}
e_i = \frac{1}{| \mathcal{D}^i_\numS |} \sum_{\mathbf{D}_{1}^\numS \in  \mathcal{D}^i_\numS}  \mathbb{E}_{x \sim { \pi \left( \mathbf{D}_{1}^\numS \right)} } \left[ x  \right].
\end{equation*}
By \Cref{lem:support} we have that $e_i \in [0,1]$ for all $i \in \{0, \ldots, \numS \}$, $e_0 = 0$ and $e_\numS =1$ which implies that  $\DS$  is a well defined sum-based policy.

To ease notations, let $S_j$ denote the event $\{\sum_{i=1}^{\numS} D_i = j \}$ for every $j \in \{0, \ldots, \numS \}$.
We note that for every $\mu \in [0,1]$ the cost of the policy $\pi$ satisfies

\begin{align}
\frac{\mathcal{C} \left(\pi, \mathcal{B} \left( \mu \right) ,\numS \right)}{b+h} 
&\stackrel{(a)}{=} \frac{1}{b+h} \mathbb{E}_{\mathbf{D}_1^\numS \sim \mathcal{B} \left(\mu \right) }\left [ \mathbb{E}_{x \sim { \pi \left( \mathbf{D}_{1}^\numS \right)} } \left[  \mu b \left( 1 - x \right) + \left(1 - \mu \right) h x \right]  \right]\nonumber\\
&=  \mu \cdot  q +  \sum_{i=0}^{\numS}  \left(1 - \mu - q \right) \cdot \mathbb{E}_{\mathbf{D}_1^\numS \sim \mathcal{B} \left(\mu \right) }\left [ \mathbb{E}_{x \sim { \pi \left( \mathbf{D}_{1}^\numS \right)} } \left[ x  \right] \, \Big \vert \, S_i  \right] \cdot \mathbb{P} \left( S_i \right)  \nonumber \\
&\stackrel{(b)}{=} \mu \cdot  q +  \sum_{i=0}^{\numS}  \left(1 - \mu - q \right) \cdot  \frac{1}{| \mathcal{D}^i_\numS |} \sum_{\mathbf{D}_{1}^\numS \in  \mathcal{D}^i_\numS}  \mathbb{E}_{x \sim { \pi \left( \mathbf{D}_{1}^\numS \right)} } \left[ x  \right] \cdot \mathbb{P} \left( S_i \right) \nonumber \\
&= \mu \cdot  q +  \sum_{i=0}^{\numS}  \left(1 - \mu - q \right) \cdot  e_i  \cdot \mathbb{P} \left( S_i \right)  = \mathcal{C} \left(\DS, \mathcal{B} \left( \mu \right) ,\numS \right), \label{eq:equal_DS}
\end{align}
where $(a)$ holds because the support of $\pi \left( \mathbf{D}_1^{\numS} \right)$ is a subset of $[D_{1:\numS},D_{\numS:\numS}]$, which is included in $[0,1]$ for Bernoulli distributions and $(b)$ follows from the fact that, for Bernoulli distributions, the distribution of  $\mathbf{D}_{1}^\numS$ conditional on $\{\sum_{j=1}^{\numS} D_j = i \}$ is  that of a uniform law on $\mathcal{D}^i_\numS$.

\Cref{lem:support} along with \eqref{eq:equal_DS} imply that the minimax problem across the general set of data-driven policies is actually equivalent to a minimax problem in which the space of policies is parameterized by a $(\numS+1)$ dimensional space. Namely, we have showed that
\begin{equation}
\label{eq:initial_problem_reduction}
 \inf_{\pi \in \Pi_\numS} \sup_{ \mu \in [0,1]} \ratio_{\numS} \left( \pi, \mathcal{B} \left( \mu \right) \right) = \inf_{\substack{ \mathbf{e} \in [0,1]^{\numS+1} \\ e_0 =0, \, e_\numS =1} } \sup_{ \mu \in [0,1]} \ratio_{\numS} \left( \DS, \mathcal{B} \left( \mu \right) \right).
\end{equation}
Recall that for a policy $\DS$, $e_i$ represents the inventory prescribed by the policy after observing $i$ ones. Natural candidate policies in this space are ones for which the inventory level prescribed is increasing as a function of the number of ones observed in historical data. Our next result formalizes this idea.
\begin{lemma}
\label{lem:increasing_only}
For any $\numS \geq 1$,
\begin{equation*}
\inf_{\substack{ \mathbf{e} \in [0,1]^{\numS+1}\\ e_0 =0, \, e_\numS =1} } \sup_{ \mu \in [0,1]} \ratio_{\numS} \left( \DS, \mathcal{B} \left( \mu \right) \right) = \inf_{\substack{\mathbf{e} \in [0,1]^{\numS+1}\\ e_0 =0, \, e_\numS =1 \\ \left(e_i \right) \, \text{non-decreasing}} } \sup_{ \mu \in [0,1]} \ratio_{\numS} \left( \DS, \mathcal{B} \left( \mu \right) \right).
\end{equation*}
\end{lemma}
The proof is deferred to \Cref{apx:sec_body}.

The last step of our proof consists in showing that the performance of any policy $\DS \in \Pi_\numS^{DS}$ such that $e_0 =0$, $e_\numS =1$ and $(e_i)_{i \in \{0,\ldots,\numS \}}$ is non-decreasing, can be reproduced by a mixture of order statistics policy. Consider a sequence $(e_i)_{i \in \{0,\ldots,\numS \}}$ satisfying these assumptions and define the vector of probabilities $\bm{\lambda}$ such that for all $i \in \{1,\ldots, \numS\}$,
\begin{equation*}
\lambda_i = e_{n-i+1} - e_{n-i}.
\end{equation*}
Note that for all $i \in \{1,\ldots, \numS\}$, $\lambda_i \geq 0$ by monotonicity of $(e_i)_{i \in \{0,\ldots,\numS \}}$ and $\sum_{i =1}^{\numS} \lambda_i = e_{\numS} - e_0 = 1$. Hence $\bm{\lambda}$ is a well defined probability vector. We now show that the mixture of order statistics policy $\mix$ incurs the same cost as $\DS$ against any Bernoulli distribution. Let $\mu \in [0,1]$, then the cost of $\mix$ is
\begin{align*}
\frac{\mathcal{C} \left(\mix, \mathcal{B} \left( \mu \right) ,\numS \right)}{b+h} &=  \mu \cdot  q +  \sum_{i=0}^{\numS}  \left(1 - \mu - q \right) \cdot \mathbb{E}_{\mathbf{D}_1^\numS \sim \mathcal{B} \left(\mu \right) }\left [ \mathbb{E}_{x \sim { \mix \left( \mathbf{D}_{1}^\numS \right)} } \left[ x  \right] \, \Big \vert \, S_i\right] \cdot \mathbb{P} \left( S_i \right)  \nonumber\\
&\stackrel{(a)}{=}  \mu \cdot  q +  \sum_{i=0}^{\numS}  \left(1 - \mu - q \right) \cdot \sum_{k=1}^{\numS} \lambda_k \cdot \mathbb{E}_{\mathbf{D}_1^\numS \sim \mathcal{B} \left(\mu \right) }\left [ D_{k:\numS} \, \Big \vert \, S_i \right] \cdot \mathbb{P} \left( S_i \right) \nonumber\\
& \stackrel{(b)}{=} \mu \cdot  q +  \sum_{i=0}^{\numS}  \left(1 - \mu - q \right) \cdot \mathbb{P} \left( S_i \right) \cdot \sum_{k=n-i+1}^{\numS} \lambda_k    \nonumber\\
&= \mu \cdot  q +  \sum_{i=0}^{\numS}  \left(1 - \mu - q \right) \cdot \mathbb{P} \left( S_i \right) \cdot  e_i   = \mathcal{C} \left(\DS, \mathcal{B} \left( \mu \right), \numS \right),
\end{align*}
where $(a)$ holds because $\mix$ prescribes $D_{k:\numS}$ with probability $\lambda_k$ for any $k \in \{1,\ldots,\numS\}$ and $(b)$ follows from the fact that for every $k \in \{1,\ldots,\numS\}$,
\begin{equation*}
D_{k:\numS} = \begin{cases}
0 \qquad &\text{a.s. if $\sum_{j=1}^{\numS} D_j \leq n-k$} \\
1 \qquad &\text{a.s. if $\sum_{j=1}^{\numS} D_j \geq n-k+1$}. 
\end{cases}
\end{equation*}
As a consequence, 
\begin{equation}
\label{eq:non_decreasingDS}
\inf_{\substack{ \mathbf{e} \in [0,1]^{\numS+1}\\ e_0 =0, \, e_\numS =1 \\ \left(e_i \right) \, \text{non-decreasing}} } \sup_{ \mu \in [0,1]} \ratio_{\numS} \left( \DS, \mathcal{B} \left( \mu \right) \right) \geq \inf_{\mix \in \Pi_{\numS}^{OS}} \sup_{\mu \in [0,1]} \ratio_{\numS} \left( \mix, \mathcal{B} \left( \mu \right) \right).
\end{equation}
We finally conclude that,
\begin{align*}
 \inf_{\pi \in \Pi_\numS} \sup_{ \mu \in [0,1]} \ratio_{\numS} \left( \pi, \mathcal{B} \left( \mu \right) \right) &\stackrel{(a)}{=} \inf_{\substack{ \mathbf{e} \in [0,1]^{\numS+1}\\ e_0 =0, \, e_\numS =1} } \sup_{ \mu \in [0,1]} \ratio_{\numS} \left( \DS, \mathcal{B} \left( \mu \right) \right) \\
 &\stackrel{(b)}{=} \inf_{\substack{ \mathbf{e} \in [0,1]^{\numS+1}\\ e_0 =0, \, e_\numS =1 \\ \left(e_i \right) \, \text{non-decreasing}} } \sup_{ \mu \in [0,1]} \ratio_{\numS} \left( \DS, \mathcal{B} \left( \mu \right) \right) \\
 & \stackrel{(c)}{\geq} \inf_{\mix \in \Pi_{\numS}^{OS}} \sup_{\mu \in [0,1]} \ratio_{\numS} \left( \mix, \mathcal{B} \left( \mu \right) \right),
\end{align*}
where $(a)$ holds by \eqref{eq:initial_problem_reduction}, $(b)$ follows from \Cref{lem:increasing_only} and $(c)$ is a consequence of \eqref{eq:non_decreasingDS}.
This completes the proof.

\end{proof}

\begin{proof}[\textbf{Proof of \Cref{prop:degen_opt}}]
Assume that,
\begin{equation}
\label{eq:imbalance_OS1}
\sup_{\mu \in [0,1-q]} \ratio_\numS \left( \OS{1}, \mathcal{B} \left( \mu \right) \right) > \sup_{\mu \in [1-q,q]} \ratio_\numS \left( \OS{1}, \mathcal{B} \left( \mu \right) \right).
\end{equation}
We show that $\OS{1}$ is an optimal mixture of order statistics policy. Note that for every $r \in \{2, \ldots, \numS\}$ we have
\begin{align*}
\sup_{\mu \in [0,1]} \ratio_\numS \left( \OS{1}, \mathcal{B} \left( \mu \right) \right) 
&\stackrel{(a)}{=} \sup_{\mu \in [0,1-q]} \ratio_\numS \left( \OS{1}, \mathcal{B} \left( \mu \right) \right)\\
&\stackrel{(b)}{<}  \sup_{\mu \in [0,1-q]} \ratio_\numS \left( \OS{r}, \mathcal{B} \left( \mu \right) \right) \leq \sup_{\mu \in [0,1]} \ratio_\numS \left( \OS{r}, \mathcal{B} \left( \mu \right) \right),
\end{align*}
where $(a)$ follows from \eqref{eq:imbalance_OS1} and $(b)$ holds by \Cref{lem:monotonic_phi} stated and proved in \Cref{app:add_aux}.

In turn, for every $\mix \in \Pi_\numS^{OS}$ such that $\lambda_1 <1$, we have that,
\begin{align*}
\sup_{\mu \in [0,1]} \ratio_\numS \left( \mix, \mathcal{B} \left( \mu \right) \right) = \sup_{\mu \in [0,1]} \sum_{i=1}^{\numS} \lambda_i \ratio_\numS \left( \OS{i}, \mathcal{B} \left( \mu \right) \right) > \sup_{\mu \in [0,1]} \ratio_\numS \left( \OS{1}, \mathcal{B} \left( \mu \right) \right). 
\end{align*}
As a consequence, $\OS{1}$ is optimal and satisfies,
\begin{equation*}
\ratio_{\numS}^{*} = \inf_{ \mix \in \Pi_\numS^{OS}}  \sup_{\mu \in [0,1]} \ratio_{\numS} \left( \pi^{\bm{\lambda}}, \mathcal{B} \left(  \mu\right) \right) = \sup_{\mu \in [0,1]} \ratio_{\numS} \left( \OS{1}, \mathcal{B} \left(  \mu\right) \right).
\end{equation*}

Similarly, assuming that \eqref{eq:extreme_n} does not hold, we show by a similar argument that $\OS{\numS}$ is optimal for Problem \eqref{eq:minmax}.
\end{proof}

\begin{proof}[\textbf{Proof of \Cref{prop:necessary}}]

Suppose first that
\begin{equation} \label{eq:cond1}
\sup_{\mu \in [0,1-q]} \ratio_\numS \left( \pi^{\bm{\lambda}}, \mathcal{B} \left( \mu \right) \right) > \sup_{\mu \in [1-q,1]} \ratio_\numS \left( \pi^{\bm{\lambda}}, \mathcal{B} \left( \mu \right) \right).
\end{equation}
In such a case, we show that there exists an alternative policy with strictly lower worst-case performance.

We first argue that $\bm{\lambda}$ must be such that $\lambda_1 < 1$. Indeed, note that, by assumption, we have 
\begin{equation} \label{eq:cond2}
\sup_{\mu \in [0,1-q]} \ratio_\numS \left( \OS{1}, \mathcal{B} \left( \mu \right) \right) \leq \sup_{\mu \in [1-q,1]} \ratio_\numS \left( \OS{1}, \mathcal{B} \left( \mu \right) \right),
\end{equation}
The conjunction of \eqref{eq:cond1} and \eqref{eq:cond2} implies that $\lambda_1 < 1$.

Next we argue that the policy $\OS{1}$ is strictly better than $\mix$ if $\mu \in [0,1-q]$. We have
\begin{align}
 \sup_{\mu \in [0,1-q]} \ratio_\numS \left( \pi^{\bm{\lambda}}, \mathcal{B} \left( \mu \right) \right) 
 &= \sup_{\mu \in [0,1-q]} \sum_{i=1}^{\numS} \lambda_i \ratio_\numS \left( \OS{i}, \mathcal{B} \left( \mu \right) \right) \nonumber \\
&\stackrel{(a)}{>}  \sup_{\mu \in [0,1-q]} \sum_{i=1}^{\numS} \lambda_i \ratio_\numS \left( \OS{1}, \mathcal{B} \left( \mu \right) \right) \nonumber \\
 &=  \sup_{\mu \in [0,1-q]} \ratio_\numS \left( \OS{1}, \mathcal{B} \left( \mu \right) \right). \label{eq:strict_improvement}
\end{align}
where $(a)$ follows from the fact that $\lambda_1 <1$, together with \Cref{lem:monotonic_phi} stated and proved in \Cref{app:add_aux}.

Next, we construct an explicit policy that improves upon $\mix$. 

For any $\nu \in [0,1]$, consider the policy $\tilde{\pi}_{\nu}$ which chooses the policy $\OS{1}$ with probability $\nu$ and the policy $\pi^{\bm{\lambda}}$ with probability $1-\nu$. Remark that $\tilde{\pi}_{\nu}$ is a mixture of order statistics policy and for any $F \in \mathcal{F}$,
\begin{equation*}
\ratio_\numS \left( \tilde{\pi}_{\nu}, F \right) = \nu \cdot \ratio_\numS \left( \OS{1}, F \right) + \left( 1- \nu \right) \cdot \ratio_\numS \left( \pi^{\bm{\lambda}}, F \right).
\end{equation*}

Define the mapping $L$ from $[0,1]$ to $\mathbb{R}$ such that,
\begin{equation*}
L : \nu \mapsto \sup_{ \mu \in [0,1-q]} \ratio_\numS \left( \tilde{\pi}_{\nu}, \mathcal{B} \left( \mu \right) \right) - 
\sup_{ \mu \in [1-q,1]} \ratio_\numS \left( \tilde{\pi}_{\nu}, \mathcal{B} \left( \mu \right) \right).
\end{equation*}
We first show that $L$ is continuous. Remark that it is sufficient to show that, the following mapping $g$ is continuous.
\begin{equation*}
g : \nu \mapsto \sup_{ \mu \in [0,1-q]}  f (\nu, \mu),
 \end{equation*}
where  $f (\nu, \mu):= \ratio_\numS \left( \tilde{\pi}_{\nu}, \mathcal{B} \left( \mu \right) \right)$. First remark that by \Cref{lem:OS_vs_bern}, the mapping $\mu \mapsto \ratio_\numS \left( \pi^{\bm{\lambda}'} , \mathcal{B} \left(\mu \right) \right)$ is continuous for every mixture of order statistic  $\pi^{\bm{\lambda}'}$. Hence, $f$ is continuous in its second component. Moreover, $f$ is affine in its first component, and $f \left( \cdot, \mu \right)$ is $M$- Lipschitz  for every $\mu \in [0,1-q]$, where $M := \sup_{ \mu \in [0,1-q]} | \ratio_\numS \left( \mix , \mathcal{B} \left(\mu \right) \right) - \ratio_\numS \left( \OS{1} , \mathcal{B} \left(\mu \right) \right)|$. Remark that $M < \infty$ as it is the supremum of a continuous function on a compact. By continuity of $f(\nu, \cdot)$ on a compact set we also have that, for every $\nu_1, \nu_2 \in [0,1]$, there exists $\mu_1$ and $\mu_2$ achieving the maximum for $f(\nu_1, \cdot )$ and $f(\nu_2, \cdot )$ and
\begin{align*}
g(\nu_1) - g(\nu_2) &= f(\nu_1,\mu_1) - f(\nu_2,\mu_2)\\
 & = f(\nu_1,\mu_1) - f(\nu_1,\mu_2) + f(\nu_1,\mu_2) - f(\nu_2,\mu_2) \leq f(\nu_1,\mu_2) - f(\nu_2,\mu_2)
\leq M | \nu_1 - \nu_2|.
\end{align*}
Which implies that $g$ is $M$-Lipschitz on $[0,1]$ and thus continuous.

Hence $L$ is continuous. Moreover, it satisfies $L(0) > 0$ and $L(1) \leq 0$ so, by the intermediate value theorem, we conclude that there exists $\nu^* \in (0,1]$ such that $L(\nu^*) = 0$. 

We now show that $\tilde{\pi}_{\nu^{*}}$ strictly improves on $\mix$. Indeed, we have
\begin{equation*}
\sup_{\mu \in [0,1]} \ratio_\numS \left( \tilde{\pi}_{\nu^{*}}, \mathcal{B} \left( \mu \right) \right) \stackrel{(a)}{=} \sup_{ \mu \in [0,1-q]} \ratio_\numS \left( \tilde{\pi}_{\nu^{*}}, \mathcal{B} \left( \mu \right) \right) \stackrel{(b)}{<} \sup_{\mu \in [0,1-q]} \ratio_\numS \left( \pi^{\bm{\lambda}}, \mathcal{B} \left( \mu \right) \right),
\end{equation*}
where $(a)$ holds because $L(\nu^*) = 0$ and $(b)$ follows from \eqref{eq:strict_improvement} and from the fact that $\nu^* > 0$. This shows that, $\pi^{\bm{\lambda}}$ is suboptimal.

Suppose that 
\begin{equation*}
\sup_{\mu \in [0,1-q]} \ratio_\numS \left( \pi^{\bm{\lambda}}, \mathcal{B} \left( \mu \right) \right) < \sup_{\mu \in [1-q,1]} \ratio_\numS \left( \pi^{\bm{\lambda}}, \mathcal{B} \left( \mu \right) \right).
\end{equation*}
In this case, the same reasoning, but by increasing the weight on the $\numS^{th}$ order statistic would lead to a strict improvement. Therefore, if an optimal policy $\mix$ exists for problem \eqref{eq:problem_ub} it must satisfy,
\begin{equation*}
\sup_{\mu \in [0,1-q]} \ratio_\numS \left( \pi^{\bm{\lambda}}, \mathcal{B} \left( \mu \right) \right) = \sup_{\mu \in [1-q,1]} \ratio_\numS \left( \pi^{\bm{\lambda}}, \mathcal{B} \left( \mu \right) \right).
\end{equation*}

\end{proof}

\begin{proof}[\textbf{Proof of \Cref{prop:balancing_regret}}]
It follows from \eqref{eq:extreme_1} and \eqref{eq:extreme_n} that there exists a $k \in \{2, \ldots, \numS\}$ such that 
\begin{eqnarray*}
\sup_{\mu \in [0,1-q]} \ratio_{\numS} \left(\OS{k-1},\mathcal{B} \left( \mu \right) \right) &\leq&\sup_{\mu \in [1-q,1]} \ratio_{\numS} \left(\OS{k-1},\mathcal{B} \left( \mu \right) \right) \\
\sup_{\mu \in [0,1-q]} \ratio_{\numS} \left(\OS{k},\mathcal{B} \left( \mu \right) \right) &\geq&\sup_{\mu \in [1-q,1]} \ratio_{\numS} \left(\OS{k},\mathcal{B} \left( \mu \right) \right). 
\end{eqnarray*}
Pick a $k$ verifying these two relations. We now construct a policy $\optpol$ randomizing between $D_{k-1:\numS}$ and $D_{k:\numS}$ and which satisfies the necessary condition  \eqref{eq:necessary_balance}. Consider the family of policies $\left( \pi^{k,\lambda} \right)_{\lambda \in [0,1]}$ prescribing $D_{k:\numS}$ w.p $\lambda$ and $D_{k-1:\numS}$ w.p $1-\lambda$.

We consider the function $L$ defined from $[0,1]$ to $\mathbb{R}$ as,  
\begin{equation*}
L \,: \lambda \mapsto \sup_{\mu 
\in [0,1-q]}  \ratio_{\numS} \left(\pi^{k,\lambda},\mathcal{B} \left( \mu \right) \right) -  \sup_{\mu 
\in [1-q,1]} \ratio_{\numS} \left(\pi^{k,\lambda},\mathcal{B} \left( \mu \right) \right),
\end{equation*}
and note that $L(0) \leq 0$ and $L(1) \geq 0$.  Moreover, $L$ is continuous on $[0,1]$ (see proof of \Cref{prop:necessary}). Thus by the intermediate value theorem, $L(\gamma) = 0 $ for some $\gamma \in [0,1]$.
We denote by $\optpol$ our candidate policy that prescribes the order statistic  $D_{k:\numS}$ w.p $ \gamma$ and $D_{k-1:\numS}$ w.p $1 - \gamma$.
We define $\mu^{-} \in \argmax_{\mu 
\in [0,1-q]} \ratio_{\numS} \left(\optpol,\mathcal{B} \left( \mu \right) \right)$ and $\mu^{+} \in \argmax_{\mu 
\in [1-q,1]} \ratio_{\numS} \left(\optpol,\mathcal{B} \left( \mu \right) \right)$ which exists by continuity on a compact. 
By construction of $\mu^{+}$, $\mu^{-}$ and because $L(\gamma) = 0$, we conclude that,
\begin{equation*}
\ratio_n\left(\pi^{k,\gamma}, \mathcal{B}( \mu^{-}) \right) =  \ratio_n\left(\optpol, \mathcal{B}( \mu^{+}) \right) = \sup_{\mu \in [0,1]} \ratio_n\left(\optpol, \mathcal{B}( \mu^{+}) \right).
\end{equation*}
\end{proof}

\begin{proof}[\textbf{Proof of \Cref{prop:prior}}]
Consider the family of priors $\left( p_{\delta} \right)_{\delta \in [0,1]}$ supported on $\left \{ \mu^{-},\mu^{+} \right \}$ and such that for any $\delta \in [0,1]$,
\begin{equation*}
p_\delta(\mu) = \begin{cases}
\delta \qquad &\text{if $\mu =  \mu^{+}$}\\
1 - \delta \qquad &\text{if $\mu = \mu^{-}$}. 
\end{cases}
\end{equation*}
We now show that there exists $\delta$ such that $\optpol$ is optimal for the problem,
\begin{equation}
\label{eq:bayesian_problem}
\inf_{\mix \in \Pi_\numS^{OS}}
 \mathbb{E}_{\mu \sim p_\delta} \left[ \ratio_\numS(\mix,\mathcal{B}(\mu)) \right].
\end{equation}

We first establish a sufficient condition for a policy $\mix$ to be optimal for problem \eqref{eq:bayesian_problem}.

Remark that policies in $\Pi_\numS^{OS}$ observe samples prior to decision hence,
\begin{equation*}
 \inf_{\mix \in \Pi_\numS^{OS}}
 \mathbb{E}_{\mu \sim p_\delta} \left[ \ratio_\numS(\pi,\mathcal{B}(\mu)) \right] = \sum_{j=0}^{\numS} \mathbb{P} \left( \sum_{i=1}^{\numS} D_i = j \right)   \inf_{\mix \in \Pi_\numS^{OS}}  \mathbb{E}_{\mu \sim p_\delta} \left[ \ratio_\numS \left (\mix,\mathcal{B}(\mu) \right) \Big \vert \sum_{i=1}^{\numS} D_i = j  \right],
\end{equation*}
where the equality holds because $\mix$ observes the historical samples and because the posterior distribution of $p_\delta$ only depends on the number of ones observed, i.e. $\sum_{i=1}^{\numS} D_i $ is a sufficient statistic. To ease notations, let $S_j$ denote the event $\{\sum_{i=1}^{\numS} D_i = j \}$ for every $j \in \{0, \ldots, \numS \}$.
To solve the inner optimization problem, we first notice that for every $j \in \{0, \ldots, \numS \}$ we have
\fontsize{10.5pt}{10.5pt}\selectfont
\begin{align*}
\mathbb{E}_{\mu \sim p_\delta} \left[ \ratio_\numS(\mix,\mathcal{B}(\mu)) \vert S_j \right]
&= \mathbb{P} \left(\mu = \mu^{+} \vert S_j \right) \frac{ \mu^{+} - ( 1- q)  + \left(1 - \mu^{+} - q \right) \mathbb{E}_{\mathbf{D}_1^\numS \sim \mathcal{B} \left(\mu^{+} \right) }\left [  \mathbb{E}_{x \sim { \mix \left( \mathbf{D}_{1}^\numS \right)} } [x] \, \Big \vert \, S_j \right] }{(1-\mu^{+})\left(1-q \right)}\\
&\qquad+ \mathbb{P} \left(\mu = \mu^{-} \vert  S_j \right) \frac{\left(1 - \mu^{-} - q \right) \mathbb{E}_{\mathbf{D}_1^\numS \sim \mathcal{B} \left(\mu^{-} \right) }\left [  \mathbb{E}_{x \sim { \mix \left( \mathbf{D}_{1}^\numS \right)} } [x] \, \Big \vert \, S_j \right]}{ \mu^{-}q}\\
&\stackrel{(a)}{=}   a_j \cdot \mathbb{E}_{\mathbf{D}_1^\numS \sim \mathcal{B} \left(\frac{1}{2} \right) }\left [  \mathbb{E}_{x \sim { \mix \left( \mathbf{D}_{1}^\numS \right)} } [x] \, \Big \vert \, S_j \right]  + b_j, 
\end{align*}
\normalsize
where $(a)$ holds because  $\mathbb{E}_{\mathbf{D}_1^\numS \sim \mathcal{B} \left(\mu \right) }\left [  \mathbb{E}_{x \sim { \mix \left( \mathbf{D}_{1}^\numS \right)} } [x ] \, \Big \vert \, S_j \right] = \mathbb{E}_{\mathbf{D}_1^\numS \sim \mathcal{B} \left(\mu' \right) }\left [  \mathbb{E}_{x \sim { \mix \left( \mathbf{D}_{1}^\numS \right)} } [x] \, \Big \vert \, S_j \right]$ for any $\mu, \mu' \in (0,1)$. This follows from the fact that for Bernoulli distributions, the distribution of  $\mathbf{D}_{1}^\numS$ conditional on $S_j$ is the same for every $\mu \in (0,1)$. Moreover,
\begin{equation*}
a_j :=  \mathbb{P} \left(\mu = \mu^{+} \, \vert \, S_j \right) \frac{ 1 - q - \mu^{+}}{ (1-\mu^{+})(1-q)} + \mathbb{P} \left(\mu = \mu^{-} \, \vert \, S_j \right) \frac{ 1- q - \mu^{-} }{ \mu^{-} q}
\end{equation*}
and $b_j = \mathbb{P} \left(\mu = \mu^{+} \, \vert \, S_j \right) \frac{\mu^{+} - (1-q)}{(1-\mu^{+}) (1-q)} $.

Note that, by definition of mixture of order statistics policies, we must have that, $ \mix \left( \mathbf{D}_{1}^\numS \right)$ has a support included in $[0,1]$ since $\mathbf{D}_{1}^\numS \in \{ 0, 1 \}^\numS$.
Hence, we obtain that for any $x_0 \in [0,1]$ if a policy $\mix$ satisfies the following property,
\begin{equation*}
 \mathbb{E}_{\mathbf{D}_1^\numS \sim \mathcal{B} \left(\frac{1}{2} \right) }\left [ \mathbb{E}_{x \sim { \mix \left( \mathbf{D}_{1}^\numS \right)} } \left[x \right] \, \Big \vert \, S_j  \right] = \begin{cases}
0 \qquad \text{if $\; a_j  > 0$}\\
x_0 \quad \, \, \, \text{if $\; a_j  = 0$}\\
1 \qquad \text{if $\; a_j  < 0$,}
\end{cases}
\end{equation*}
then it is optimal for problem \eqref{eq:bayesian_problem}. We now aim at proving that there exists a prior such that $\optpol$ satisfies this sufficient condition. The challenge is that the sufficient condition involves the sign of the coefficients $(a_j)_{j \in \{0,\ldots, \numS\}}$ depending on $\mu^{-}$, $\mu^{+}$ and $\delta$. We simplify this dependence with our next lemma  by showing that  for any choice of $\mu^{-} < 1-q \leq \mu^{+}$, we can construct a prior such that the sequence $(a_j)_{j \in \{0,\ldots, \numS\}}$ is decreasing and hits $0$ exactly once. Formally, we show the following.
\begin{lemma}
\label{lem:posterior_structure}
For any $\mu^{-} \in [0,1-q)$, $\mu^{+} \in [1-q,1)$, and for any $j_0 \in \{1, \ldots, \numS\}$, there exists $\delta' \in [0,1] $ such that under prior $p_{\delta'}$, the sequence $(a_j)_{j \in \{0,\ldots, \numS \}}$ is strictly decreasing and $a_{j_0} = 0$.
\end{lemma}
The proof is deferred to \Cref{apx:sec_body}. \Cref{lem:posterior_structure} implies that for any $j_0 \in \{1,\ldots, \numS \}$ and $x_0 \in [0,1]$, any policy $\mix \in \Pi_{\numS}^{OS}$ that satisfies
\begin{equation*}
 \mathbb{E}_{\mathbf{D}_1^\numS \sim \mathcal{B} \left(\frac{1}{2} \right) }\left [
 \mathbb{E}_{x \sim { \mix \left( \mathbf{D}_{1}^\numS \right)} } \left[x\right] \, \Big \vert \, S_j  \right] = \begin{cases}
0 \qquad \text{if $\; j \leq j_0 - 1 $}\\
x_0 \quad \, \, \, \text{if $\; j = j_0 $}\\
1 \qquad \text{if $\; j \geq j_0 +1$,}
\end{cases}
\end{equation*}
is optimal for problem \eqref{eq:bayesian_problem}. 
We finally prove that $\optpol$ satisfies this simplified sufficient condition.

Note that by construction, for any $j \in \{1, \ldots, \numS \}$,
\begin{equation*}
 \mathbb{E}_{\mathbf{D}_1^\numS \sim \mathcal{B} \left(\frac{1}{2} \right) }\left [ \mathbb{E}_{x \sim  \optpol \left( \mathbf{D}_{1}^\numS \right) }  \left[x \right] \, \Big \vert \, S_j  \right]= \lambda \mathbb{E} \left[ D_{k:\numS} \, \Big \vert \, S_j \right] +  (1 -\lambda) \mathbb{E} \left[ D_{k-1:\numS} \,\Big \vert \, S_j  \right]
\end{equation*}
which implies that,
 \begin{equation}
  \mathbb{E}_{\mathbf{D}_1^\numS \sim \mathcal{B} \left(\frac{1}{2} \right) }\left [ 
 \mathbb{E}_{x \sim {\optpol  \left( \mathbf{D}_{1}^\numS \right)} } \left[x \right] \, \Big \vert \, S_j  \right] = \begin{cases}
0 \qquad \text{if $\; j  \leq n - k $}\\
\lambda \quad \, \, \, \text{if $\; j = n - k + 1 $}\\
1 \qquad \text{if $\; j  \geq n - k + 2$.}
\end{cases}
\end{equation}
Therefore, \Cref{lem:posterior_structure} applied with $j_0 = n - k + 1 $  implies that there exists $\delta_k$ such that, $\optpol$ is optimal for problem \eqref{eq:bayesian_problem}.  Setting $p^{*} = p_{\delta_k}$, we showed that,
\begin{equation*}
\inf_{ \mix \in \Pi_{\numS}^{OS} }  \mathbb{E}_{\mu \sim p^{*}} \left[ \ratio_\numS(\mix,\mathcal{B}(\mu)) \right] =  \mathbb{E}_{\mu \sim p^{*}} \left[ \ratio_\numS(\optpol,\mathcal{B}(\mu)) \right].
\end{equation*}
\end{proof}

\begin{proof} [\textbf{Proof of \Cref{thm:optimal_min_max}}]
First assume that, \eqref{eq:extreme_1} and \eqref{eq:extreme_n} hold and consider $k \in \{2, \ldots, \numS \}$, $\gamma \in [0,1]$, $\mu^{-} \in [0,1-q]$ and $\mu^{+} \in [1-q,1]$ as defined in \Cref{prop:balancing_regret}. We have that $\optpol$ satisfies the necessary condition \eqref{eq:indifference_body}.

We now show that $\optpol$ is optimal for the problem
\begin{equation*}
\inf_{ \mix \in \Pi_\numS^{OS}}  \sup_{\mu \in [0,1]} \ratio_{\numS} \left( \pi^{\bm{\lambda}}, \mathcal{B} \left(  \mu\right) \right).
\end{equation*}
First, we remark that,
\begin{equation*}
\ratio_{\numS} \left( \pi^{k,\lambda}, \mathcal{B}(\mu^{-})\right)= \sup_{\mu \in [0,1]} \ratio_{\numS} \left( \optpol, \mathcal{B} \left(  \mu \right) \right) \geq  \inf_{ \mix \in \Pi_\numS^{OS}}  \sup_{\mu \in [0,1]} \ratio_{\numS} \left( \pi^{\bm{\lambda}}, \mathcal{B} \left(  \mu\right) \right)
\end{equation*}
because $\optpol \in \Pi_\numS^{OS}$. To prove the lower bound, we note that for these choices of $k$, $\gamma$, $\mu^{-}$ and $\mu^{+}$, \Cref{prop:prior} ensures that there exists a prior $p^{*}$ supported on $\{ \mu^{-},\mu^{+} \}$ such that,
\begin{equation}
\label{eq:opt_dual}
\inf_{ \mix \in \Pi_{\numS}^{OS} }  \mathbb{E}_{\mu \sim p^{*}} \left[ \ratio_\numS(\mix,\mathcal{B}(\mu)) \right] =  \mathbb{E}_{\mu \sim p^{*}} \left[ \ratio_\numS(\optpol,\mathcal{B}(\mu)) \right].
\end{equation}
Therefore, 
 \begin{align*}
\inf_{ \mix \in \Pi_\numS^{OS}}  \sup_{\mu \in [0,1]} \ratio_{\numS} \left( \pi^{\bm{\lambda}}, \mathcal{B} \left(  \mu\right) \right)
 &=  \inf_{ \mix \in \Pi_\numS^{OS}}  \sup_{p \in \Delta \left([0,1] \right)} \mathbb{E}_{\mu \sim p} \left[ \ratio_{\numS} \left( \pi^{\bm{\lambda}}, \mathcal{B} \left(  \mu\right) \right) \right] \\
 &\stackrel{(c)}{\geq } \sup_{p \in \Delta \left([0,1] \right)}  \inf_{ \mix \in \Pi_\numS^{OS}}  \mathbb{E}_{\mu \sim p}\left[ \ratio_{\numS} \left( \mix, \mathcal{B} \left(  \mu\right) \right) \right]\\
 &\geq  \inf_{ \mix \in \Pi_\numS^{OS}}  \mathbb{E}_{\mu \sim p^*}\left[ \ratio_{\numS} \left( \mix, \mathcal{B} \left(  \mu\right) \right) \right]\\
 &\stackrel{(d)}{=} \mathbb{E}_{\mu \sim p^{*}} \left[ \ratio_\numS(\optpol,\mathcal{B}(\mu)) \right] \stackrel{(e)}{=}  \ratio_{\numS} \left( \pi^{k,\lambda}, \mathcal{B}(\mu^{-})\right) ,
 \end{align*}
where $(c)$ holds by weak duality, $(d)$ is a consequence of \eqref{eq:opt_dual}  and $(e)$ follows from \eqref{eq:indifference_body}. 
The lower bound matches our upper bound. Thus all inequalities are equalities and we have exhibited a saddle point for  \eqref{eq:opt_mixed_strategies_body}. This implies that,
\begin{equation*}
 \ratio_{\numS}^{*} \stackrel{(a)}{=} \inf_{ \mix \in \Pi_\numS^{OS}}  \sup_{\mu \in [0,1]} \ratio_{\numS} \left( \pi^{\bm{\lambda}}, \mathcal{B} \left(  \mu\right) \right) = \ratio_{\numS} \left( \pi^{k,\gamma}, \mathcal{B}(\mu^{-})\right) = \sup_{F \in \mathcal{F}} \ratio_n\left( \pi^{k,\gamma}, F \right),
\end{equation*}
where $(a)$ follows from \Cref{thm:minimax_reduction}.

In other words, $ \pi^{k,\gamma} $ is an optimal minimax data-driven algorithm and its performance can be explicitly computed by evaluating it against a specific Bernoulli distribution.
\end{proof}

\begin{proof}[\textbf{Proof of \Cref{cor:derand}}]
Fix $\numS \geq 1$.
We note that it is sufficient to show that, for every $F \in \mathcal{F}$, 
\begin{equation}
\label{eq:uniform_improve}
\ratio_n\left( \derpol , F \right) \leq \ratio_n\left( \optpol, F \right).
\end{equation}
We then conclude the proof by remarking that,
\begin{equation*}
\ratio_{\numS}^{*} \leq \sup_{F \in \mathcal{F}} \ratio_n\left( \derpol , F \right) \stackrel{(a)}{\leq} \sup_{F \in \mathcal{F}} \ratio_n\left( \optpol , F \right) \stackrel{(b)}{=} \ratio_{\numS}^{*},
\end{equation*}
where $(a)$ follows from \eqref{eq:uniform_improve} and $(b)$ is a consequence of \Cref{thm:optimal_min_max}. 

We now prove \eqref{eq:uniform_improve}. Fix $F \in \mathcal{F}$. We have that,
\begin{align*}
\mathcal{C}(\derpol,F,\numS)&= \mathbb{E}_{\mathbf{D}_1^\numS \sim F}\left [ \mathbb{E}_{x \sim { \derpol \left( \mathbf{D}_{1}^\numS \right)} } \left[  c_F(x) \right] \right]\\
&= \mathbb{E}_{\mathbf{D}_1^\numS \sim F}\left [  c_F( \left(1 - \gamma \right) D_{k-1:\numS} + \gamma D_{k:\numS})  \right]\\
&\stackrel{(a)}{\leq} \mathbb{E}_{\mathbf{D}_1^\numS \sim F}\left [  \left(1 - \gamma \right) c_F(  D_{k-1:\numS}) + \gamma c_F( D_{k:\numS})  \right]= \mathcal{C}(\optpol,F,\numS),
\end{align*}
where $(a)$ holds because $x \mapsto c_F(x)$ is the expectation of a family of convex functions and is thus convex.
\end{proof}


\setcounter{equation}{0}
\setcounter{proposition}{0}
\setcounter{lemma}{0}
\setcounter{theorem}{0}

\section{Proofs for Section \ref{sec:asymptotic}} \label{apx:sec_asymptotic}

\begin{proof}[\textbf{
Proof of \Cref{thm:asymptotic}}]
The proof of this theorem goes as follows. We first establish, in \Cref{lem:asymptotic_exact},  a characterization of the asymptotic behavior of single order statistic policy sequences for different regimes. As a corollary, we derive the asymptotic approximation of the worst-case performance of SAA.

Finally, we leverage the characterization of the optimal policy derived in \Cref{thm:optimal_min_max} to reduce the understanding of the optimal performance to a problem involving mixture of order statistics. We finally,  conclude by applying again \Cref{lem:asymptotic_exact}.

\vspace{2mm}
\noindent \textbf{Step 1:}
We characterize the performance of a sequence of single order statistic policies. We first remark that a sequence of single order statistic policies can be characterized by a sequence $\mathbf{r} := (r_{\numS})_{\numS \geq 1}$ where for each $\numS \geq 1$, $r_{\numS} \in \{1,\ldots,\numS\}$. We denote by $\bm{\pi}^{\mathbf{r}}$, the sequence of policies such that for any $\numS \geq 1$, $\pi^{\mathbf{r}}_{\numS} = \OS{r_{\numS}}$.

The next result characterizes the asymptotic behavior of the worst-case performance of  
$\bm{\pi}^{\mathbf{r}}$.
\begin{lemma}
\label{lem:asymptotic_exact}
\begin{enumerate}[i)]
\item If $\mathbf{r}$ is such that, $ \lim_{\numS \to \infty} \frac{|r_{\numS} - q\numS |}{\sqrt{\numS}} = \ell < \infty$,  then
\begin{equation*}
\lim_{\numS \to \infty} \sqrt{\numS} \cdot \sup_{F \in \mathcal{F}} \ratio_{\numS} \left( \pi^{\mathbf{r}}_{\numS}, F \right) =  \max 
\left[
\max_{\delta \geq 0 } H^{+}(\delta,\ell),
\max_{\delta \geq 0 } H^{-}(\delta,\ell) 
 \right].
\end{equation*}
where, $H^{+}(\delta,\ell) :=  \frac{\delta}{q(1-q)} \left(1 - \Phi \left( \frac{\delta - \ell}{\sqrt{q(1-q)}} \right) \right)$, $H^{-}(\delta,\ell) := \frac{\delta}{q(1-q)} \left(1 - \Phi \left( \frac{\delta + \ell}{\sqrt{q(1-q)}} \right) \right)$ and $\Phi$ is the cdf of the standard gaussian distribution.

Furthermore, 
\begin{itemize}
\item If the sequence $\mu_\numS$ is such that $ \lim_{\numS \to \infty} \sqrt{\numS} \left( 1 - q - \mu_\numS \right) = \delta >0$, then
\begin{equation*}
\lim_{\numS \to \infty} \sqrt{\numS} \ratio_\numS(\pi_{\numS}^{\mathbf{r}},\mathcal{B}(\mu_\numS)) = H^{+} \left( \delta, \ell \right).
\end{equation*}
\item If the sequence $\mu_\numS$ is such that $ \lim_{\numS \to \infty} \sqrt{\numS} \left (\mu_\numS - \left( 1 - q \right) \right) = \delta >0$, then
\begin{equation*}
\lim_{\numS \to \infty} \sqrt{\numS}  \ratio_\numS(\pi_{\numS}^{\mathbf{r}},\mathcal{B}(\mu_\numS))  = H^{-} \left( \delta, \ell \right).
\end{equation*}
\end{itemize}
\item If $\mathbf{r}$ is such that,  $\lim_{\numS \to \infty} \frac{|r_{\numS} - q\numS |}{\sqrt{\numS}} = \infty$  then, 
\begin{equation*}
\lim_{\numS \to \infty} \sqrt{\numS} \cdot \sup_{F \in \mathcal{F}} \ratio_{\numS} \left( \pi^{\mathbf{r}}_{\numS}, F \right) = \infty
\end{equation*}
and there exists a sequence of elements $\{\mu_{\numS}\}_{\numS \ge 1}$ in $[0,1]$ such that
\begin{equation*}
\lim_{\numS \to \infty} \sqrt{\numS}  \ratio_{\numS}(\pi^{\mathbf{r}}_{\numS}, \mathcal{B}(\mu_{\numS}))  = \infty.
\end{equation*}
\end{enumerate}
\end{lemma}
The proof is presented in \Cref{apx:sec_body}.

\Cref{lem:asymptotic_exact} establishes that there are two notable regimes driving the asymptotic worst-case performance. In the first regime where the sequence of order statistics is asymptotically ``close'' to $\ceil{ q \numS}$, in the sense that $r_\numS = q \numS + \mathcal{O} \left(\sqrt{\numS}\right)$, the worst-case relative regret decreases at a rate of $\Theta \left( 1 / \sqrt{\numS} \right)$. We also establish a closed form expression of the exact limiting constant associated to the rate of convergence and characterize the family of near worst-case distributions, namely  Bernoulli distributions whose means go to $1-q$ at a rate $\Theta \left( 1/\sqrt{\numS} \right)$.

In the second regime for which the sequence of order statistics is asymptotically ``far'' from $\ceil{ q \numS}$, we show that the worst-case relative regret decreases at a slower rate as it converges at a rate of $\omega \left( 1 / \sqrt{\numS} \right)$. This naturally implies that this family of policy sequences is necessarly suboptimal and strictly dominated by sequences of order statistics asymptotically ``close'' to $\ceil{ q \numS}$.

\vspace{2mm}
\noindent \textbf{Step 2:} We now characterize the asymptotic performance of SAA.
Let $\mathbf{r}^{\text{SAA}} = \left( \ceil{q \numS} \right)_{\numS \in \mathbb{N}}$ and recall that for every $\numS \in \mathbb{N}$, we have $\pi^{\text{SAA}}_{\numS} = \pi^{\mathbf{r}^{\text{SAA}} }_{\numS}$. Note that $\lim_{\numS \to \infty} \frac{|r^{\text{SAA}}_{\numS} - q\numS |}{\sqrt{\numS}} = 0$. Therefore, by (i) in \Cref{lem:asymptotic_exact} we conclude that,
\begin{equation}
\label{eq:SAA_asympt}
\lim_{\numS \to \infty} \sqrt{\numS} \sup_{F \in \mathcal{F}} \ratio_\numS \left( \pi_\numS^{\text{SAA}}, F \right) = \max 
\left[
\max_{\delta \geq 0 } H^{+}(\delta,0),
\max_{\delta \geq 0 } H^{-}(\delta,0) 
 \right] = C^*.
\end{equation}

\vspace{2mm}
\noindent \textbf{Step 3:} We finally derive an asymptotic approximation of the optimal performance.   
We use the characterization of the optimal policy derived in \Cref{thm:optimal_min_max}  and denote by $(k_{\numS})_{\numS \geq 1}$ and $(\gamma_{\numS})_{\numS \geq 1}$ the sequences of parameters that describe the optimal policy when facing $\numS$ samples. For any $\numS \in \mathbb{N}$, we have $\pi^{\mathbf{k},\bm{\gamma}}_\numS = \pi^{k_\numS,\gamma_\numS} $.
For every $\numS \in \mathbb{N}$, and every $\mu_\numS \in [0,1]$ remark that,
\small{
\begin{equation}
\label{eq:ub_lb_asymptotic}
\sup_{F \in \mathcal{F}} \ratio_\numS \left( \pi_\numS^{\text{SAA}}, F \right) \geq  \ratio_\numS^* \stackrel{(a)}{=} \sup_{F \in \mathcal{F}} \ratio_\numS \left( \pi^{\mathbf{k},\bm{\gamma}}_\numS, F \right) \stackrel{(b)}{\geq} \gamma_\numS \ratio_\numS \left( \pi^{k_\numS}, \mathcal{B} \left( \mu_\numS \right) \right) + (1-\gamma_\numS) \ratio_\numS \left( \pi^{k_\numS-1}, \mathcal{B} \left( \mu_\numS \right) \right),
\end{equation}
}
where $(a)$ holds by \Cref{thm:optimal_min_max} and $(b)$ is by definition of $\pi^{\mathbf{k},\bm{\gamma}}_\numS$.

Remark that, \eqref{eq:SAA_asympt} together with the first inequality of \eqref{eq:ub_lb_asymptotic} imply that,
\begin{equation}
\label{eq:ub_asymptotic_ratio}
\limsup_{\numS \to \infty} \sqrt{\numS} \ratio_\numS^{*} \leq C^*.
\end{equation}
We now compute a lower bound on the limit of $\sqrt{\numS} \ratio_\numS^*$ that matches the upper bound derived in \eqref{eq:ub_asymptotic_ratio}. We only need to show that $\liminf_{\numS \to \infty} \sqrt{\numS} \sup_{F \in \mathcal{F}} \ratio_\numS \left( \pi^{\mathbf{k},\bm{\gamma}}_\numS, F \right) \geq C^*$. Consider an increasing function $\psi$ such that $\sqrt{\psi(\numS)} \sup_{F \in \mathcal{F}} \ratio_{\psi(\numS)} \left( \pi^{\mathbf{k},\bm{\gamma}}_{\psi(\numS)}, F \right)$ converges. 
By inequality $(b)$ in \eqref{eq:ub_lb_asymptotic}, one only need to show that there exists a sequence $\left( \mu_\numS \right)_{\numS \in \mathbb{N}}$ such that, 
$$
\limsup_{\numS \to \infty} \sqrt{\psi(\numS)} \left( \gamma_{\psi(\numS)} \cdot \ratio_{\psi(\numS)} \left( \pi^{k_{\psi(\numS)}}, \mathcal{B} \left( \mu_{\psi(\numS)} \right) \right) + (1-\gamma_{\psi(\numS)}) \cdot \ratio_{\psi(\numS)} \left( \pi^{k_{\psi(\numS)-1}}, \mathcal{B} \left( \mu_{\psi(\numS)} \right) \right) \right) \geq C^*.
$$
We prove that this lower bound holds by considering different scenarios for the sequence $\mathbf{k}$. 
Consider an increasing function $\tilde{\psi}$ such that
 $\frac{|k_{\tilde{\psi} \left(\psi(\numS) \right)} - q \tilde{\psi} \left(\psi(\numS) \right) |}{\sqrt{\tilde{\psi} \left(\psi(\numS) \right)}}$ converges to a limit $\ell$ in $\mathbb{R} \cup \{ \infty \}$.  
To ease notations , we let $f := \tilde{\psi} \circ \psi$ and $\mathbf{k}_f := \left(k_{f(\numS)}\right)_{\numS \in \mathbb{N}}$.

 \vspace{2mm}
\noindent \textit{Case 1: $\ell = \infty$.}
Note that,
\begin{equation*}
\lim_{\numS \to \infty} \frac{|k_{f(\numS)} - q f(\numS) |}{\sqrt{f(\numS)}} = \lim_{\numS \to \infty} \frac{|k_{f(\numS)} - 1 - q f(\numS) |}{\sqrt{f(\numS)}} = \infty.
\end{equation*}
Hence, by $(ii)$ in \Cref{lem:asymptotic_exact} we conclude that there exists a sequence $\mu_\numS$ such that
$$\lim_{\numS \to \infty} \sqrt{f(\numS)} \ratio_\numS \left( \pi^{k_{f(\numS)}}, \mathcal{B} \left(\mu_{f(\numS)} \right) \right) =  \lim_{\numS \to \infty} \sqrt{f(\numS)} \ratio_{f(\numS)} \left( \pi^{k_{f(\numS) - 1}}, \mathcal{B} \left( \mu_{f(\numS)} \right) \right) = \infty.$$
and so,
$$
\lim_{\numS \to \infty} \sqrt{f(\numS)} \left( \gamma_{f(\numS)} \cdot \ratio_{f(\numS)} \left( \pi^{k_{f(\numS)}}, \mathcal{B} \left( \mu_{f(\numS)} \right) \right) + (1-\gamma_{f(\numS)}) \cdot \ratio_{f(\numS)} \left( \pi^{k_{f(\numS)-1}}, \mathcal{B} \left( \mu_{f(\numS)} \right) \right) \right) = \infty.
$$

 \vspace{2mm}
\noindent \textit{Case 2: $\ell < \infty$.}
In this case, $(i)$ in \Cref{lem:asymptotic_exact} establishes the asymptotic behavior of the worst-case expected relative regret. We remark that the limit depends only on $\ell$. Therefore, 
\begin{align*}
\lim_{\numS \to \infty} \sqrt{f(\numS)} \ratio_\numS \left( \pi^{k_{f(\numS)}}, \mathcal{B} \left(\mu_{f(\numS)} \right) \right) &= \max 
\left[
\max_{\delta \geq 0 } H^{+}(\delta,\ell),
\max_{\delta \geq 0 } H^{-}(\delta,\ell) 
 \right]\\
\lim_{\numS \to \infty} \sqrt{f(\numS)} \ratio_{f(\numS)} \left( \pi^{k_{f(\numS) - 1}}, \mathcal{B} \left( \mu_{f(\numS)} \right) \right) &=\max 
\left[
\max_{\delta \geq 0 } H^{+}(\delta,\ell),
\max_{\delta \geq 0 } H^{-}(\delta,\ell) 
 \right].
\end{align*}
Let $\delta^{+} \in \argmax_{\delta \geq 0}  H^{+}(\delta,\ell)$ and $\delta^{-} \in \argmax_{\delta \geq 0}  H^{-}(\delta,\ell)$. Assume that $H^{+}(\delta^{+},\ell) \geq H^{-}(\delta^{-},\ell)$ (the other case is proved by a similar argument) and consider the sequence $\left( \mu_{\numS} \right)_{\numS \in \mathbb{N}}$ defined as, $\mu_\numS = 1-q - \frac{\delta^{+}}{\sqrt{\numS}}$ for every $\numS \geq 1$.  By \Cref{lem:asymptotic_exact} we conclude that, 
\begin{equation*}
\lim_{\numS \to \infty} \sqrt{f(\numS)} \ratio_\numS \left( \pi^{k_{f(\numS)}}, \mathcal{B} \left(\mu_{f(\numS)} \right) \right) = \lim_{\numS \to \infty} \sqrt{f(\numS)} \ratio_\numS \left( \pi^{k_{f(\numS)-1}}, \mathcal{B} \left(\mu_{f(\numS)} \right) \right) = H^{+} \left(\delta^{+},\ell \right).
\end{equation*}
Hence, 
$$
\lim_{\numS \to \infty} \sqrt{f(\numS)} \left( \gamma_{f(\numS)} \cdot \ratio_{f(\numS)} \left( \pi^{k_{f(\numS)}}, \mathcal{B} \left( \mu_{f(\numS)} \right) \right) + (1-\gamma_{f(\numS)}) \cdot \ratio_{f(\numS)} \left( \pi^{k_{f(\numS)-1}}, \mathcal{B} \left( \mu_{f(\numS)} \right) \right) \right) = H^{+} \left(\delta^{+},\ell \right).
$$
To conclude the proof, we need to show that $H^{+} \left(\delta^{+},\ell \right) \geq C^*$. This is a straightforward consequence of the definition of $\delta^{+}$ together wit the following lemma.
\begin{lemma}
\label{lem:SAA_opt}
For any $\ell \in \mathbb{R}$,
\begin{equation*}
\max 
\left[
\max_{\delta \geq 0 } H^{+}(\delta,\ell),
\max_{\delta \geq 0 } H^{-}(\delta,\ell) 
 \right] \geq C^*.
\end{equation*}
\end{lemma}
The proof is deferred to \Cref{apx:sec_body}.

We hence conclude that 
$$
\limsup_{\numS \to \infty} \sqrt{\psi(\numS)} \left( \gamma_{\psi(\numS)} \cdot \ratio_{\psi(\numS)} \left( \pi^{k_{\psi(\numS)}}, \mathcal{B} \left( \mu_{\psi(\numS)} \right) \right) + (1-\gamma_{\psi(\numS)}) \cdot \ratio_{\psi(\numS)} \left( \pi^{k_{\psi(\numS)-1}}, \mathcal{B} \left( \mu_{\psi(\numS)} \right) \right) \right) \geq C^*.
$$
Therefore we showed that,
\begin{equation*}
C^* = \lim_{\numS \to \infty} \sqrt{\numS} \sup_{F \in \mathcal{F}} \ratio_\numS \left( \pi_\numS^{\text{SAA}}, F \right) \geq \limsup_{\numS \to \infty} \ratio_{\numS}^* \geq \liminf_{\numS \to \infty} \ratio_{\numS}^* \geq C^*,
\end{equation*}
which concludes the proof.

\end{proof}

\setcounter{equation}{0}
\setcounter{proposition}{0}
\setcounter{lemma}{0}
\setcounter{theorem}{0}
\section{Proofs of Auxiliary Results}\label{apx:sec_body}

\begin{proof}[\textbf{Proof of \Cref{lem:cost}.}] Fix $F \in {\cal F}$. We next analyze $c_F(x)$ by decomposing the expected cost. We have
\bearn
c_F(x) &=& \mathbb{E}_{D \sim F} \left[{ b(D - x)^{+} +  h(x - D)^{+}}\right]\\
&=& h \int_{[0,x]} (x-y)dF(y) + b \int_{(x,\infty)} (y - x) dF(y) \\
&=& (b+h) \int_{[0,x]} (x-y)dF(y) +  b \int_{[0,\infty)} (y - x) dF(y)\\
&=&(b+h)\left( xF(x) - \int_{[0,x]} y dF(y) \right) + b(\mathbb{E}_F[D] - x) \\
&\stackrel{(a)}{=}&b(\mathbb{E}_F[D] - x) + (b+h) \int_{[0,x]} F(y)dy,
\eearn
where equation $(a)$ is a consequence of Riemann-Stieltjes integration by part. 
\end{proof}

\begin{proof}[\textbf{ Proof of \Cref{lem:OS_vs_bern}}]
For every $\alpha \in [0,1]$, let  $F_{\alpha} := \mathcal{B}( 1- \alpha)$. For every $r \in \{1,\ldots, \numS \}$ we have,
\begin{align*}
\mathcal{C} \left( \OS{r}, F_{\alpha}, \numS \right) &= \mathbb{E}_{\mathbf{D}_1^\numS \sim F_{\alpha}}\left [ \mathbb{E}_{x \sim {\OS{r} \left( \mathbf{D}_{1}^\numS \right)} } \left[  c_{F_{\alpha}}(x) \right] \right]\\
&= \mathbb{E}_{\mathbf{D}_1^\numS \sim F_{\alpha}}\left [  c_{F_{\alpha}}(D_{r:\numS}) \right] \\
&=    \mathbb{E}_{\mathbf{D}_1^\numS \sim F_{\alpha}}\left [ \alpha h   D_{r:\numS}   + (1-\alpha)b \left(1-   D_{r:\numS}  \right) \right]\\
&= (b+h) \left[ (\alpha - q) \mathbb{E}_{\mathbf{D}_1^\numS \sim F_{\alpha}}\left [  D_{r:\numS} \right] +q(1-\alpha) \right].
\end{align*}
When $\alpha \leq q$, we observe that $x^{*}_{F_{\alpha}} = 1$ and $\opt(F_{\alpha}) = \alpha h $. Therefore,
\begin{equation*}
\ratio_n(\OS{r}, \mathcal{B}( 1- \alpha)) = \frac{(q - \alpha ) \left( 1 -  \mathbb{E}_{\mathbf{D}_1^\numS \sim F_{\alpha}}\left [  D_{r:\numS} \right]\right)  }{(1-q)\alpha}.
\end{equation*}
We have
$$\mathbb{E}_{\mathbf{D}_1^\numS \sim F_{\alpha}}\left [  D_{r:\numS} \right] = \mathbb{P}_{\mathbf{D}_1^\numS \sim F_{\alpha}} \left( D_{r:\numS} =1 \right) = 1- \mathbb{P}_{\mathbf{D}_1^\numS \sim F_{\alpha}} \left( D_{r:\numS} =0 \right) = 1 - B_{r,\numS}(\alpha),$$
 where the last equality follows from the definition of the Bernstein polynomial. In turn,  we conclude that for $\alpha \in [0,q)$,
\begin{equation}
\label{eq:OS_down}
\ratio_n(\OS{r}, \mathcal{B}( 1- \alpha)) = \frac{(q - \alpha ) B_{r,\numS}(\alpha) }{(1-q)\alpha}.
\end{equation}
When $\alpha \in [q,1]$, we observe that $x^{*}_{F_{\alpha}} = 0$ , $\opt(F_{\alpha}) = (1-\alpha) b$ and 
we establish similarly that,
\begin{equation}
\label{eq:OS_up}
\ratio_n(\OS{r}, \mathcal{B}( 1- \alpha)) = \frac{(\alpha - q) \left(1 -  B_{r,\numS}(\alpha) \right) }{q (1 -\alpha)}.
\end{equation}
We conclude the proof by remarking that \eqref{eq:OS_down} and \eqref{eq:OS_up} imply that for every $\alpha \in [0,1]$,
\begin{equation*}
\ratio_{\numS} \left( \OS{r} , \mathcal{B} \left( 1 - \alpha \right)  \right) = \frac{ (1- B_{r,\numS}(\alpha))(\alpha - q) + q(1- \alpha) } { \min \left\{ (1-q) \alpha, q(1-\alpha) \right\}} - 1.
\end{equation*}
\end{proof}

\begin{proof}[\textbf{Proof of \Cref{lem:support}}]
Consider a policy $\pi \in \Pi_\numS$ such that for some $\hat{\mathbf{D}}_1^{\numS} \in \{0,1\}^{\numS}$, the support of $\pi \left( \hat{\mathbf{D}}_1^{\numS} \right)$ is not a subset of $[\hat{\mathbf{D}}_{1:\numS},\hat{\mathbf{D}}_{\numS:\numS}]$. We now construct a policy $\pi' \in \Pi_\numS$ such that the support of $\pi' \left( \hat{\mathbf{D}}_1^{\numS} \right)$ is a subset of $[\hat{\mathbf{D}}_{1:\numS},\hat{\mathbf{D}}_{\numS:\numS}]$ and which ensures a cost at least as good as the one incurred by $\pi$ against any Bernoulli distribution.

Assume first that $\hat{D}_{1:\numS} = 0$ and $\hat{D}_{\numS:\numS} = 1$.

Recall that $G^{\pi}_{\hat{\mathbf{D}}_1^{\numS}}$ is the cdf of the distribution $\pi \left( \hat{\mathbf{D}}_1^{\numS} \right)$. We define $\pi'$ such that for all $\mathbf{D}_1^{\numS} \in \{0,1\}^{\numS}$, if $\mathbf{D}_1^{\numS} \neq \hat{\mathbf{D}}_1^{\numS}$, we have $\pi' \left( \mathbf{D}_1^{\numS} \right) = \pi \left( \mathbf{D}_1^{\numS} \right)$ and we construct the cdf of  $\pi' \left( \hat{\mathbf{D}}_1^{\numS} \right)$ in order to ensure that the support is $[0,1]$ as follows.

\vspace{-2mm}
\begin{equation*}
G^{\pi'}_{\hat{\mathbf{D}}_1^{\numS}} (y) = \begin{cases}
0 \qquad &\text{if $y < 0 $}\\
G^{\pi}_{\hat{\mathbf{D}}_1^{\numS}} (y) \qquad &\text{if $y \in [0,1)$}\\
1 &\text{if $y \geq 1.$}
\end{cases}
\end{equation*}

\vspace{-1mm}

For any $\mu \in [0,1]$, let $F_\mu$ be the cdf of the Bernoulli distribution $\mathcal{B} \left(\mu \right)$. We have for any $x < 0$,

\vspace{-1mm}

\begin{equation*}
c_{F_\mu} (x)  = \mu \cdot b \cdot (1-x)^{+}  + (1-\mu) \cdot b \cdot (-x)^{+}
>  \mu \cdot b  = c_{F_\mu} (0).
\end{equation*}
Similarly, one can show that for any $x > 1$, $c_{F_\mu} (x) > c_{F_\mu} (1)$. 

Therefore, for every $\mu \in [0,1]$, the difference in costs between $\pi$ and $\pi'$ satisfies,
\fontsize{10.1pt}{10.1pt}\selectfont
\begin{align*}
\mathcal{C} \left(\pi, \mathcal{B} \left(\mu \right), \numS \right) - \mathcal{C} \left(\pi', \mathcal{B} \left(\mu \right), \numS \right) 
&\stackrel{(a)}{=} \mathbb{P}_{ \mathbf{D}_1^\numS \sim \mathcal{B} \left(\mu \right)} \left( \mathbf{D}_1^\numS = \hat{\mathbf{D}}_1^\numS \right) \left( \mathbb{E}_{x \sim { \pi \left( \hat{\mathbf{D}}_{1}^\numS \right)} } \left[ c_{F_{\mu} (x)} \right] - \mathbb{E}_{x \sim { \pi' \left( \hat{\mathbf{D}}_{1}^\numS \right)} } \left[ c_{F_{\mu} (x)} \right] \right) \\
&= \mathbb{P}_{ \mathbf{D}_1^\numS \sim \mathcal{B} \left(\mu \right)} \left( \mathbf{D}_1^\numS = \hat{\mathbf{D}}_1^\numS \right) \left( \int_\mathbb{R} c_{F_\mu}(y) \, dG^{\pi}_{\hat{\mathbf{D}}_1^{\numS}}(y) - \int_{[0,1]} c_{F_\mu}(y) \, dG^{\pi'}_{\hat{\mathbf{D}}_1^{\numS}}(y)\right) \\
&\stackrel{(b)}{\geq} \mathbb{P}_{ \mathbf{D}_1^\numS \sim \mathcal{B} \left(\mu \right)} \left( \mathbf{D}_1^\numS = \hat{\mathbf{D}}_1^\numS \right) \left( \int_{(-\infty;0)} c_{F_\mu}(0) \, dG^{\pi}_{\hat{\mathbf{D}}_1^{\numS}}(y)  \right . \\
&\quad  \left. + \int_{[0,1)} c_{F_\mu}(y) \, dG^{\pi}_{\hat{\mathbf{D}}_1^{\numS}}(y)  + \int_{[1,\infty)} c_{F_\mu}(1) \, dG^{\pi}_{\hat{\mathbf{D}}_1^{\numS}}(y) - \int_{[0,1]} c_{F_\mu}(y) \, dG^{\pi'}_{\hat{\mathbf{D}}_1^{\numS}}(y)\right) \\
& \stackrel{(c)}{=}0,
\end{align*}
\normalsize
where $(a)$ holds because for all $\mathbf{D}_1^{\numS} \in \{0,1\}^{\numS}$, such that $\mathbf{D}_1^{\numS} \neq \hat{\mathbf{D}}_1^{\numS}$, we have $\pi' \left( \mathbf{D}_1^{\numS} \right) = \pi \left( \mathbf{D}_1^{\numS} \right)$, $(b)$ follows from the fact that $c_{F_\mu}(x) < c_{F_\mu}(0) $ for $x < 0$, and $c_{F_\mu}(x) > c_{F_\mu}(1) $ for $x > 1$. $(c)$ is a consequence of the constructions of $ G^{\pi'}_{\hat{\mathbf{D}}_1^{\numS}}$. Hence, this shows that we weakly improve the cost of policy $\pi$ with the policy $\pi'$.

We now consider the case where $\hat{D}_{1:\numS} = \hat{D}_{\numS:\numS} = 0$. In that case, we have that $\mathcal{C} \left( \pi, \mathcal{B} \left( 0 \right), \numS \right) > 0$. Remarking that $\opt \left( \mathcal{B} \left( 0 \right) \right) = 0$, we conclude that $\ratio_\numS \left( \pi, \mathcal{B}\left( 0 \right) \right) = \infty$, which shows the strict sub-optimality of $\pi$.

A similar reasoning holds for $\hat{D}_{1:\numS} = \hat{D}_{\numS:\numS} = 1$.

We conclude the proof by repeating this process for every value of $\hat{\mathbf{D}}_{1}^\numS$.
\end{proof}

\begin{proof}[\textbf{ Proof of \Cref{lem:increasing_only}}]
It is clear that, 
\begin{equation*}
\inf_{\substack{\mathbf{e} \in [0,1]^{\numS+1}\\ e_0 =0, \, e_\numS =1} } \sup_{ \mu \in [0,1]} \ratio_{\numS} \left( \DS, \mathcal{B} \left( \mu \right) \right) \leq \inf_{\substack{ \mathbf{e} \in [0,1]^{\numS+1}\\ e_0 =0, \, e_\numS =1 \\ \left(e_i \right) \, \text{non-decreasing}} } \sup_{ \mu \in [0,1]} \ratio_{\numS} \left( \DS, \mathcal{B} \left( \mu \right) \right).
\end{equation*}
We now show that,
\begin{equation*}
\inf_{\substack{ \mathbf{e} \in [0,1]^{\numS+1}\\ e_0 =0, \, e_\numS =1} } \sup_{ \mu \in [0,1]} \ratio_{\numS} \left( \DS, \mathcal{B} \left( \mu \right) \right) \geq \inf_{\substack{ \mathbf{e} \in [0,1]^{\numS+1}\\ e_0 =0, \, e_\numS =1 \\ \left(e_i \right) \, \text{non-decreasing}} } \sup_{ \mu \in [0,1]} \ratio_{\numS} \left( \DS, \mathcal{B} \left( \mu \right) \right).
\end{equation*}

Consider $\mathbf{e} = \left( e_i \right)_{i \in \{0, \ldots \numS +1 \}}$ and assume that there exists $j \in \{0,\ldots \numS-1\}$ such that $e_j > e_{j+1}$. We consider the sequence $\mathbf{f} :=  \left( f_i \right)_{i \in \{0, \ldots \numS +1 \}}$ such that, 
\begin{equation*}
f_i = \begin{cases}
e_i &\qquad \text{if $ i \in \{0,\ldots, \numS\} \setminus \{ j, j+1 \}$}\\
\frac{1}{\frac{q}{\numS -j} + \frac{1-q}{j} } \left( \frac{q}{\numS -j} e_j + \frac{1-q}{j} e_{j+1} \right) &\qquad \text{if $ i \in \{ j, j+1 \}$.}
\end{cases}
\end{equation*}
We show that the cost of $\DSf$ is weakly lower than the one of $\DS$. We have, for any $\mu \in [0,1]$, 
\begin{align*}
\mathcal{C} \left(\DS, \mathcal{B} \left(\mu \right), \numS \right) - \mathcal{C} \left(\DSf, \mathcal{B} \left(\mu \right), \numS \right) &\stackrel{(a)}{=} \mathbb{P}_{ \mathbf{D}_1^\numS \sim \mathcal{B} \left(\mu \right)} \left( \sum_{i=1}^{\numS} D_{i} = j \right) \left( c_{F_{\mu}} \left( e_j \right) -  c_{F_{\mu}} \left( f_j \right) \right) \\
&\qquad + \mathbb{P}_{ \mathbf{D}_1^\numS \sim \mathcal{B} \left(\mu \right)} \left( \sum_{i=1}^{\numS} D_{i} = j+1 \right) \left( c_{F_{\mu}} \left( e_{j+1} \right) -  c_{F_{\mu}} \left( f_{j+1} \right) \right)\\
&= C_\mu \cdot \left(1 - \mu - q \right)  \cdot \left( \frac{1- \mu}{n-j}  \left(e_j - f_j \right)  + \frac{\mu}{j+1} \left(e_{j+1} - f_{j+1} \right)  \right),
\end{align*}
where $C_\mu = \left(b+h\right)  \cdot \mu^j \cdot \left( 1- \mu \right)^{\numS - j - 1} \frac{\numS!}{j! \, (\numS - j -1)! }$ and $(a)$ holds because $e_i = f_i$ for any $i$ different from $j$ and $j+1$. Note that $C_\mu \geq 0$ for all $\mu \in [0,1]$. Letting $\mathcal{L} \left(\mu \right) =  \frac{1- \mu}{n-j}  \left(e_j - f_j \right)  + \frac{\mu}{j+1} \left(e_{j+1} - f_{j+1} \right)$ for all $\mu \in [0,1]$, we only need to show that
\begin{align*}
\mathcal{L} \left(\mu \right)  &\geq 0  \qquad \text{for $\mu \in [0,1-q]$}\\
\mathcal{L} \left(\mu \right)  &\leq 0  \qquad \text{for $\mu \in [1-q,1]$.} 
\end{align*}
Note that $\mathcal{L}$ is a linear function of $\mu$, $\mathcal{L} \left(0 \right) = \frac{e_j - f_j}{\numS - j} \geq 0$ and $\mathcal{L} \left(1\right) = \frac{e_{j+1} - f_{j+1}}{j} \leq 0$, therefore, one only need to check $\mathcal{L} \left(1-q \right) = 0$. The latter equality holds by definition of $f_j$ and $f_{j+1}$.

We hence conclude that for any $\mu \in [0,1]$
\begin{equation*}
\mathcal{C} \left(\DS, \mathcal{B} \left(\mu \right), \numS \right) - \mathcal{C} \left(\DSf, \mathcal{B} \left(\mu \right), \numS \right) \geq 0.
\end{equation*}
By iterating this process for any $i$ such that $e_i > e_{i+1}$, we conclude that
\begin{equation*}
\inf_{\substack{ \mathbf{e} \in [0,1]^{\numS+1}\\ e_0 =0, \, e_\numS =1} } \sup_{ \mu \in [0,1]} \ratio_{\numS} \left( \DS, \mathcal{B} \left( \mu \right) \right) \geq \inf_{\substack{ \mathbf{e} \in [0,1]^{\numS+1}\\ e_0 =0, \, e_\numS =1 \\ \left(e_i \right) \, \text{non-decreasing}} } \sup_{ \mu \in [0,1]} \ratio_{\numS} \left( \DS, \mathcal{B} \left( \mu \right) \right).
\end{equation*}
This completes the proof.
\end{proof}

\begin{proof}[\textbf{Proof of \Cref{lem:posterior_structure}.}]

To ease notations, let $p^{+}_j(\delta) := \mathbb{P}(\mu = \mu^{+} \, \vert \, \sum_{i=1}^\numS D_i = j   )$ for $j \in \{1, \ldots, \numS \}$, and remark that,
\begin{equation*}
a_j =p^{+}_j(\delta) \frac{1 - q - \mu^{+} }{ (1- \mu^{+}) ( 1-q) } + \left(1 -  p^{+}_j(\delta) \right) \frac{1 - q - \mu^{-}}{ \mu^{-}q}.
\end{equation*}

\vspace{2mm}
\noindent \textbf{Step 1:} We first show that the sequence $(a_j)_{j \in \{0, \ldots, \numS \} }$ is decreasing.
By assumption, $0 < \mu^{-} <1- q < \mu^{+} <1$, therefore, $\frac{1- q - \mu^{-} }{ \mu^{-}q }  > 0 $ and $\frac{1 - q - \mu^{+}}{ (1 - \mu^{+})(1-q)}<  0$. Hence, to show that $(a_j)_{j \in \{0, \ldots, \numS \} }$ is decreasing, it is sufficient to show that $(p^{+}_j(\delta))_{j \in \{ 0, \ldots, \numS\} }$ is increasing. For every $j \in \{0,\ldots, \numS\}$  we have,
\begin{align*}
p^{+}_j(\delta) &= \frac{\mathbb{P}(\sum_{i=1}^\numS D_i = j  \vert \mu = \mu^{+}) \mathbb{P} \left( \mu = \mu^{+} \right)}{\mathbb{P}(\sum_{i=1}^\numS D_i = j \vert \mu = \mu^{-} )  \mathbb{P} \left( \mu = \mu^{-} \right)+  \mathbb{P}(\sum_{i=1}^\numS D_i = j \vert \mu = \mu^{+} )  \mathbb{P} \left( \mu = \mu^{+} \right) } \\
&= \frac{\delta {\numS \choose j} (\mu^{+})^j (1 - \mu^{+})^{\numS-j} }{(1-\delta){\numS \choose j} (1- \mu^{-})^{\numS - j} (\mu^{-})^{j} + \delta  {\numS \choose j} (1-\mu^{+})^{\numS-j} (\mu^{+})^j}.
\end{align*}
Therefore, for $j \in \{0,\ldots, \numS-1\}$, we have
\begin{equation*}
p^{+}_{j+1}(\delta) - p^{+}_j(\delta) = \delta  (\mu^{+})^j (1 - \mu^{+})^{\numS-j-1} \frac{ \mu^{+} d_j - (1-\mu^{+}) d_{j+1} }{d_j d_{j+1}},
\end{equation*}
where $d_j := \delta  (1 - \mu^{+})^{\numS - j} (\mu^{+})^{j} + (1-\delta)  (1- \mu^{-})^{\numS-j} (\mu^{-})^j \geq 0$. Furthermore,
\begin{align*}
\mu^{+} d_j - (1-\mu^{+}) d_{j+1} &= (1-\delta) \mu^{+}  (1-\mu^{-})^{\numS-j} (\mu^{-})^j - (1-\delta) (1-\mu^{+})  (1-\mu^{-})^{\numS-j-1} (\mu^{-})^{j+1} \\
&= (1-\delta) (1-\mu^{-})^{\numS-j-1} (\mu^{-})^j  \left(\mu^{+} - \mu^{-} \right) > 0.
\end{align*}
\noindent \textbf{Step 2:} Let $j \in \{0, \ldots, \numS \}$ and remark that $p_j^{+}(0) = 0$ whereas, $p_j^{+}(1) = 1$.
Hence by making explicit the dependency of $a_j(\delta)$ in $\delta$, we have that, $a_j(0) > 0$, $a_j(1)  < 0$ and $\delta \mapsto a_j(\delta)$ is continuous. Hence, by the intermediate value theorem, there exists $\delta'\in [0,1]$ such that $a_j(\delta') = 0$.
\end{proof}

\begin{proof}[ \textbf{Proof of \Cref{lem:asymptotic_exact}}]

To prove this lemma we characterize two regimes for the sequence of order statistic policies $\bm{\pi}^{\mathbf{r}}$. These two regimes depends on the value of the following limit
\begin{equation*}
\lim_{\numS \to \infty} \frac{| r_\numS - q \numS |}{\sqrt{\numS}} := \ell.
\end{equation*}
When $\ell < \infty$, we show that the worst-case expected regret scales at a rate $\Theta \left( \frac{1}{\sqrt{\numS}} \right)$ and we derive an exact closed form characterization of the limiting constant, along with a family of candidate hard cases that nature may select.

We also show that for $\ell = \infty$, the worst-case expected regret scales at a rate $\omega \left( \frac{1}{\sqrt{\numS}} \right)$, which naturally implies sub-optimality of this class of order statistic policies.

First remark that \Cref{thm:main_mixture} implies that,
\begin{equation*}
\sup_{F \in \mathcal{F}} \ratio_\numS \left( \pi^{\mathbf{r}}_\numS, F \right) = \sup_{\mu \in [0,1]} \ratio_\numS \left( \pi^{\mathbf{r}}_\numS, \mathcal{B} \left( \mu \right) \right). 
\end{equation*}
Furthermore, we derive from \Cref{lem:OS_vs_bern} a closed form expression of the relative-regret of an order statistic policy against a Bernoulli distribution. Namely, we have that
\begin{align}
\ratio_\numS \left( \pi^{\mathbf{r}}_\numS, \mathcal{B} \left( \mu \right) \right) &= \frac{\mu - (1-q)}{(1-q) (1-\mu)} \mathbb{P} \left( D_{r_\numS:\numS} = 0 \right) \qquad \text{if $ \mu \geq 1-q$} \label{eq:reg_1}\\
\ratio_\numS \left( \pi^{\mathbf{r}}_\numS, \mathcal{B} \left( \mu \right) \right) &= \frac{ 1-q - \mu }{q\mu} \mathbb{P} \left( D_{r_\numS:\numS} = 1 \right) \qquad \text{if $ \mu \leq 1-q$} \label{eq:reg_2}
\end{align}

\vspace{2mm}

\noindent \textbf{Step 1: $\ell < \infty$.}
We show that in this case,
\begin{equation*}
\lim_{\numS \to \infty} \sqrt{\numS} \cdot \sup_{F \in \mathcal{F}} \ratio_{\numS} \left( \pi^{\mathbf{r}}_{\numS}, F \right) =  \max 
\left[
\max_{\delta \geq 0 } H^{+}(\delta,\ell),
\max_{\delta \geq 0 } H^{-}(\delta,\ell) 
 \right],
\end{equation*}
where, $H^{+}(\delta,\ell) :=  \frac{\delta}{q(1-q)} \left(1 - \Phi \left( \frac{\delta + \ell}{\sqrt{q(1-q)}} \right) \right)$ and $H^{-}(\delta,\ell) := \frac{\delta}{q(1-q)} \left(1 - \Phi \left( \frac{\delta - \ell}{\sqrt{q(1-q)}} \right) \right)$. 

It follows from \Cref{thm:main_mixture} that it is sufficient to show that
\begin{equation}
\label{eq:asympt_form_main}
\lim_{\numS \to \infty} \sqrt{\numS} \cdot \sup_{\mu \in [0,1]} \ratio_{\numS} \left( \pi^{\mathbf{r}}_{\numS}, \mathcal{B} \left( \mu \right) \right) =  \max 
\left[
\max_{\delta \geq 0 } H^{+}(\delta,\ell),
\max_{\delta \geq 0 } H^{-}(\delta,\ell) 
 \right].
\end{equation}

We analyze the worst-case expected regret incurred by a sequence policy of order statistics policies against different regimes of Bernoulli distribution with means $(\mu_\numS)_{\numS \in \mathbb{N}}$. These regimes are defined by the limit of the following sequence,
\begin{equation*}
\lim_{\numS \to \infty} \sqrt{\numS} \left(\mu_\numS - (1 - q) \right) = \delta,
\end{equation*}
where $\delta \in \mathbb{R} \cup \{- \infty, \infty\}$. We show that if $\delta \in \{0, -\infty , \infty\}$, meaning that  $\mu_\numS$ does not converge to $1-q$ at a rate of $\frac{1}{\sqrt{\numS}}$, the asymptotic performance decreases at a faster rate than $\frac{1}{\sqrt{\numS}}$.

\vspace{2mm}
\noindent \textit{Case (a): $\delta = 0$}. 

In this cases, we have that $\mu_\numS \to 1-q$ as $\numS \to \infty$. Let $\mathcal{N}_{+} := \{ \numS \, \vert \, \mu_\numS \geq 1-q\}$ and $\mathcal{N}_{-} := \{ \numS \, \vert \, \mu_\numS < 1-q\}$ and consider the subsequences $(\mu_\numS')_{\numS \in \mathcal{N}^{+}}$ and $(\mu_\numS'')_{\numS \in \mathcal{N}^{-}}$. 
Since $\mathbb{N} = \mathcal{N}_{+} \cup \mathcal{N}_{-}$, one of these two sets must be infinite. If one of them is finite, one only need to compute the limit for the second subsequence. When both are infinite, we establish the existence of the limit of $\sqrt{\numS} \ratio_\numS \left( \pi^{\mathbf{r}}_\numS, \mathcal{B} \left( \mu_\numS\right) \right)$ and compute its value, by deriving the limit of the relative regret against both subsequence $(\mu_\numS')_{\numS \in \mathcal{N}^{+}}$ and $(\mu_\numS'')_{\numS \in \mathcal{N}^{-}}$ and prove that the limit coincides.

In particular, remark that \eqref{eq:reg_1} implies that, for every $\numS \in \mathcal{N}_{+}$,
\begin{equation*}
\sqrt{\numS} \ratio_\numS \left( \pi^{\mathbf{r}}_\numS, \mathcal{B} \left( \mu_\numS' \right) \right)  \leq \sqrt{\numS} \frac{\mu_\numS' - (1-q)}{(1-q) (1-\mu_\numS')} \stackrel{(a)}{\to} 0 \quad \text{as $\numS \to \infty$},
\end{equation*}
where $(a)$ holds because the numerator goes to $0$ as $\delta = 0$ and the denominator converges to a constant. Similarly, \eqref{eq:reg_2} implies that for every $\numS \in \mathcal{N}_{-}$,
\begin{equation*}
\sqrt{\numS} \ratio_\numS \left( \pi^{\mathbf{r}}_\numS, \mathcal{B} \left( \mu''_\numS \right) \right) \leq \sqrt{\numS} \frac{ 1-q - \mu''_\numS}{q\mu''_\numS} \to 0 \quad \text{as $\numS \to \infty$}.
\end{equation*}
Therefore,
\begin{equation*}
\lim_{\numS \to \infty} \sqrt{\numS} \ratio_\numS \left( \pi^{\mathbf{r}}_\numS, \mathcal{B} \left( \mu_\numS \right) \right) = 0.
\end{equation*}

\vspace{2mm}
\noindent \textit{Case (b): $\delta \in \{ -\infty, \infty\}$}. 

This case is handled by the following lemma proved at the end of \Cref{apx:sec_body}.
\begin{lemma}
\label{lem:asymptotic_fast}
If $r_\numS$ is such that, $ \lim_{\numS \to \infty} \frac{|r_\numS - q\numS|}{\sqrt{\numS}} = \ell < \infty$, 
then for any sequence $\mu_\numS$ such that
$$ \lim_{\numS \to \infty} \sqrt{\numS} | 1 -q - \mu_\numS|= \infty,$$
we have that
\begin{equation*}
\lim_{\numS \to \infty} \sqrt{\numS}  \ratio_\numS(\Pi_\numS^{\mathbf{r}},\mathcal{B}(\mu_\numS)) = 0.
\end{equation*}
\end{lemma}

\vspace{2mm}
\noindent \textit{Case (c): $\delta \in \mathbb{R} \setminus \{0 \} $}. 

Assume first that $\delta > 0$. This implies that, $ \mu_\numS - (1-q) \sim \frac{\delta}{\sqrt{\numS}}$. 

 Remark that for $\numS$ large enough, $\mu_\numS \geq 1-q$ and by \eqref{eq:reg_1}, 
\begin{equation*}
\sqrt{\numS} \ratio_\numS(\pi_{\numS}^{\mathbf{r}},\mathcal{B}(\mu_\numS)) = \sqrt{\numS}  \frac{ \mu_\numS - (1 - q)}{(1-q) (1- \mu_\numS)} \mathbb{P} \left( D_{r_\numS:\numS} = 0 \right).
\end{equation*}
By assumption, 
\begin{equation*}
\sqrt{\numS}  \frac{ \mu_\numS - (1 - q)}{(1-q) (1- \mu_\numS)} \to \frac{\delta}{(1-q)q} \quad \text{as $\numS \to \infty$}.
\end{equation*} 
Moreover,
\begin{align}
 \mathbb{P}(D_{r_\numS:\numS} = 0) 
 &= \mathbb{P} \left(\sum_{i=1}^\numS D_i \leq \numS - r_\numS \right)  \nonumber \\
&=  \mathbb{P} \left( \frac{1}{\sqrt{ \numS \mu_\numS (1-\mu_\numS)}} \sum_{i=1}^\numS \left(  D_i - \mu_\numS \right) -  \frac{n (1-\mu_\numS) - r_n}{\sqrt{\numS \mu_\numS(1-\mu_\numS)}}  \leq  0 \right) \label{eq:proba_rewrite} .
\end{align}
To conclude the proof, we show that 
\begin{equation}
\label{eq:weak_cv}
\frac{1}{\sqrt{ \numS \mu_\numS (1-\mu_\numS)}} \sum_{i=1}^\numS \left(  D_i - \mu_\numS   \right) -  \frac{n (1-\mu_\numS) - r_n}{\sqrt{ \numS \mu_\numS (1-\mu_\numS)}}\Longrightarrow \mathcal{N} \left(\frac{\delta + \ell}{\sqrt{q(1-q)}},1 \right) \quad \text{as $\numS \to \infty$}.
\end{equation}
Remark that,
\begin{equation*}
 \frac{n (1- \mu_\numS) - r_n}{\sqrt{ \numS \mu_\numS (1-\mu_\numS)}} = \frac{n (1-q- \mu_\numS) + q \numS - r_n}{\sqrt{ \numS \mu_\numS (1-\mu_\numS)}} \to - \frac{\delta + \ell}{\sqrt{q(1-q)}} \quad \text{as $\numS \to \infty$}.
\end{equation*}
Hence, it also converges in distribution. Moreover, for all $\numS \geq 1$ and $m \leq \numS$,  let 
$$Z_{\numS,m} := \frac{1}{\sqrt{\numS \mu_\numS (1-\mu_\numS)}} \left( D_m -  \mu_\numS \right).$$
We have that for every $1 \leq m \leq \numS$, 
\begin{equation*}
\mathbb{E} \left[ Z_{\numS,m}^2 \right] = \frac{\mu_\numS (1-\mu_\numS)^2 + (1-\mu_\numS)\mu_\numS^2}{\numS \mu_\numS (1-\mu_\numS)} = \frac{1}{\numS}.  
\end{equation*}
Hence, for any $\numS \geq 1$,
\begin{equation}
\label{eq:cond1_LF_CLT}
\sum_{m=1}^\numS \mathbb{E} \left[ Z_{\numS,m}^2 \right] = \frac{\mu_\numS (1- \mu_\numS)}{(1-q) q} \to 1 \quad \text{as $\numS \to \infty$}.
\end{equation}
Furthemore, let $\epsilon > 0$ then, for $\numS$ large enough, $|Z_{\numS,m}| < \epsilon$ almost surely for every $m \leq \numS$. Hence, 
\begin{equation}
\label{eq:cond2_LF_CLT}
\lim_{\numS \to \infty} \sum_{m=1}^\numS \mathbb{E} \left[ Z_{n,m}^2 ; \, |Z_{n,m}| > \epsilon \right]  = 0.
\end{equation}
By \eqref{eq:cond1_LF_CLT} and \eqref{eq:cond2_LF_CLT} and \Cref{thm:LF_CLT_thm} stated below, we conclude that,
\begin{equation}
\label{eq:normal_cv}
\frac{1}{\sqrt{ \numS \mu_\numS (1-\mu_\numS)}} \sum_{m=1}^\numS \left(  D_i - \mu_\numS   \right) = \sum_{m=1}^\numS Z_{\numS,m} \Longrightarrow \mathcal{N}(0,1) \quad \text{as $\numS \to \infty$}.
\end{equation}
\begin{theorem}[Linderberg-Feller theorem (see Theorem 3.4.10 in \cite{durrett2019probability})] 
\label{thm:LF_CLT_thm}
For each $\numS$, let $X_{n,m},$ $1 \leq m \leq n$, be independent random variables with $\mathbb{E}[ X_{n,m}] = 0$. Suppose,
\begin{enumerate}
\item $\sum_{m=1}^\numS \mathbb{E} \left[ X_{n,m}^2 \right]  \to \sigma^2 >0$,
\item For all $\epsilon > 0$, $\lim_{\numS \to \infty} \sum_{m=1}^\numS \mathbb{E} \left[ X_{n,m}^2 ; \, |X_{n,m}| > \epsilon \right] = 0$. 
\end{enumerate}
Then,
$\sum_{m=1}^\numS X_{n,m}  \Longrightarrow \mathcal{N}(0,\sigma^2)$
as $n \to \infty$.
\end{theorem}
We thus conclude from Slutsky's theorem that \eqref{eq:weak_cv} is satisfied  and for any $\delta' > 0$ such that $\mu_\numS - (1-q) \sim \frac{\delta}{\sqrt{\numS}}$,  we have,
\begin{equation*}
\sqrt{\numS} \ratio_\numS(\pi_{\numS}^{\mathbf{r}},\mathcal{B}(\mu_\numS)) \to \frac{\delta}{q(1-q)} \left(1 - \Phi \left(\frac{\ell + \delta}{\sqrt{q(1-q)}} \right) \right) \quad \text{as $\numS \to \infty$}.
\end{equation*}

We now consider the case where there exists $\delta > 0$, such that $ 1-q- \mu_\numS  \sim \frac{\delta}{\sqrt{\numS}}$. Note that \eqref{eq:reg_2} together with \eqref{eq:proba_rewrite} implies that
\begin{equation*}
\sqrt{\numS} \ratio_\numS(\pi_{\numS}^{\mathbf{r}},\mathcal{B}(\mu_\numS)) = \sqrt{\numS} \frac{1- q - \mu_\numS}{q \mu_\numS} \mathbb{P} \left( \frac{1}{\sqrt{ \numS q(1-q)}} \sum_{i=1}^\numS \left(  D_i - \mu_\numS   \right) -  \frac{n (1- \mu_\numS) - r_n}{\sqrt{\numS q(1-q)}}  >  0 \right).
\end{equation*}
We hence conclude by a similar argument that
$$
\mathbb{P} \left( \frac{1}{\sqrt{ \numS q(1-q)}} \sum_{i=1}^\numS \left(  D_i - \mu_\numS   \right) -  \frac{n (1- \mu_\numS) - r_n}{\sqrt{\numS q(1-q)}}  >  0 \right) \to 1 - \Phi \left( \frac{\delta - \ell}{\sqrt{q(1-q)}} \right) \quad \text{as $\numS \to \infty$}
$$
and
\begin{equation*}
\sqrt{\numS} \ratio_\numS(\pi_{\numS}^{\mathbf{r}},\mathcal{B}(\mu_\numS)) \to \frac{\delta}{q(1-q)} \left(1 - \Phi \left(\frac{\delta - \ell}{\sqrt{q(1-q)}} \right) \right) \quad \text{as $\numS \to \infty$}.
\end{equation*}

\vspace{2mm}
\noindent \textit{Conclusion of step 1.} We now prove that \eqref{eq:asympt_form_main} holds. 

For any $\numS \in \mathbb{N}$,  let $\mu_{\numS}^{*} := \argmax_{\mu \in [0,1]} \ratio_\numS(\pi_{\numS}^{ \mathbf{r}},\mathcal{B}(\mu))$. Note that this sequence is well defined because $\mu \mapsto \ratio_\numS(\pi_{\numS}^{ \mathbf{r}},\mathcal{B}(\mu))$ is continuous on a compact set. First remark that for any sequence $(\mu_\numS')_{\numS \in \mathbb{N}}$, if there exists a sequence $(\mu_\numS'')_{\numS \in \mathbb{N}}$ such that
$$ \lim_{\numS \to \infty} \ratio_\numS(\pi_{\numS}^{ \mathbf{r}},\mathcal{B}(\mu'_\numS)) <  \lim_{\numS \to \infty} \ratio_\numS(\pi_{\numS}^{ \mathbf{r}},\mathcal{B}(\mu''_\numS)),$$
then, one cannot have $\mu'_\numS \in \argmax_{\mu \in [0,1]} \ratio_\numS(\pi_{\numS}^{ \mathbf{r}},\mathcal{B}(\mu))$ for all $\numS \in \mathbb{N}$. We will use this property to characterize the asymptotic behavior of $(\mu_\numS^*)_{\numS \in \mathbb{N}}$.

We consider the sequence, $\left( \sqrt{\numS} |\mu^*_{\numS} - (1- q)| \right)_{\numS \in \mathbb{N}}$ and first assume that it is unbounded. Therefore, we can consider an increasing mapping $\psi$  from $\mathbb{N}$ to $\mathbb{N}$ such that the induced  subsequence $\left( \sqrt{\psi(\numS)} |\mu^*_{\psi(\numS)} - (1- q)| \right)_{\numS \in \mathbb{N}}$ converges to $\infty$. Moreover, let $\mathbf{r}_\psi$ denote the subsequence defined as $ \left( r_{\psi(\numS)} \right)_{\numS \in \mathbb{N}}$.
By case $(b)$ defined above, this implies that
\begin{equation*}
\lim_{\numS \to \infty} \sqrt{\numS}  \ratio_\numS(\Pi_\numS^{\mathbf{r}_{\psi}},\mathcal{B}(\mu^*_{\psi(\numS)} )) = 0.
\end{equation*}
In turn, consider the sequence $ \left( \hat{\mu}_\psi(\numS) \right)_{\numS \in \mathbb{N}}$ such that, $ \sqrt{\psi(\numS)} |\hat{\mu}_{\psi(\numS)} - (1- q)| \to \delta \in \mathbb{R}^*$. It follows from case $(c)$ that  
\begin{equation*}
\lim_{\numS \to \infty} \sqrt{\numS}  \ratio_\numS(\Pi_\numS^{\mathbf{r}_{\psi}},\mathcal{B}(\hat{\mu}_{\psi(\numS)}  )) > 0,
\end{equation*}
which contradicts the optimality of $\mu^*_{\psi(\numS)}$ for $\numS$ large enough.

Hence, $\left( \sqrt{\numS} |\mu^*_{\numS} - (1- q)| \right)_{\numS \in \mathbb{N}}$ is bounded. By Bolzano Weirestrass, we can consider a subsequence that converges. Remark that, if the limit of this subsequence is $0$, we obtain a contradiction by the same argument using case $(a)$ and case $(c)$. We thus conclude that any subsequence has to satisfy,
$$\lim_{\numS \to \infty} \sqrt{\psi(\numS)} \left( \mu^*_{\psi(\numS)} - 1 -q  \right) = \delta,$$
with $\delta \in \mathbb{R} \setminus \{0 \}$. Case $(c)$ implies that,
\begin{align*}
\lim_{\numS \to \infty} \sqrt{\numS}  \ratio_\numS(\Pi_\numS^{\mathbf{r}_{\psi}},\mathcal{B}(\mu^*_{\psi(\numS)} )) = H^+ \left( \delta, \ell \right) \qquad \text{if $\delta > 0$} \\
\lim_{\numS \to \infty} \sqrt{\numS}  \ratio_\numS(\Pi_\numS^{\mathbf{r}_{\psi}},\mathcal{B}(\mu^*_{\psi(\numS)} )) = H^{-} \left( - \delta, \ell \right) \qquad \text{if $\delta < 0$},
\end{align*}
where $H^{+}(\delta,\ell) :=  \frac{\delta}{q(1-q)} \left(1 - \Phi \left( \frac{\delta + \ell}{\sqrt{q(1-q)}} \right) \right)$ and $H^{-}(\delta,\ell) := \frac{\delta}{q(1-q)} \left(1 - \Phi \left( \frac{\delta - \ell}{\sqrt{q(1-q)}} \right) \right)$.
Furthermore, remark that if $\delta >0$, we must have $H^{+} \left(\delta, \ell \right) = \max \left( \max_{\delta' \geq 0} H^{+} \left(\delta', \ell \right), \max_{\delta' \geq 0} H^{-} \left( \delta', \ell \right) \right) $ otherwise, we can construct a sequence that strictly improves the limit. A similar result holds if $\delta <0$. 

Therefore,  for any converging subsequence $\left( \sqrt{\psi(\numS)} |\mu^*_{\psi(\numS)} - (1- q)| \right)_{\numS \in \mathbb{N}}$ we have that 
$$
\lim_{\numS \to \infty} \sqrt{\numS}  \ratio_\numS(\Pi_\numS^{\mathbf{r}_{\psi}},\mathcal{B}(\mu^*_{\psi(\numS)} )) = \max 
\left[
\max_{\delta \geq 0 } H^{+}(\delta,\ell'),
\max_{\delta \geq 0 } H^{-}(\delta,\ell')
\right],
$$
which implies that,
$$
\lim_{\numS \to \infty} \sup_{F \in \mathcal{F}} \sqrt{\numS} \ratio_\numS(\pi_{\numS}^{\mathbf{r}}, F) = \lim_{\numS \to \infty} \sqrt{\numS}  \ratio_\numS(\Pi_\numS^{\mathbf{r}},\mathcal{B}(\mu^*_{\numS} )) = \max 
\left[
\max_{\delta \geq 0 } H^{+}(\delta,\ell'),
\max_{\delta \geq 0 } H^{-}(\delta,\ell')
\right].
$$

\noindent \textbf{Step 2: $\ell = \infty$.}
We show that in this case,
\begin{equation*}
\lim_{\numS \to \infty} \sup_{F \in \mathcal{F}} \ratio_\numS \left( \pi^{\mathbf{r}}_\numS, F \right) = \infty.
\end{equation*}
Remark that, it is sufficient to construct a sequence $\mu_\numS$ such that,
\begin{equation*}
\lim_{\numS \to \infty} \ratio_\numS \left( \pi^{\mathbf{r}}_\numS, \mathcal{B} \left( \mu_\numS \right) \right) = \infty
\end{equation*}

We  assume that for $n$ large enough, $r_\numS \leq q \numS$  (the case $r_\numS \geq q \numS$  is similar).

Let's consider the sequence of means such that for each $n \geq 1$, $\mu_\numS=  1- \frac{r_\numS}{\numS}$. For $\numS$ large enough we have that $r_\numS \leq q \numS$ which implies that $\mu_\numS \geq 1-q$.  Therefore, for $n$ large enough the expected relative regret of $\pi^{\mathbf{r}}_\numS$ against $\mathcal{B} \left(\mu_\numS \right)$ is given by \eqref{eq:reg_1}.
Remark that,
\begin{align}
\mathbb{P} \left( D_{r_\numS:\numS} = 0 \right) &= \mathbb{P}\left( \sum_{i=1}^\numS D_i \leq \numS - r_\numS \right) \nonumber \\
&\stackrel{(a)}{=} \mathbb{P}\left( \sum_{i=1}^\numS \left( D_i  - \mu_\numS \right) \leq 0 \right) \nonumber \\
&= \mathbb{P}\left(\frac{1}{\sqrt{\numS \mu_\numS \left( 1- \mu_\numS \right) }} \sum_{i=1}^\numS \left(  D_i  - \mu_\numS \right) \leq 0 \right) \stackrel{(b)}{\to} \frac{1}{2} \quad \text{as $\numS \to \infty$}, \label{eq:eq:LF_centered_demand}
\end{align}
where $(a)$ holds as  $\mu_\numS = 1 - \frac{r_\numS}{\numS}$ and $(b)$ follows from \eqref{eq:normal_cv}.
We further remark that,
\begin{equation}
\label{eq:cv_first_term}
\sqrt{\numS}\frac{\mu_\numS - (1 -q )}{(1-q) (1- \mu_\numS)}  \geq \sqrt{\numS}\frac{\mu_\numS - (1 -q )}{(1-q)q} =  \frac{(q{\numS}-r_\numS)}{(1-q)q \sqrt{\numS}} \to \infty \quad \text{as $\numS \to \infty$},
\end{equation}
where the limit holds because $\ell = \infty$. Finally, \eqref{eq:normal_cv} and  \eqref{eq:cv_first_term} imply that 
\begin{equation*}
\lim_{\numS \to \infty} \sqrt{\numS} \ratio_\numS  \left( \OS{r_\numS},B(\mu_\numS) \right) = \infty.
\end{equation*}

\end{proof}

\begin{proof}[\textbf{Proof of \Cref{lem:SAA_opt}}]
Define the function $H$ from $\mathbb{R}$ to $\mathbb{R}$ as, 
\begin{equation*}
H : \ell \longmapsto  \max 
\left[
\max_{\delta \geq 0 } H^{+}(\delta,\ell),
\max_{\delta \geq 0 } H^{-}(\delta,\ell)
\right].
\end{equation*}
We show that $H$ is an even function increasing on $[0, \infty)$.

We remark that for any $\ell \in \mathbb{R}$, 
\begin{equation*}
H^{+}(\delta,\ell) = H^{-}(\delta,-\ell),
\end{equation*}
which shows that $H$ is even. To show that it is increasing on $[0,\infty)$, consider $\ell \geq 0$. 
As $\Phi$ is increasing, we have for any $\delta \geq 0$ that $\Phi \left(  \frac{\delta + \ell}{\sqrt{q (1-q)}} \right) \geq \Phi \left(  \frac{\delta - \ell}{\sqrt{q (1-q)}} \right)$ and , $H(\ell) = \max_{\delta \geq 0}  H^{-}(\delta,\ell)$.
Let $\delta^{*} \in \argmax_{\delta \geq 0} H^{-}(\delta,\ell)$.
For any $\ell' \geq \ell$, we note that $\Phi \left( \frac{\delta^{*} - \ell}{\sqrt{q(1-q)}} \right)  \geq \Phi \left( \frac{\delta^{*} - \ell'}{\sqrt{q(1-q)}} \right)$.
Therefore,
$H^{-}(\delta^{*},\ell) \leq H^{-}(\delta^{*},\ell')$. We conclude that,
\begin{equation*}
H(\ell) =  H^{-}(\delta^{*},\ell) \leq  H^{-}(\delta^{*},\ell') \leq H(\ell').
\end{equation*}

Since $H$ is an even increasing function, its minimum is achieved at $0$ and we conclude that, for every $\ell \in \mathbb{R}$
\begin{equation*}
\max 
\left[
\max_{\delta \geq 0 } H^{+}(\delta,\ell),
\max_{\delta \geq 0 } H^{-}(\delta,\ell)
\right] = H(\ell) \geq H(0) = C^*.
\end{equation*}
\end{proof}

\begin{proof}[\textbf{Proof of \Cref{lem:asymptotic_fast}}]
For the sake of simple notations, consider the sequence $\alpha_\numS = 1- \mu_\numS$ and consider a converging sequence $(\beta_n)_{n \in \mathbb{N}}$ that  is a subsequence of $(\alpha_n)_{n \in \mathbb{N}}$. Let $\beta$ denotes its limit. We assume that for $n$ large enough $\beta_n \in [0,q]$ therefore, $\beta \in [0,q]$. (The case where  $\beta_n \in [q,1]$ is treated by similar arguments). 

For every $r \in \{1, \ldots, \numS\}$,  and $\alpha \in [0,1]$, let
\begin{equation*}
\phi_{r}^{\numS}(\alpha) = \ratio_n \left( \OS{r}, \mathcal{B}(1-\alpha) \right)
\end{equation*}
and remark that since $\beta_n \leq q$ for $n$ large enough, we have by \Cref{lem:OS_vs_bern} that 
\begin{equation*}
\phi_{r}^{\numS}(\beta_\numS) =  \sqrt{\numS}\frac{(q-\beta_n)}{(1-q) \beta_\numS}  B_{r_n,\numS}(\beta_\numS).
\end{equation*}

\vspace{2mm}
\textbf{Case 1: $\beta = 0$.} 
By Taylor expansion, we obtain that for every $n \geq 1$, there exists $\xi_n \in [0, \beta_n]$ such that,
\begin{equation}
\label{eq:taylor_exp_0}
\phi_{r_n}^{\numS}(\beta_n) = \phi_{r_n}^{\numS}(0) + \phi_{r_n}^{'\numS}(\xi_n) \beta_n  .
\end{equation}
Since $B_{r_n,\numS}'(\alpha) = \numS b_{r_n-1,\numS-1}(\alpha)$, we obtain that for $\alpha \in [0,q]$,
\begin{equation*}
\phi_{r_n}^{'\numS}(\alpha) = \frac{\numS \left(q - \alpha \right) \alpha b_{r_n-1,\numS-1}(\alpha) - q B_{r_n,\numS} (\alpha)} {(1-q)\alpha^2} \leq \frac{q \numS b_{r_n-1,\numS-1}(\alpha) }{(1-q)\alpha}.
\end{equation*}
Furthermore, 
\begin{align}
\phi_{r_n}^{'\numS}(\xi_n) \beta_n &\leq \frac{q \numS}{1-q} {\numS - 1 \choose r_n -1 } \xi_n^{r_n-2}(1-\xi_n)^{n-r_n} \beta_n \nonumber\\
&\leq \frac{q \numS}{1-q} {\numS - 1 \choose r_n -1 } \xi_n^{r_n-2} \beta_n \stackrel{(a)}{\leq} \frac{q \numS}{1-q} {\numS - 1 \choose r_n -1 } \beta_n^{r_n-1}
\label{eq:bdd_derivative},
\end{align}
where $(a)$ holds as $\xi_\numS \leq \beta_\numS$ and $r_\numS \geq 2$ for $\numS$ large enough. The latter holds because $q >0$ and $\lim_{\numS \to \infty} \frac{|q {\numS}-r_\numS|}{\sqrt{\numS}} < \infty$ which implies that $r_\numS \to \infty$ as $\numS$ gets large.

Remarking that $\phi_{r_n}^{\numS}(0) = 0$ and that ${\numS - 1 \choose r_n -1 } \sim q {\numS  \choose r_n }$, we conclude from \eqref{eq:taylor_exp_0} and \eqref{eq:bdd_derivative} that it is sufficient to show that, $\numS^{3/2} {\numS \choose r_n} \beta_n^{r_n-1} \to 0$. Note that, $r_\numS = q\numS + O(\sqrt{\numS} )$ therefore,
\begin{align*}
{\numS \choose r_n} \sim \sqrt{\frac{\numS}{2\pi r_\numS \left(\numS- r_\numS \right)}} \frac{\numS^\numS}{r_\numS^{r_\numS} \left(\numS- r_\numS \right)^{ \numS- r_\numS}} \sim \frac{1}{\sqrt{ 2\pi q(1-q)n}} \frac{\numS^\numS}{r_\numS^{r_\numS} \left(\numS- r_\numS \right)^{ \numS- r_\numS}}.
\end{align*}
Furthermore, we have that, $ \frac{r_\numS}{\numS} \to q$. Therefore,
\begin{align*}
\frac{\numS^\numS}{r_\numS^{r_\numS} \left(\numS- r_\numS \right)^{ \numS- r_\numS}} = \exp \left( r_\numS \log \left( \frac{\numS}{r_\numS} \right) + \left(\numS- r_\numS \right) \log \left( \frac{\numS}{\numS- r_\numS} \right)  \right)
\stackrel{(a)}{=} \exp \left( O(n) \right),
\end{align*}
where $(a)$ holds because $\frac{\numS}{r_\numS} \to \frac{1}{q}$ and $\frac{\numS}{\numS - r_\numS} \to \frac{1}{1-q}$.
Moreover, note that $n = o \left( r_n \log( \beta_\numS) \right)$ because $\beta_{\numS} \to 0$. As a consequence,
\begin{align*}
\numS^{3/2} {\numS \choose r_n} \beta_n^{r_n-1} &\sim \frac{1}{\sqrt{2 \pi q(1-q)}} \cdot  \numS \cdot \exp \left(O(\numS) \right) \cdot \beta_n^{r_n-1}\\
&= \frac{1}{\sqrt{2 \pi q(1-q)}} \cdot \numS \cdot \exp \left( (r_n - 1) \log( \beta_\numS) + o( r_n \log( \beta_\numS) \right) \to 0,
\end{align*}
where the limit holds because $(r_n - 1) \log( \beta_\numS) \to - \infty$ and $\numS = \exp(o \left( r_n \log( \beta_\numS) \right))$.

\vspace{2mm}
\textbf{Case 2: $\beta > 0$.} 
Note that, by assumption on the sequence $(\alpha_n)_{\numS \in \mathbb{N}}$, we have that for $\numS$ large enough, $ \beta_\numS \leq \frac{r_\numS}{\numS}$. Indeed, if $\beta_\numS > \frac{r_\numS}{\numS}$, we would have that $\sqrt{\numS} (q - \beta_\numS) \leq \sqrt{\numS}( q- \frac{r_\numS}{\numS})$ which, in turn, would contradict that  $ \lim_{\numS \to \infty} \frac{| q\numS - r_\numS| }{\sqrt{\numS}} < \infty$ since  $\lim_{\numS \to \infty} \sqrt{\numS} |q - \beta_n| = \infty$. Therefore, for $\numS$ large enough, we have,
\begin{align*}
\mathbb{P} \left( D_{r_\numS:\numS} = 0 \right) &= \mathbb{P}\left(  \sum_{i=1}^\numS  D_i < \numS (1 - \beta_\numS) - r_\numS + \numS \beta_\numS  \right)\\
&= \mathbb{P}\left(  \sum_{i=1}^\numS  D_i < \numS \mathbb{E} \left[D \right] - \left( r_\numS - \numS \beta_\numS\right)  \right)
\stackrel{(a)}{\leq} e^{-2 (\beta_\numS- \frac{r_\numS}{\numS})^2 \numS}
= e^{-2 \left( \sqrt{\numS}  ( \beta_\numS- q)  + \frac{q \numS - r_\numS}{\sqrt{\numS}}  \right)^2 },
\end{align*}
where $(a)$ follows from  Hoeffding inequality for bounded random variables (Theorem 2 in \cite{hoeffding1994probability}).
 Moreover, note that $\left( \sqrt{\numS}  ( \beta_\numS- q)  + \frac{q \numS - r_\numS}{\sqrt{\numS}}  \right)^2 = \left( \sqrt{\numS}  ( \beta_\numS- q) \right)^2 \left(1 + o(1) \right)$.  Hence,
\begin{align*}
\sqrt{\numS} \ratio_\numS(\pi_{\numS}^{ \mathbf{r}},\mathcal{B}(1-\beta_\numS)) \sim \sqrt{\numS}\frac{(q-\beta_\numS)}{(1-q) \beta} \mathbb{P} \left( D_{r_\numS:\numS} = 0 \right)
\leq  \sqrt{\numS}\frac{(q-\beta_\numS)}{(1-q) \beta}  e^{-2 \left( \sqrt{\numS}  ( \beta_\numS- q) \right)^2 \left(1 + o(1) \right)} \to 0,
\end{align*}
where the limit holds because $\sqrt{\numS} | \beta_\numS - q| \to \infty$ as $\numS \to \infty$. 
Therefore, for any converging subsequence  $(\beta_{\numS})_{\numS \in \mathbb{N}}$ of the sequence $(\alpha_{\numS})_{\numS \in \mathbb{N}}$, we have$\sqrt{\numS} \ratio_\numS(\pi_{\numS}^{ \mathbf{r}},\mathcal{B}(1-\beta_\numS)) \to 0$. This implies the same result on the sequence $(\alpha_{\numS})_{\numS \in \mathbb{N}}$ itself.
\end{proof}

\setcounter{equation}{0}
\setcounter{proposition}{0}
\setcounter{lemma}{0}
\setcounter{theorem}{0}

\section{Additional Results and their Proofs} \label{app:add_aux}

\begin{lemma}
\label{lem:epigraph_form}
Fix $\numS \ge 1$ and  $\pi \in \Pi_\numS$. Then,
 problem \eqref{eq:ratio} of finding the worst case performance for the policy $\pi$ is equivalent to
\begin{subequations}
\nonumber
\begin{alignat*}{2}
\sup_{F \in \mathcal{F}} \ratio_\numS(\pi, F) \:=\: &\! \inf_{z \in \mathbb{R}_{+} }        &\qquad& z \\
&\text{s.t.} &      & \mathcal{C} \left( \pi, F, \numS \right) \leq (z+1) \opt(F) \qquad \forall F \in \mathcal{F}, 
\end{alignat*}
\end{subequations}
in the sense that both problems admit the same value.
\end{lemma}

\begin{proof}[ \textbf{Proof of \Cref{lem:epigraph_form}}]
Fix $\pi \in \Pi_\numS$. We have that $\sup_{F \in \mathcal{F}} \ratio_\numS(\pi, F) \ge 0$ and
\begin{subequations}
\label{eq:whole_class_pb}
\begin{alignat}{2}
\sup_{F \in \mathcal{F}} \ratio_\numS(\pi, F) \:=\: &\! \inf_{z \in \mathbb{R}_{+} }        &\qquad& z \\
&\text{s.t.} &      & \ratio_\numS(\pi, F) \leq z \qquad \forall F \in \mathcal{F}. \label{eq:constraint_epi}
\end{alignat}
\end{subequations}
We first show that for every $F \in \mathcal{F}$ and $z \in \mathbb{R}_{+}$, the following equivalence holds.
$$
 \ratio_\numS(\pi, F) \leq z \quad \mbox{if and only if} \quad \mathcal{C} \left( \pi, F, \numS \right) \leq (z+1) \opt(F) .
$$
Recall that, for every $F \in \mathcal{F}$, we have $\ratio_\numS(\pi, F) = \frac{\mathcal{C} \left( \pi, F, \numS \right)}{\opt(F)} - 1$.
Moreover, $\opt(F) \geq 0$ for all $F \in \mathcal{F}$.  When $\opt(F) > 0$, the equivalence holds trivially. If $\opt(F) = 0$, we consider two cases. If $\mathcal{C} \left( \pi, F, \numS \right) > 0$, we remark that both inequalities do not hold for any $z \in \mathbb{R}_{+}$. If $\mathcal{C} \left( \pi, F, \numS \right) = 0$, recall that by convention we set  $\ratio_\numS(\pi, F) = 0$. Hence, both inequalities are satisfied for all $z \in \mathbb{R}_{+}$.

As a consequence, problem \eqref{eq:whole_class_pb} is equivalent to
\begin{subequations}
\begin{alignat*}{2}
\sup_{F \in \mathcal{F}} \ratio_\numS(\pi, F) \:=\: &\! \inf_{z \in \mathbb{R}_{+} }        &\qquad& z \\
&\text{s.t.} &      & \mathcal{C} \left( \pi, F, \numS \right) \leq (z+1) \opt(F) \qquad \forall F \in \mathcal{F}. 
\end{alignat*}
\end{subequations}
\end{proof}

\begin{lemma}
\label{lem:single_cond}
For every $\numS$, the following two conditions cannot hold simultaneously.
\begin{eqnarray*}
\sup_{\mu \in [0,1-q]} \ratio_{\numS} \left(\OS{1},\mathcal{B} \left( \mu \right) \right) &>&\sup_{\mu \in [1-q,1]} \ratio_{\numS} \left(\OS{1},\mathcal{B} \left( \mu \right) \right)\\
\sup_{\mu \in [0,1-q]} \ratio_{\numS} \left(\OS{\numS},\mathcal{B} \left( \mu \right) \right) &<&\sup_{\mu \in [1-q,1]} \ratio_{\numS} \left(\OS{\numS},\mathcal{B} \left( \mu \right) \right).
\end{eqnarray*}
\end{lemma}

\begin{proof}[\textbf{Proof of \Cref{lem:single_cond}}]
Assume by contradiction that there exists $\numS$ such that,
\begin{eqnarray*}
\sup_{\mu \in [0,1-q]} \ratio_{\numS} \left(\OS{1},\mathcal{B} \left( \mu \right) \right) &>&\sup_{\mu \in [1-q,1]} \ratio_{\numS} \left(\OS{1},\mathcal{B} \left( \mu \right) \right)\\
\sup_{\mu \in [0,1-q]} \ratio_{\numS} \left(\OS{\numS},\mathcal{B} \left( \mu \right) \right) &<&\sup_{\mu \in [1-q,1]} \ratio_{\numS} \left(\OS{\numS},\mathcal{B} \left( \mu \right) \right).
\end{eqnarray*}
We have that,
\begin{align*}
\sup_{\mu \in [0,1-q]} \ratio_{\numS} \left(\OS{1},\mathcal{B} \left( \mu \right) \right)  
&\stackrel{(a)}{\leq} \sup_{\mu \in [0,1-q]} \ratio_{\numS} \left(\OS{\numS},\mathcal{B} \left( \mu \right) \right)\\ 
&\stackrel{(b)}{<} \sup_{\mu \in [1-q,1]} \ratio_{\numS} \left(\OS{\numS},\mathcal{B} \left( \mu \right) \right) \stackrel{(c)}{\leq} \sup_{\mu \in [1-q,1]} \ratio_{\numS} \left(\OS{1},\mathcal{B} \left( \mu \right) \right),
\end{align*}
where $(a)$ and $(c)$ follow from \Cref{lem:monotonic_phi}, and $(b)$ holds by assumption. We therefore obtain a contradiction.
\end{proof}

\begin{lemma}
For $\numS \geq \frac{2}{\min(q,(1-q))^2}$,
\label{lem:crossing}
\begin{eqnarray*}
\sup_{\mu \in [0,1-q]} \ratio_{\numS} \left(\OS{1},\mathcal{B} \left( \mu \right) \right) &\leq&\sup_{\mu \in [1-q,1]} \ratio_{\numS} \left(\OS{1},\mathcal{B} \left( \mu \right) \right)\\
\sup_{\mu \in [0,1-q]} \ratio_{\numS} \left(\OS{\numS},\mathcal{B} \left( \mu \right) \right) &\geq&\sup_{\mu \in [1-q,1]} \ratio_{\numS} \left(\OS{\numS},\mathcal{B} \left( \mu \right) \right).
\end{eqnarray*}
\end{lemma}

\begin{proof}[\textbf{Proof of \Cref{lem:crossing}.}]
For each $r \in \{1, \ldots, \numS\}$ we define $\phi^{\numS}_{r}$ such that for every $\alpha \in [0,1]$,
\begin{equation*}
\phi_{r}^{\numS}(\alpha) = \ratio_n \left( \OS{r}, \mathcal{B}(1-\alpha) \right). 
\end{equation*}

We first show that,
\begin{equation}
\label{eq:order_stat_1_imbalance}
\sup_{\alpha \leq q} \phi^{\numS}_{1}(\alpha) \geq \sup_{\alpha \geq q} \phi^{\numS}_{1}(\alpha).
\end{equation}
Note that by \Cref{lem:OS_vs_bern} we have for all $r \in \{1,\ldots, \numS \}$,
\bearn
\phi^{\numS}_{r}(\alpha)  &=&  \begin{cases}
\frac{(q - \alpha ) B_{r,\numS}(\alpha) }{(1-q)\alpha} \qquad \text{if $\alpha \in [0,q],$}\\
\frac{(\alpha - q) \left( 1 -  B_{r,\numS}( \alpha) \right) }{(1-\alpha) q} \qquad \text{if $\alpha \in [q,1]$.}
\end{cases}
\eearn
Observe that by definition, 
$B_{1,\numS}(\alpha) = \sum_{j=1}^\numS {\numS \choose j} \alpha^{j} (1- \alpha)^{\numS-j} = 1 - (1-\alpha)^{\numS}$.
Hence, for $\alpha \geq q$
\begin{align*}
\phi_1^{\numS}(\alpha) &= \frac{(\alpha - q)( 1 - \alpha)^\numS}{q(1-\alpha)} \leq \frac{(1-q)^\numS}{q}. 
\end{align*}
Moreover, $\numS \geq \frac{2}{\min(q,(1-q))^2}$ implies that $\frac{1}{\numS} \leq  \min(q,(1-q))^2 \leq q^2 \leq q$. Therefore, we have that,
\begin{align*}
\sup_{\alpha \leq q} \phi^{\numS}_{1}(\alpha) \geq \phi^{\numS}_{1} \left( \frac{1}{\numS} \right) = \frac{(q - \numS^{-1})(1- (1- \numS^{-1})^{\numS})}{(1-q)\numS^{-1}} \stackrel{(a)}{\geq} \frac{q(1 - q)(1- e^{-1})}{(1-q)q^2} =   \frac{1- e^{-1}}{q}
\end{align*}
where inequality $(a)$ holds because $\numS^{-1} \leq q^2$ and $\left(1 - \frac{1}{n} \right)^n \leq e^{-1}$.
When $q \geq \frac{1}{2}$, we remark that since $\numS \geq 4$, \eqref{eq:order_stat_1_imbalance} must hold because,
\begin{equation*}
1-e^{-1} \geq \frac{1}{2^\numS} \geq (1-q)^\numS.
\end{equation*}
To prove equation \eqref{eq:order_stat_1_imbalance} for $q \leq \frac{1}{2}$, as $\numS \geq 2q^{-2}$, we remark that, $ (1-q)^\numS \leq (1-q)^{q^{-2}}$, therefore
 it is sufficient to show that 
\begin{equation*}
(1-e^{-1}) \geq (1-q)^{q^{-2}}.
\end{equation*}
This holds because $q \mapsto(1-q)^{q^{-2}}$ is non-decreasing 
on $[0,\frac{1}{2}]$ and the inequality is true at $q = \frac{1}{2}$.

\vspace{2mm}
\noindent We similarly show that,
\begin{equation*}
\sup_{\alpha \leq q} \phi^{\numS}_{\numS}(\alpha) \leq \sup_{\alpha \geq q} \phi^{\numS}_{\numS}(\alpha).
\end{equation*}
To do so, we first remark that, $B_{\numS,\numS}(\alpha) = \alpha^\numS$ and, for any $\alpha \leq q$,
$\phi^{\numS}_{\numS}(\alpha) \leq \frac{q^\numS}{1-q}$
and,
\begin{equation*}
\sup_{\alpha \geq q} \phi^{\numS}_{\numS}(\alpha) \geq \frac{1-e^{-1}}{1-q}.
\end{equation*}
We then conclude in a similar way.
\end{proof}

\begin{lemma}
\label{lem:monotonic_phi}
For any $\mu \leq 1-q$ and any $r \in \{1,\ldots,\numS -1 \} $,
\begin{equation*}
 \ratio_{\numS} \left(\OS{r}, \mathcal{B}(\mu)) \right) \leq  \ratio_{\numS} \left(\OS{r+1}, \mathcal{B}(\mu)) \right) .
\end{equation*}
Inequality is strict for $\mu \not \in \{  0,1-q \}$. Furthermore,
\begin{equation*}
 \sup_{\mu \in[0,1-q]} \ratio_{\numS} \left(\OS{r}, \mathcal{B}(\mu)) \right) <   \sup_{\mu \in[0,1-q]} \ratio_{\numS} \left(\OS{r+1}, \mathcal{B}(\mu)) \right) .
\end{equation*}
 Similarly, for any $\mu \geq 1-q$ and any $r \in \{1,\ldots,\numS -1 \} $,
\begin{equation*}
 \ratio_{\numS} \left(\OS{r}, \mathcal{B}(\mu)) \right) \geq  \ratio_{\numS} \left(\OS{r+1}, \mathcal{B}(\mu)) \right) .
\end{equation*}
Inequality is strict for $\mu \not \in \{  1-q,1 \}$. Furthermore
\begin{equation*}
 \sup_{\mu \in[1-q,1]}  \ratio_{\numS} \left(\OS{r}, \mathcal{B}(\mu)) \right) >  \sup_{\mu \in[1-q,1]}  \ratio_{\numS} \left(\OS{r+1}, \mathcal{B}(\mu)) \right) .
\end{equation*}
\end{lemma}

\begin{proof}[\textbf{Proof of \Cref{lem:monotonic_phi}}]
For each $r \in \{1, \ldots, \numS\}$ we define $\phi^{\numS}_{r}$ such that for every $\alpha \in [0,1]$,
\begin{equation*}
\phi_{r}^{\numS}(\alpha) = \ratio_n \left( \OS{r}, \mathcal{B}(1-\alpha) \right). 
\end{equation*}

Remark that for all $r \in \{1, \ldots, \numS-1\}$ and $\alpha \leq q$ we obtain from \Cref{lem:OS_vs_bern} that,
\begin{align*}
\phi^{\numS}_{r}(\alpha) &=  \frac{(q - \alpha) B_{r,\numS}(\alpha) }{(1-q)\alpha} \\
&= \frac{(q - \alpha) B_{r+1,\numS}(\alpha) }{(1-q)\alpha} + \frac{(q-\alpha)b_{r+1,\numS}(\alpha)}{(1-q)\alpha}\\
&= \phi^{\numS}_{r+1}(\alpha) + \frac{(q-\alpha)b_{r+1,\numS}(\alpha)}{(1-q)\alpha} \geq \phi^{\numS}_{r+1}(\alpha).
\end{align*}
Note that, $b_{r+1,\numS}(\alpha) > 0$ for $\alpha \in (0,1)$ thus the above inequality is an equality only for $\alpha \in \{0,q\}$. Moreover, let $\alpha^* \in \argmax_{\alpha \in [0,q]} \phi^{\numS}_{r}(\alpha)$ (which exists by continuity of $\phi^{\numS}_r$). Remark that $\alpha^* \in (0,q)$ because, $\phi^{\numS}_r(0) = \phi^{\numS}_r(q) = 0$ and $\phi^{\numS}_r(q/2) >0$.  Therefore, we have
\begin{equation*}
\sup_{\alpha \in [0,q]} \phi^{\numS}_{r}(\alpha) = \phi^{\numS}_{r}(\alpha^*) > \phi^{\numS}_{r+1}(\alpha^*) \geq \sup_{\alpha \in [0,q]} \phi^{\numS}_{r+1}(\alpha).
\end{equation*}

Similarly, we have that for $\alpha \geq q$,
\begin{align*}
\phi^{\numS}_{r}(\alpha) &=   \frac{(\alpha - q) \left(1 -  B_{r,\numS}(\alpha) \right) }{q (1 -\alpha)} \\
&=  \frac{(\alpha - q) \left(1 -  B_{r+1,\numS}(\alpha) \right) }{q (1 -\alpha)} - \frac{(\alpha - q) b_{r+1,\numS}(\alpha) }{q (1- \alpha)}\\
&= \phi^{\numS}_{r+1}(\alpha) - \frac{(\alpha - q) b_{r+1,\numS}(\alpha) }{q (1- \alpha)} \leq \phi^{\numS}_{r+1}(\alpha).
\end{align*}
We conclude by an argument similar to the one derived in the case where $\alpha \in [0,q]$ that equality holds only for $\alpha \in \{q, 1\}$ and that, $\sup_{\alpha \in [q,1]} \phi^{\numS}_{r}(\alpha) < \sup_{\alpha \in [q,1]} \phi^{\numS}_{r+1}(\alpha).$
\end{proof} 

\begin{lemma}
\label{lem:levi_bound}
Let $U(\numS) = \int_0^\infty 2 \exp \left (-\frac{\numS \epsilon^2}{18 + 8 \epsilon} \min(q,1-q) \right) d\epsilon$. Then,
$ U(\numS) = \mathcal{O} \left( \frac{1}{\sqrt{\numS}} \right)$.
\end{lemma}
\begin{proof}[\textbf{Proof of \Cref{lem:levi_bound}}]
Remark that
\begin{align*}
 U(\numS) &= \int_0^\infty 2 \exp \left(-\frac{\numS \epsilon^2}{18 + 8 \epsilon} \min(q,1-q) \right) d\epsilon\\
 &\leq 2 \left( \int_0^{\frac{18}{8}} \exp \left (-\frac{\numS \epsilon^2}{18} \min(q,1-q) \right) d \epsilon +     \int_{\frac{18}{8}}^{\infty}   \exp \left (-\frac{\numS \epsilon}{8 } \min(q,1-q)  \right) \right) d \epsilon \\
 &= 2  \int_0^{\frac{18}{8}} \exp \left (-\frac{\numS \epsilon^2}{18} \min(q,1-q) \right) d \epsilon + \mathcal{O} \left( \frac{1}{\numS} \right)\\
 &\leq \frac{2}{\sqrt{\numS}} + 2\int_{\frac{1}{\sqrt{\numS}}}^{\frac{18}{8}} \frac{\epsilon}{\epsilon} \exp \left (-\frac{\numS \epsilon^2}{18} \min(q,1-q) \right) d \epsilon + \mathcal{O} \left( \frac{1}{\numS} \right)\\
 &\leq \frac{2}{\sqrt{\numS}} + 2\sqrt{\numS} \int_{\frac{1}{\sqrt{\numS}}}^{\frac{18}{8}} \epsilon \exp \left (-\frac{\numS \epsilon^2}{18} \min(q,1-q) \right) d \epsilon + \mathcal{O} \left( \frac{1}{\numS} \right)\\
&= \frac{1}{\sqrt{\numS}} \left(2 + \frac{18}{\min\left(q,1-q\right)} \exp \left (-\frac{\min(q,1-q)}{18}  \right)  \right) + o \left( \frac{1}{\sqrt{\numS}} \right).\end{align*}
\end{proof}

\section{Discussion: Algorithms and Performance}

\subsection{Suboptimality of SAA}\label{app:subopt_SAA}

\Cref{prop:necessary} provides a necessary condition for a policy to be optimal. We present in \Cref{fig:SAA_variation} a counter-example showing that SAA is not always achieving this necessary condition. \Cref{fig:SAA_variation} presents the performance of SAA against Bernoulli distributions with different means with a value of $q =.9$ and $\numS= 20$. 
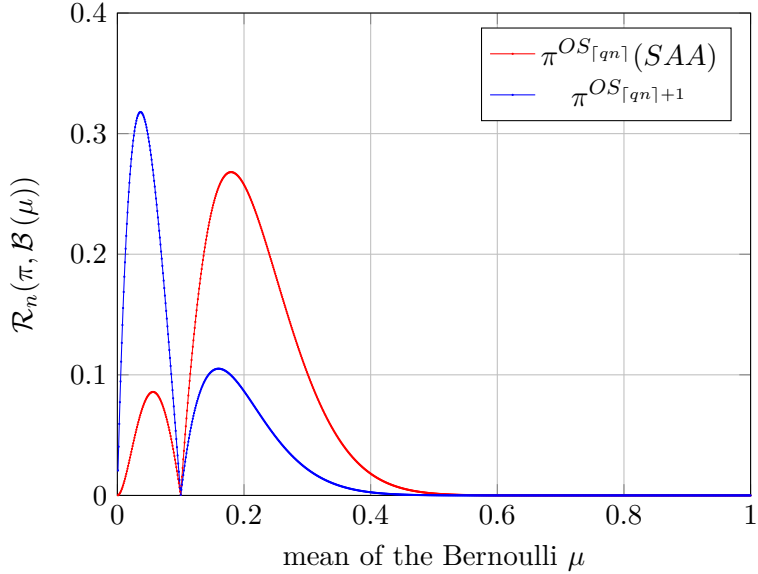
\begin{figure}[h!]
\begin{center}
\begin{tikzpicture}
\begin{axis}[
            title={},
            xmin=0,xmax=1,
            ymin=0.0,ymax=.4,
            width=10cm,
            height=8cm,
            table/col sep=comma,
            xlabel = mean of the Bernoulli $\mu$,
            ylabel = ${\ratio_{\numS}  (\pi,\mathcal{B} \left(\mu\right) )}$,
            grid=both,
            skip coords between index={0}{1},
            legend pos=north east]      
\addplot [red,mark=*,mark options={scale=.1}] table[x=mus,y=SAA] {Data/intuition_subopt_SAA_n20_0.9.csv};
    \addlegendentry{$\OS{\ceil{q \numS}} (SAA)$ }
\addplot [blue,mark=*,mark options={scale=.1}] table[x=mus,y=SAA+1] {Data/intuition_subopt_SAA_n20_0.9.csv};
    \addlegendentry{$ \OS{\ceil{q \numS} + 1}$}
\end{axis}
\end{tikzpicture}
\end{center}
\caption{\textbf{ Performance of the $\mathbf{\ceil{qn}^{th}}$ (SAA) and $\mathbf{(\ceil{qn}+1)^{th}}$  order statistic policies against Bernoulli distributions}. The figure depicts the performance of two order statistic policies against Bernoulli distribution as the mean $\mu$ varies ($q=.9$, $\numS=20$).  }\label{fig:SAA_variation}
\end{figure}

We observe  in that case that
\begin{equation*}
\sup_{\mu \in [0,1-q]} \ratio_{\numS} \left(\SAA,\mathcal{B} \left( \mu \right) \right) < \sup_{\mu \in [1-q,1]} \ratio_{\numS} \left( \SAA,\mathcal{B} \left( \mu \right) \right).
\end{equation*}
This implies the suboptimality of SAA according to \Cref{prop:necessary}. We note from \Cref{fig:SAA_variation} that, in this example, SAA suffers from a larger regret in the mode associated with large values of $\mu$ compared to the mode associated with smaller values of $\mu$. In contrast, the $(\ceil{qn}+1)^{th}$  order statistic policy suffers from a larger regret than SAA  on the mode associated with the small values of $\mu$ and from a smaller one in the regime where $\mu$ is large. This observation implies that a carefully chosen randomization of both policies would perform strictly better than SAA.

\subsection{Insight on the Minimax Optimal Policy}\label{apx:values_k}

The algorithm derived in \Cref{thm:optimal_min_max} is defined as a randomization of the $(k-1)^{th}$ and $k^{th}$ order statistics for some $k \in \{2,\ldots,\numS\}$. However, \Cref{thm:optimal_min_max} does not provide any quantification of the value of $k$. We present in \Cref{tab:values_k} the values of $k$ for different critical quantiles obtained by computing the minimax optimal policy for sample size smaller than $200$.
\begin{table}[h!]
\centering
\begin{tabular}{ccc}
 &$k = \ceil{q\numS}$ & $k = \ceil{q\numS} + 1$\\
 \hline
$q=.7$& $40.5\%$ & $59.5\%$ \\
\hline
$q=.8$& $41\%$ & $59\%$ \\
\hline
$q=.9$& $42.5\%$ & $57.5\%$ 
\end{tabular}
\caption{\textbf{Parameter of the minimax optimal policy}. The table presents the proportion of time the parameter $k$ (defined in \Cref{thm:optimal_min_max}) is respectively equal to $\ceil{q\numS}$ and $\ceil{q\numS} + 1$ for different values of $q$. This proportion is derived by computing the parameters of the optimal policy for any data size smaller than $200$.}
\label{tab:values_k}
\end{table}
Recall that SAA uses the $\ceil{q\numS}^{th}$ order statistic.  Therefore, \Cref{tab:values_k} shows that for the first $200$ samples, the minimax optimal policy always has in its support SAA and a neighboring order statistic. The relation between  $k$ and $\ceil{q\numS}$ is further discussed in the proof of \Cref{thm:asymptotic}. We show in the proof that $k$ has to scale as $\ceil{q \numS} + o(\sqrt{\numS})$.

\subsection{Algorithmic Implementation of the Optimal Policy}\label{apx:implement_alg}

\Cref{thm:optimal_min_max} presents the structure of the optimal data-driven policy. We next detail how to find the optimal tuning parameters $k$ and $\gamma$  for an optimal policy. To that end, we establish an additional structural result on single order statistic policies. We show that for any $r \in \{1,\ldots, \numS - 1 \}$,
\bear
\label{eq:monotonic_phi_1}
\ratio_{\numS} \left(\OS{r}, \mathcal{B}(\mu)) \right) &\leq&  \ratio_{\numS} \left(\OS{r+1}, \mathcal{B}(\mu)) \right)  \qquad \mbox{for all } \mu \leq 1-q\\
\label{eq:monotonic_phi_2}
 \ratio_{\numS} \left(\OS{r}, \mathcal{B}(\mu)) \right) &\geq&  \ratio_{\numS} \left(\OS{r+1}, \mathcal{B}(\mu)) \right)
 \qquad \mbox{for all } \mu \geq 1-q.
\eear 
This is formally stated by \Cref{lem:monotonic_phi} presented in \Cref{app:add_aux}.

Equations \eqref{eq:monotonic_phi_1} and \eqref{eq:monotonic_phi_2} formalize the fact that, against Bernoulli distributions, the performance of smaller order statistics is worse than larger ones when the mean is large, as they tend to underestimate the optimal inventory quantity. On the contrary, they perform better than larger order statistics when the mean is smaller than $1-q$ since underestimating is valuable in that case.

Given equations \eqref{eq:monotonic_phi_1} and \eqref{eq:monotonic_phi_2} , we now present an efficient algorithm to compute the parameters of the optimal policy. \Cref{alg:optimal} only needs to perform $\mathcal{O} \left( \log ( \numS) \right)$ line searches in order to find an order statistic $k$ such that \eqref{eq:crossing_1} and \eqref{eq:crossing_2} are satisfied.

{\setstretch{1}
\medskip
\begin{algorithm}[H]
\hrulefill\\
 \KwData{critical fractile $q$, number of samples $\numS$}
 \KwResult{Order statistic ranking $k$, weight $\gamma$ and optimal value $\ratio_{\numS}^{*}$}
 \eIf{\eqref{eq:extreme_1} and \eqref{eq:extreme_n} hold}{
 Set $j= 1$ and $k= \numS$\;
 \While{j < k}{
  $m = (j+k)/2$\;
  \eIf{ $\sup_{\mu \in [1-q,1]} \ratio_n\left(\OS{m}, \mathcal{B}(\mu)) \right) - \sup_{\mu \in [0,1-q]} \ratio_n\left(\OS{m}, \mathcal{B}(\mu)) \right) \geq 0$}{
 $j= m+1$\;
   }{
 $k= m$\;
 }
}
Find the solution $\gamma$ of the following equation by performing a line search to solve $\sup_{\mu \in[1-q,1]} \gamma \ratio_n\left(\OS{k}, \mathcal{B}(\mu)) \right) + ( 1 - \gamma) \ratio_n\left(\OS{k-1}, \mathcal{B}(\mu)) \right) = \sup_{\mu \in [0,1 - q]} \gamma \ratio_n\left(\OS{k}, \mathcal{B}(\mu)) \right)  + ( 1 - \gamma)\ratio_n\left(\OS{k-1}, \mathcal{B}(\mu)) \right) $\;
}
{
$\gamma = 1$\;
If \eqref{eq:extreme_1} does not hold, $k=1$, whereas if \eqref{eq:extreme_n} does not hold, $k=\numS$\;
}
 Set $\ratio_{\numS}^{*} =\sup_{\mu \in[1-q,1]} \gamma \ratio_n\left(\OS{k}, \mathcal{B}(\mu)) \right) + ( 1 - \gamma ) \ratio_n\left(\OS{k-1}, \mathcal{B}(\mu)) \right)$\;
 \Return $k$, $\gamma$, $\ratio_{\numS}^{*}$\;
\hrulefill\\
\caption{Optimal data-driven policy}
\label{alg:optimal}
\end{algorithm}
\medskip
}

\end{document}